\documentclass[a4paper,11pt, reqno]{amsart}
\usepackage{amsfonts}
\usepackage{amsmath, array, bigints}
\usepackage{amssymb}
\usepackage{mathrsfs}
\usepackage[colorlinks]{hyperref}
 \usepackage{graphicx}
\usepackage{enumerate}
\usepackage[usenames]{color}
\makeatletter
\newcommand{\setword}[2]{%
  \phantomsection
  #1\def\@currentlabel{\unexpanded{#1}}\label{#2}%
}
\makeatother
\newtheorem{thm}{Theorem}[section]
\newtheorem{cor}[thm]{Corollary}
\newtheorem{lem}[thm]{Lemma}
\newtheorem{prop}[thm]{Proposition}

\newtheorem{remark}{Remark}
\numberwithin{equation}{section}
\usepackage[english]{babel} 
\usepackage{blindtext}

\theoremstyle{definition}
\newtheorem{definition}[thm]{Definition}

\begin{document}

\allowdisplaybreaks 
 \title[Brezis-Nirenberg type problems driven by superposition operators ]{Brezis-Nirenberg type problems associated with nonlinear superposition operators of mixed fractional order}

 \author[Yergen Aikyn, Sekhar Ghosh, Vishvesh Kumar, and Michael Ruzhansky]{Yergen Aikyn, Sekhar Ghosh, Vishvesh Kumar, and Michael Ruzhansky}

\address[Yergen Aikyn]{Department of Mathematics: Analysis, Logic and Discrete Mathematics, Ghent University, Ghent, Belgium}
\email{aikynyergen@gmail.com}

\address[ Sekhar Ghosh]{Department of Mathematics, National Institute of Technology Calicut, Kozhikode, Kerala, India - 673601}
\email{sekharghosh1234@gmail.com / sekharghosh@nitc.ac.in}
\address[Vishvesh Kumar]{Department of Mathematics: Analysis, Logic and Discrete Mathematics, Ghent University, Ghent, Belgium}
\email{vishveshmishra@gmail.com / vishvesh.kumar@ugent.be}

\address[Michael Ruzhansky]{Department of Mathematics: Analysis, Logic and Discrete Mathematics, Ghent University, Ghent, Belgium\newline and \newline
School of Mathematical Sciences, Queen Marry University of London, United Kingdom}
\email{michael.ruzhansky@ugent.be}
\date{}

\begin{abstract}
This paper aims to study the Brezis-Nirenberg type problem  driven by the nonlinear superposition of operators of the form $$A_{\mu, p}u:=\int_{[0,1]}(-\Delta)_{p}^{s} u\,\, d \mu(s),$$
where $\mu$ denotes the signed measure over $[0, 1]$.

We consider nonlinear nonlocal equations associated with $A_{\mu, p}$, involving critical nonlinearity and lower-order perturbation. Using variational techniques, we establish existence results for the critical problem by employing weak lower semicontinuity arguments under general assumptions on the perturbation term. We discuss the multiplicity results when the perturbation term vanishes at the origin. Additionally, when the lower-order term is a pure power function, we examine the Brezis–Nirenberg–type problem using the mountain pass technique. Furthermore, we address the existence of solutions to subcritical problems associated with $A_{\mu, p}.$

Our findings are novel, even in the case of the sum of two distinct fractional $p$-Laplacians or a combination of a fractional $p$-Laplacian with a classical $p$-Laplacian. More generally, our framework is sufficiently broad to accommodate finite sums of different fractional $p$-Laplacians as well as cases involving fractional Laplacians with ``wrong" signs.

A key contribution of this study is the development of a unified approach that systematically addresses these problems by incorporating a broad class of operators and lower-order perturbation terms within a common theoretical framework. The results remain new even in the case of linear superposition of fractional operators of different orders.

\end{abstract}


\maketitle

\tableofcontents
\section{Introduction and main results}
\subsection{Motivation and state of the art}
The main purpose of this work is to study critical Brezis-Nirenberg problems involving a nonlocal operator obtained through the superposition of nonlinear fractional operators of different orders, which was introduced in a recent paper by Dipierro et al. \cite{DPSV}. The results presented in this paper are highly general in nature, yet they are novel, even in specific cases that can be derived as consequences of our comprehensive approach. 
\\
Brezis and Nirenberg \cite{BN83} studied the following nonlinear critical equation:
\begin{equation} \label{EucBNL}
\begin{cases}
&-\Delta u = |u|^{2^*-2} u+\lambda u \quad \text{in} \,\, \Omega,\\&
u=0 \quad \text{on} \,\, \partial \Omega,
\end{cases}
\end{equation}
for $N\geq 3$ and for the Sobolev critical exponent $2^* := \frac{2N}{N-2},$ where $\Omega$ is a smooth bounded domain of $\mathbb{R}^N$ and $\lambda$ is a real parameter. They established that if $N \geq 4$, then problem \eqref{EucBNL} admits a positive solution for $\lambda \in (0, \lambda_1)$, where $\lambda_1$ denotes the first eigenvalue of $(-\Delta)$ in $H^1_0(\Omega)$.
In the three-dimensional Euclidean space, they inferred the existence of a constant $\lambda_* \in (0, \lambda_1)$ such that for any $\lambda \in (\lambda_*, \lambda_1)$, problem \eqref{EucBNL} possesses a positive solution. Furthermore, they demonstrated that equation \eqref{EucBNL} has a positive solution if and only if $\lambda \in (\lambda_1/4, \lambda_1)$ when $\Omega$ is a ball.
In contrast, for $\lambda \notin (0, \lambda_1)$, it can be observed that using the Poho$\check{\text{z}}$aev identity, problem \eqref{EucBNL} does not have a positive solution. Capozzi et al. \cite{CFP85} established that for $N \geq 4$, equation \eqref{EucBNL} admits a nontrivial solution for every parameter $\lambda$.
    
The existence and non-existence of solutions for semi-linear equations are strongly influenced by their nonlinear terms. When these equations involve the critical Sobolev exponent, the standard variational method fails due to a lack of compactness. This issue arises in numerous variational problems in PDEs and (sub)-Riemannian geometry, which initially motivated the study of the Brezis-Nirenberg problem \eqref{EucBNL}. A prominent example directly linked to the Brezis-Nirenberg problem is the celebrated {\it (CR)-Yamabe problem} on (sub)-Riemannian manifolds. Specifically:

{\it Given an $N$-dimensional compact Riemannian manifold $(M, g)$, where $N \geq 3$, with scalar curvature $k := k(x)$, the Yamabe problem seeks to find a metric $\tilde{g}$ conformal to $g$ with constant scalar curvature $\tilde{k}$.}

In fact, by setting $\tilde{g} = u^{\frac{4}{N-2}} g$, where $u > 0$ is the conformal factor, the Yamabe problem can be reformulated as the following nonlinear critical equation involving the Laplace-Beltrami operator $\Delta_M$ on $(M, g)$:
\begin{equation}
-\frac{N-1}{N-2} \Delta_M u = \tilde{k} u^{2^*-1} - k(x) u.
\end{equation}

Moreover, the Brezis-Nirenberg problem \eqref{EucBNL} is connected to the existence of extremal functions for functional inequalities, as well as to the existence of non-minimal solutions for the Yang-Mills functions and $H$-systems (see \cite{BN83}). For motivation and more information on the sub-Riemannian version of the Brezis-Nirenberg problem, we refer to some recent works of the authors \cite{GKR22,  GKR24, GK25} and references therein. 

The nonlocal analogue of the Brezis-Nirenberg equation \eqref{EucBNL} was intensively investigated by Servadei \cite{SER13, SER14}, and Valdinoci \cite{SERV05, SERV13}. Specifically, they considered the following nonlocal problem:
\begin{equation} \label{EucBNNL}
\begin{cases}
(-\Delta)^s u = |u|^{2_s^*-2} u+\lambda u \quad &\text{in} \,\, \Omega,\\
u=0 \quad &\text{in} \,\, \mathbb{R}^N \backslash \Omega,
\end{cases}
\end{equation}
where  $s \in (0, 1)$ is fixed, $2_s^*:=\frac{2N}{N-2s},$ $N>2s,$ and  $(-\Delta)^s,$ is the well-known fractional Laplacian defined by:
\begin{equation*}
    (-\Delta)^su(x):= C_{N,s} \, P.V. \int_{\mathbb{R}^N} \frac{u(x)-u(y)}{|x-y|^{N+2s}} \,dy,\quad x\in \mathbb{R}^N.
\end{equation*} 
They completely extended the classical results of Brezis and Nirenberg \cite{BN83}, Capozzi et al. \cite{CFP85}, Zhang \cite{Zhang}, and Gazzola and Ruf \cite{GR97}.  One can summarize their finding as follows: The problem \eqref{EucBNNL} has a nontrivial weak solution in some appropriate Sobolev space $H^s(\mathbb{R}^N),$ whenever 
\begin{itemize}
    \item $N> 4s$ and $\lambda >0;$
    \item  $N=4s$ and $\lambda>0$ is different from the eigenvalue of $(-\Delta)^s;$
    \item $2s<N<4s$ and $\lambda$ is sufficiently large.
\end{itemize} 

This problem was later studied for the nonlinear nonlocal operator, namely, the fractional $p$-Laplacian, $1<p<\infty,$ defined as  
\begin{equation*}
    (-\Delta_p)^su(x):= 2 C_{N,p,s} \lim _{\varepsilon \backslash 0} \int_{\mathbb{R}^{N} \backslash B_{\varepsilon}(x)} \frac{|u(x)-u(y)|^{p-2}(u(x)-u(y))}{|x-y|^{N+s p}} d y,\quad x\in \mathbb{R}^N,
\end{equation*}
where 
\begin{equation}\label{def c}
C_{N,p,s}:=\frac{\frac{sp}{2}(1-s)2^{2s-1}}{\pi^{\frac{N-1}{2}}}\frac{\Gamma(\frac{N+ps}{2})}{\Gamma(\frac{p+1}{2})\Gamma(2-s)}    
\end{equation}
is the normalizing constant \cite{DGV21}. Owing to the nonlinear nature of the operator, additional technical difficulties naturally appear in the fractional (semilinear) case. This was successfully addressed by Mosconi et al. \cite{MPSY16}. To describe their result, let us denote the sequence of eigenvalues 
$$0<\lambda_{1,p}< \lambda_{2, p} \leq \ldots \leq \lambda_{k, p} \leq \ldots,$$
of the nonlocal eigenvalue problem 
\begin{equation}
    \begin{cases}
        (-\Delta_p)^s u = \lambda |u|^{p-2}u & \,\,\text{in}\,\,\,\, \Omega, \\
        u =0 & \text{in}\,\, \mathbb{R}^N \backslash \Omega,
    \end{cases}
\end{equation}
defined by using the $\mathbb{Z}_2$-cohomological index of Fadell and Robinowitz \cite{FR78}. They \cite{MPSY16} proved that the nonlocal nonlinear critical equation
\begin{equation}
    \begin{cases}
        (-\Delta_p)^s u = \lambda |u|^{p-2}u+|u|^{p_s^*-2} u & \text{in}\,\,\, \Omega, \\
        u =0 & \text{in}\,\, \mathbb{R}^N \backslash \Omega,
    \end{cases}
\end{equation} where $p_s^*:=\frac{pN}{N-ps}, \, N>ps,$ has a nontrivial weak solution in a suitable Sobolev space $X^p_s(\Omega)$ if the following conditions hold true:
\begin{itemize}
    \item  $N=sp^2$ and $\lambda <\lambda_{1, p};$
    \item $N>sp^2$ and $\lambda$ is not one of the eigenvalues $\lambda_{j, p}, j=1,2, \ldots;$
    \item  $\frac{N^2}{(N+s)} >sp^2;$
    \item  $\frac{N^3+s^3p^3}{N(N+s)}>sp^2$ and $\partial \Omega \in C^{1,1}.$
\end{itemize}

On the other hand, Mawhin and Molica Bisci \cite{MB: 2017} used a different approach to study the existence of solutions for critical elliptic equations other than the usually exploited methods such as the concentration compactness principle of Lions \cite{Lions85, Lions85-2}. They investigated the equation  \begin{equation} \label{pro1intro}
      \begin{cases} 
(-\Delta_{p})^s u= \mu |u|^{p_s^*-2}u+\lambda h(x, u) \quad &\text{in}\quad \Omega, \\
u=0\quad & \text{in}\quad \mathbb{R}^N\backslash \Omega,
\end{cases}
  \end{equation}
where   $\mu > 0, \lambda>0$ are real parameters and $h$ is a subcritical nonlinearity satisfying \begin{equation} \label{growthintro}
        |h(x, t)| \leq a_1+a_2|t|^{q-1}\quad \text{for}\,\,\text{a.e.}\,\,x\in \Omega,\,\, \text{and for all } t \in \mathbb{R},
    \end{equation}
    for some $a_1, a_2>0$ and $q \in [1, p_s^*)$. They proved that problem \eqref{pro1intro} has at least one weak solution for a suitable value of $\lambda>0$. The key ingredient in their proof was the weak lower semicontinuity results. This technique has already been used by Faraci and Farkas \cite{FF: 2015} and Squassina \cite{S: 2004} to study quasilinear $p$–Laplacian equations involving critical nonlinearities.

Recently, there has been growing interest in nonlinear problems driven by operators of mixed types, particularly in connection with the study of optimal animal foraging strategies (see, for example, \cite{DV21}).  From a mathematical perspective, this operator presents significant challenges owing to the interplay between nonlocal effects and lack of scaling invariance. There have been several fundamental developments in many directions, such as \cite{BDVV22, DFV24} and the references therein for more details.

In  \cite{BDVV22}, the authors initiated the Brezis-Nirenberg type problem for mixed local and nonlocal operators, $\mathcal{L}:=-\Delta +(-\Delta)^s.$ In fact, they considered the critical problem of the form
\begin{equation} \label{ln2}
    \begin{cases}
        \mathcal{L}u=u^{2^*-1}+\lambda u^p & \text{in}\,\,\, \Omega, \\
        u \gneqq  0 & \text{in}\,\, \Omega,\\
        u \equiv 0 & \text{in}\,\, \mathbb{R}^N \backslash \Omega,
    \end{cases}
\end{equation}
where $p \in [1, 2^*-1),$ $\lambda \in \mathbb{R}$ and $\Omega$ is an open and bounded set. Their first result states that if $\Omega$ is bounded and star-shaped, and $\lambda \leq 0,$ then \eqref{ln2} admits no solutions, irrespective of any value of $p \in [1, 2^*-1).$ For $p=1,$ the problem \eqref{ln2} possesses at least one solution if $\lambda \in (\lambda^*, \lambda_1),$ with $\lambda^* \in [\lambda_{1, s}, \lambda_1),$ where $\lambda_{1, s}$ and $\lambda_1$ are the first Dirichlet eigenvalues for the operators $(-\Delta)^s$ and $\mathcal{L} $, respectively. Moreover, there are no solutions to the problem \eqref{ln2} if $\lambda \geq \lambda_1$ and for every $0<\lambda \leq \lambda_{1, s},$ there are no solutions to \eqref{ln2} belonging to the suitable closed ball $\mathcal{B} \subset L^{2^*}(\mathbb{R}^N),$ see \cite[Theorem 1.4]{BDVV22}. However, for the superlinear perturbation, that is, $p \in (1, 2^*-1)$ and $N \geq 3,$ let us first set $\kappa_{s, n}:=\min\{2-2s, N-2\}$ and $\beta_{p, N}:=N- \frac{(p+1)(N-2)}{2}.$ Then it was proved \cite[Theorem 1.5]{BDVV22} that if $\kappa_{s, N}>\beta_{p, N}$ then the problem \eqref{ln2} has a solution for every $\lambda>0,$ and if $\kappa_{s, N} \leq \beta_{p, N}$ then problem \eqref{ln2} has a solution for $\lambda$ large enough. 
The problem \eqref{ln2} was further extended in \cite{DFV24} for the operator $\mathcal{L}_p= -\Delta_p+(-\Delta_p)^s$ and the existence and multiplicity of solutions were studied by combining the variational methods with some topological techniques, such as the Krasnoselskii genus and the Lusternik–Schnirelman category theory.

\subsection{Main results and discussion}  
To explain our main result, let us  consider the signed measure 
\begin{equation} \label{mudef}
    \mu:=\mu^{+}-\mu^{-}, 
\end{equation}
where the components $\mu^{+}$ and $\mu^{-}$ are two  (nonnegative) finite Borel measures on $[0,1]$ satisfying the following conditions: 
\begin{equation}\label{measure 1}
    \mu^{+}([\bar{s}, 1])>0,
\end{equation}
\begin{equation}\label{measure 2}
    \mu^{-}|_{[\bar{s}, 1]}=0,
\end{equation}
and
\begin{equation}\label{measure 3}
\mu^{-}([0, \bar{s}]) \leq \kappa \mu^{+}([\bar{s}, 1]) ,
\end{equation}
for some $\bar{s} \in(0,1]$ and $\kappa \geq 0.$
Then the main object of interest is the following nonlinear superposition fractional operator
\begin{equation}\label{superposition operator}
    A_{\mu, p} u:=\int_{[0,1]}(-\Delta)_{p}^{s} u d \mu(s).
\end{equation} 

Notably, the signed measure $\mu$ allows each operator $(-\Delta)_p^s$ to contribute to the definition of $A_{\mu, p}$  \eqref{superposition operator} with potentially different signs. However, in our setting,  assumptions \eqref{measure 1}–\eqref{measure 3} enable us to partition the interval $[0, 1]$ into two subintervals, where the right subinterval has a specific sign and dominates the left subinterval sufficiently.  Note that, since $\mu$ is a signed measure, each individual fractional $p$-Laplacian in \eqref{superposition operator} may contribute with potentially different signs. However, the conditions \eqref{measure 1}–\eqref{measure 3} ensure that the negative components of the signed measure $\mu$ are effectively ``absorbed" by the positive components. Consequently, the negative contributions in \eqref{superposition operator} do not influence the higher fractional exponent values.

The notation $(-\Delta)_p^s$ is conventionally assigned to the fractional $p$-Laplacian, defined for all $s \in (0,1)$ by
$$
(-\Delta)_{p}^{s} u(x):=2 C_{N,p,s} \lim _{\varepsilon \backslash 0} \int_{\mathbb{R}^{N} \backslash B_{\varepsilon}(x)} \frac{|u(x)-u(y)|^{p-2}(u(x)-u(y))}{|x-y|^{N+s p}} d y.
$$
The positive normalizing constant $C_{N,p,s}$ is chosen in such a way that, it provides consistent limits as $s \nearrow 1$ and as $s \searrow 0$, namely
\begin{align*}
& \lim _{s \searrow 0}(-\Delta)_{p}^{s} u=(-\Delta)_{p}^{0} u:=u, \\
& \lim _{s \nearrow 1}(-\Delta)_{p}^{s} u=(-\Delta)_{p}^{1} u:=-\Delta_{p} u=-\operatorname{div}\left(|\nabla u|^{p-2} \nabla u\right).
\end{align*}

The operator $A_{\mu, p}$ was introduced in \cite{DPSV}, which extended the construction of the linear case (i.e., $p=2$) discussed in \cite{DPSV2}. Particular cases for the operator in \eqref{superposition operator} are (minus) the $p$-Laplacian (corresponding to the choice of $\mu$ being the Dirac measure concentrated at 1), the fractional $p$-Laplacian $(-\Delta)_p^{s}$ (corresponding to the choice of $\mu$ being the Dirac measure concentrated at some fractional power $s \in (0, 1)$), the ``mixed order operator" $-\Delta_p+(-\Delta)_p^{s}$ (when $\mu$ is the sum of two Dirac measures $\delta_1+\delta_s,\,s \in (0, 1)$), etc.  {\it An intriguing aspect of the operators studied in this work is their ability to simultaneously handle nonlinear operators and an infinite (potentially uncountable) family of fractional operators. Moreover, some of these operators may carry the ``wrong sign", provided that a dominant part exists, determined by operators of higher fractional order.} At the end of this section, we present several concrete examples and discuss our main results for these specific settings in Section \ref{secexpa}.

Beyond their theoretical importance, these operators naturally arise in various applications, particularly in mathematical biology, where they model population dispersal under diverse diffusion strategies, such as Gaussian and L\'evy flights (see \cite{DV21, DPSV25}). In this context, the ability to handle complex operators with components exhibiting the ``wrong sign" is particularly intriguing for applications in biology and population dynamics. For instance, according to the Lévy flight foraging hypothesis, animal dispersal is often best represented as a superposition of (potentially fractional) operators of different orders, capturing the diverse foraging strategies adopted by individuals within a population (e.g., \cite{DV21, DPSV25}). In this framework, the inclusion of operators with the ``wrong sign" offers a natural means to model individuals who, instead of dispersing in search of food, exhibit a tendency to cluster for social interactions or mating purposes, resulting in patterns governed by a retrograde fractional heat equation.


Observe that, by assumption \eqref{measure 1}, there exists another fractional exponent $s_{\sharp} \in[\bar{s}, 1]$ such that 
\begin{equation}\label{measure 4}
\mu^{+}\left(\left[s_{\sharp}, 1\right]\right)>0. 
\end{equation}  

Subsequently, we will see that $s_\sharp$ also serves as a critical exponent. Notably, there is some flexibility in choosing $s_\sharp$ above. However, the results obtained will be stronger if $s_\sharp$ is chosen to be ``as large as possible" while still satisfying \eqref{measure 4}. In fact, we can simply take $s_\sharp:=\bar{s},$ but selecting a larger $s_\sharp,$ if feasible leads to both qualitative and quantitative improvements in the results.

In this paper, we considered two types of nonlinear elliptic problems with a critical exponent related to the operator $A_{\mu, p}$. First, let $\Omega$ be a bounded open subset of $\mathbb{R}^{N}$. We consider the following Brezis-Nirenberg problem:

\begin{equation}\label{problem 0.1}
 \left\{\begin{array}{cc}
    A_{\mu, p} u  =\lambda g(x,u) +\gamma |u|^{p_{s_\sharp}^{*}-2} u \text { in } \Omega,  \\
    u  = 0  \text { in } \mathbb{R}^{N} \backslash \Omega ,
\end{array}\right.
\end{equation}
where $\gamma$ and $\lambda$ are real positive parameters and $g: \Omega \times \mathbb{R} \rightarrow \mathbb{R}$ is a Carath\'eodory function, which means, $g$ is Lebesgue measurable in the first variable and continuous in the second variable,  verifying the following standard subcritical growth condition 
\begin{align}\label{cond on g 0}
\begin{split}
&\text{ there exist } a_{1}, a_{2}>0 \text{ and } q \in[1, p_{s_{\sharp}}^{*}), \text{ such that }\\
&|g(x, t)| \leq a_{1}+a_{2}|t|^{q-1} \text {, a.e. } x \in \Omega, \forall t \in \mathbb{R}. 
\end{split}
\end{align}
The exponent $p_{s_{\sharp}}^{*}$  defined as
\begin{equation}\label{critical exponent}
    p_{s_{\sharp}}^{*}:=\frac{N p}{N-s_{\sharp} p}
\end{equation}
is the fractional critical exponent related to the fractional exponent $s_{\sharp}$ for which \eqref{measure 4} holds true. 

To state our main results and study the problem, we introduce the fractional Sobolev space $X_{p}(\Omega)$ defined as the set of measurable functions $u: \mathbb{R}^{N} \rightarrow \mathbb{R}$ such that $u=0$ in $\mathbb{R}^{N} \backslash \Omega$ for which the following norm is finite 
\begin{equation}\label{norm on Xpinto}
  \|u\|_{X_p(\Omega)}:=\left(\int_{[0,1]}[u]_{s, p}^{p} d \mu^{+}(s)\right)^{1 / p}< +\infty, 
\end{equation}
where, for $s \in [0, 1],$ the Gagliardo norm $[u]_{s, p}$ is defined as
\begin{equation*}
    [u]_{s, p}:= \begin{cases}\|u\|_{L^{p}\left(\mathbb{R}^{N}\right)} & \text { if } s=0, \\ \left(C_{N, s, p} \iint_{\mathbb{R}^{2 N}} \frac{|u(x)-u(y)|^{p}}{|x-y|^{N+s p}} d x d y\right)^{1 / p} & \text { if } s \in(0,1), \\ \|\nabla u\|_{L^{p}\left(\mathbb{R}^{N}\right)} & \text { if } s=1.\end{cases}
\end{equation*}
We refer to Section \ref{pre} for more details and properties of these Sobolev spaces.

The first main result related to the problem \eqref{problem 0.1} is as follows.

\begin{thm}\label{main result wlsc 0}
    Let $ \Omega$ be a bounded subset of $\mathbb{R}^N$. Let $\mu=\mu^{+}-\mu^{-}$with $\mu^{+}$ and let $\mu^{-}$ satisfying \eqref{measure 1}-\eqref{measure 3} and let $s_{\sharp}$ be as in \eqref{measure 4}. Suppose $\frac{N}{s_{\sharp}}>p \geq 2$ and $q \in[1, p_{s_\sharp}^*)$, where $p_{s_\sharp}^*=\frac{p N}{N-{s_\sharp}p}$. Let $g: \Omega \times \mathbb{R} \rightarrow \mathbb{R}$ be a Carath\'eodory function that satisfies the subcritical growth condition \eqref{cond on g 0}. Then, for every $\gamma>0$, there exists $\Lambda_\gamma>0$ such that the problem \eqref{problem 0.1} admits at least one weak solution in the space $X_{p}(\Omega)$ for every $0<\lambda<\Lambda_\gamma$ which is a local minimum of the energy functional
$$
\mathcal{J}_{\gamma, \lambda}(u):=\frac{1}{p}\|u\|_{X_p(\Omega)}^{p}-\frac{1}{p} \int_{[0, \bar{s}]}[u]_{s, p}^{p} d \mu^{-}(s)-\lambda \int_{\Omega} \int_{0}^{u(x)} g(x, \tau) d \tau dx-\frac{\gamma}{p_{s_{\sharp}}^{*}} \int_{\Omega}|u|^{p_{s_\sharp}^{*}} d x
$$
for every $u \in X_{p}(\Omega)$.
Moreover, we provide explicit information on the interval $\left(0, \Lambda_\gamma\right)$ and show that 
\begin{equation*}
    \left(0, \Lambda_\gamma\right) \subset\left(0, \max _{r \geq 0} \frac{r^{p-1}-\gamma C_{p_{s_\sharp}^{*}}^{p_{s_\sharp}^{*}} r^{p_{s_\sharp}^{*}-1}}{a_{1} C_{p_{s_\sharp}^{*}}|\Omega|^{(p_{s_\sharp}^{*}-1) / p_{s_\sharp}^{*}}+a_{2} C_{p_{s_\sharp}^{*}}^{q}|\Omega|^{(p_{s_\sharp}^{*}-q) / p_{s_\sharp}^{*}} r^{q-1}}\right),
\end{equation*}
where $C_{p_{s_\sharp}^*}$ is the best constant in the continuous Sobolev embedding given by \eqref{Sobolev constant 2} and $|\Omega|$ is the Lebesgue measure of the open bounded subset $\Omega$ of $\mathbb{R}^N$.
\end{thm}

The novelty of our results lies in the generality of the assumptions on the lower-order term $g$ as well as the broader framework involving the superposition of operators. This approach not only unifies various existing results but also yields several new findings, even in the specific case of the operator $A_{\mu, p}$ and the lower-order term $g$. By relaxing restrictive conditions and adopting a more general framework, we developed a unified approach for applying variational techniques to a broader class of nonlocal problems.

It is a well-known fact that due to the presence of critical exponent the energy functional 
$$
\mathcal{J}_{\gamma, \lambda}(u):=\frac{1}{p}\|u\|_{X_p(\Omega)}^{p}-\frac{1}{p} \int_{[0, \bar{s}]}[u]_{s, p}^{p} d \mu^{-}(s)-\lambda \int_{\Omega} \int_{0}^{u(x)} g(x, \tau) d \tau dx-\frac{\gamma}{p_{s_{\sharp}}^{*}} \int_{\Omega}|u|^{p_{s_\sharp}^{*}} d x,
$$
associated with problem \eqref{problem 0.1} does not satisfy the Palais-Smale condition. The main reason for this is the lack of compactness due to the fact that the embedding $X_p(\Omega) \hookrightarrow L^{p_{s_\sharp}^*}(\Omega)$ is not compact. There are many other ways to avoid this difficulty, including using a weak lower semicontinuity. 
The proof of Theorem \ref{main result wlsc 0} is based on a weak lower semicontinuity result, which is stated in Proposition \ref{main prop on wlsc}. The idea is to show that, for every $\gamma>0$, the restriction of the functional
$$
u \mapsto \frac{1}{p} \int_{[0,1]}[u]_{s, p}^{p} d \mu^{+}(s)-\frac{1}{p} \int_{[0, \bar{s}]}[u]_{s, p}^{p} d \mu^{-}(s)-\frac{\gamma}{p_{s_\sharp}^{*}} \int_{\Omega}|u|^{p_{s_\sharp}^{*}} d x
$$
to a sufficiently small ball in the solution space $X_p(\Omega)$ is sequentially weakly lower semicontinuous. Then, the energy functional associated with equation \eqref{problem 0.1} is locally sequentially weakly lower semicontinuous. Consequently, the application of the direct minimization technique implies that for any $\gamma > 0$ and sufficiently small $\lambda$, the energy functional $\mathcal{J}_{\gamma, \lambda}$ admits a critical point (local minimum), which corresponds to a weak solution of problem \eqref{problem 0.1}.  

Let us also discuss some recent results involving the operator $A_{\mu, p}.$ In \cite{DPSV2}, the authors introduced the basic framework to study the variational problem of the operator $A_{\mu, 2}$. Subsequently, they studied the existence of a nontrivial solution of the critical growth equation with jumping nonlinearity. Dipierro el al. \cite{DPSV} consider the following nonlinear problem involving operator $A_{\mu, p}:$ 
\begin{equation} 
       \begin{cases} 
A_{\mu, p} u=   |u|^{p_{s_\sharp}^*-2}u+\lambda |u|^{p-2} u  \quad &\text{in}\quad \Omega, \\
u=0\quad & \text{in}\quad \mathbb{R}^N\backslash \Omega,
\end{cases}
\end{equation}
and proved the existence of $2m$-distinct $(m \geq 1)$ of solutions, provided that $\mu$ satisfies the conditions \eqref{measure 1}–\eqref{measure 3} and $\kappa$ is sufficiently small. They proved this using an abstract critical point theory based on $\mathbb{Z}_2$-cohomological index. They \cite{DPSV1} further generalized this framework by considering the superposition of both $s$ and $p.$ 

It is also worth noting that  the proof of Theorem \ref{main result wlsc 0} remains valid even when considering the nonlocal problem with a lower-order term $g: \Omega \times \mathbb{R} \rightarrow \mathbb{R}$ such that 
      \begin{equation}
          |g(x, t)| \leq a_1+a_2 |t|^{m-1}+a_3 |t|^{q-1}\quad \text{a. e.}\,\,\,x \in \Omega,\, \forall t \in \mathbb{R},
      \end{equation}
      for some positive constants $a_j, j=1,2,3$ and $1<m<p \leq q <p_{s_\sharp}^*$. This will be discussed in Theorem \ref{thm1.1dupl}.

We emphasize that in our analysis, the presence of possible contributions from the negative component $\mu^-$ of the measure $\mu$ in \eqref{mudef} makes the study of problem \eqref{problem 0.1} particularly challenging. Nevertheless, we overcome these difficulties and prove that problem \eqref{problem 0.1} admits a solution. We also point out that unlike the classical (quasilinear) case \cite{FF: 2015} and \cite{S: 2004}, where lower semicontinuity was established using a truncation argument, this approach is no longer applicable in our setting due to the nonlocal nature of problem \eqref{problem 0.1}. Therefore, we adopt a different method proposed in \cite{MB: 2017} instead of the truncation argument; however, it also fails to accommodate the case $p \in (1, 2). $ That is the reason we assume $p \geq 2$ in Theorem \ref{main result wlsc 0}.
\\

Although we have established the existence of a solution to problem \eqref{problem 0.1}, it is often desirable to seek nonnegative, nontrivial solutions. Indeed, if $g(x, 0)=0,$ then $0$ itself is a solution to \eqref{problem 0.1}. Therefore, it is crucial to show that $0$ is not a local minimum of functional $\mathcal{J}_{\gamma, \lambda}.$ This situation, where $g(x, 0)=0$, arises in the following results, making the existence of a nontrivial solution essential. With this objective in mind, in the next result, we analyze a specific case of problem \eqref{problem 0.1}, given by

\begin{equation} 
       \begin{cases} 
A_{\mu, p} u=  |u|^{p_{s_\sharp}^*-2}u+\lambda (|u|^{m-1}+|u|^{q-1}) \quad &\text{in}\quad \Omega, \\
u=0\quad & \text{in}\quad \mathbb{R}^N\backslash \Omega,
\end{cases}
    \end{equation}
    where $1 \leq m<p\leq q<p_{s_\sharp}^*=\frac{p N}{N-{s_\sharp}p}.$

In this scenario, we will deduce the existence of a nonnegative nontrivial weak solution. For this purpose we need to assume an additional condition on $\mu^-.$ For that let us assume that bounded domain $\Omega \in \mathbb{R}^n$ is contained in ball $B_R$ of radius $R.$ We choose $\bar{\kappa}:=\bar{\kappa}(N, R, \bar{s})$ small enough and we suppose that there exists $\delta \in (0, 1-\bar{s}]$ such that 
\begin{align} \label{extracondi}
    \mu^-([0, \bar{s})) \leq \bar{\kappa} \delta \mu^+([\bar{s}, 1-\delta]).
\end{align}

It is evident that condition \eqref{extracondi} implies condition \eqref{measure 3} when $\bar{\kappa} \delta \leq \kappa$. However, the converse does not hold. This is primarily because, in \eqref{measure 3}, the contribution on the right-hand side arises solely from the ``local component", as in the case where $\mu^+$ is a singular measure concentrated at $s=1$. On the other hand, condition \eqref{extracondi} requires that the contribution from the negative part $\mu^-$ of the measure $\mu$ on $[0,s^-)$ be properly compensated by the contribution from the ``purely nonlocal component"  of the positive part.
  In particular, we will use this property to prove Lemma \ref{energycomparison}. The same condition was employed by Dipierro et al. \cite{DPSV25}, specifically for the case $p=2$. Specifically, they used this condition to show the positivity of the first eigenfunction of the operator $A_{\mu, p}.$   Ideally, it would be desirable to establish the result without this additional condition; however, for now, we leave this as an open question. 

\begin{thm} \label{thmnonmain}
    Let $ \Omega$ be a bounded subset of $\mathbb{R}^N$. Let $\mu=\mu^{+}-\mu^{-}$ with $\mu^{+}$and $\mu^{-}$ satisfying \eqref{measure 1}-\eqref{measure 3} and \eqref{extracondi}, and let $s_{\sharp}$ be as in \eqref{measure 4}. Let $\frac{N}{s_{\sharp}}>p \geq 2$. Suppose that $m$ and $q$ are two real constants such that 
$$1 \leq m<p\leq q<p_{s_\sharp}^*=\frac{p N}{N-{s_\sharp}p}.$$
Then, there exists an open interval $\Lambda \subset (0, +\infty)$ such that, for every $\lambda \in \Lambda,$ the  nonlocal problem 
\begin{equation} \label{pro15.7intro}
       \begin{cases} 
A_{\mu, p} u=  |u|^{p_{s_\sharp}^*-2}u+\lambda (|u|^{m-1}+|u|^{q-1}) \quad &\text{in}\quad \Omega, \\
u=0\quad & \text{in}\quad \mathbb{R}^N\backslash \Omega,
\end{cases}
    \end{equation} has at least one nonnegative nontrivial weak solution $u_\lambda \in X^{p}(\Omega),$ which is a  local minimum of the energy functional 
     \begin{equation}
        \mathcal{J}_{ \lambda}(u):= \frac{1}{p}\|u\|_{X_p(\Omega)}^{p}-\frac{1}{p} \int_{[0, \bar{s}]}[u]_{s, p}^{p} d \mu^{-}(s) -\frac{1}{p_{s_\sharp}^*} \int_{\Omega} |u(x)|^{p_{s_\sharp}^*} d x-\frac{\lambda}{m} \int_{\Omega} |u(x)|^m  d x-\frac{\lambda}{q} \int_{\Omega} |u(x)|^q  d x.
    \end{equation}  
\end{thm}

In Theorem \ref{thmnonmain}, the interval $\Lambda$ can be explicitly localised and given by  
\begin{equation}
  \Lambda \subset \left( 0,\,  \max_{r>0} \frac{r^{p-1}-\gamma C_{p_{s_\sharp}^{*}}^{p_{s_\sharp}^{*}} r^{p_{s_\sharp}^{*}-1} }{ C_{p_{s_\sharp}^{*}} |\Omega|^{\frac{p_{s_\sharp}^*-1}{p_{s_\sharp}^*}}+C_{p^*_{s_\sharp}}^m|\Omega|^{\frac{p_{s_\sharp}^*-m}{p_{s_\sharp}^*}} r^{m-1}+C_{p^*_{s_\sharp}}^q|\Omega|^{\frac{p_{s_\sharp}^*-q}{p_{s_\sharp}^*}} r^{q-1}} \right).
\end{equation}
\\
The proof of Theorem \ref{thmnonmain} relies on several auxiliary results. The first is Theorem \ref{subthm}, which determines the sign of the solution under an additional condition on $g$. The second is Theorem \ref{thm1.1dupl}, which extends Theorem \ref{main result wlsc 0} to the case in which $g$ is allowed to include a lower-order term.
\\

We now turn our attention to investigating the following parameter-dependent subcritical Brezis-Nirenberg type problem (that is, assuming $\gamma = 0$ in \eqref{problem 0.1}) associated with the operator $A_{\mu, p}$
\begin{equation}\label{problem 0.1.2}
\begin{cases}A_{\mu, p}  u=\lambda g(x,u) & \text { in } \quad \Omega, \\ 
u=0 & \text { in } \mathbb{R}^{N} \backslash \Omega,\end{cases}
\end{equation}
where $\lambda>0$ is a real parameter and $g$ is a subcritical nonlinearity satisfying \eqref{cond on g 0}. The main result related to problem \eqref{problem 0.1.2} is stated in the following theorem.

\begin{thm}\label{main result subcrit 0}
     Let $\Omega$ be a bounded open subset of $\mathbb{R}^{N}$. Let $\mu=\mu^{+}-\mu^{-}$with $\mu^{+}$and $\mu^{-}$satisfying \eqref{measure 1}-\eqref{measure 3} and let $s_{\sharp}$ be as in \eqref{measure 4}. Suppose that $1<p<q<p_{s_{\sharp}}^{*}$ and $s_{\sharp}p<N$, where $p_{s_{\sharp}}^{*}:=\frac{N p}{N-s_{\sharp}p}$. Let $g: \Omega \times \mathbb{R} \rightarrow \mathbb{R}$ be a Carath\'eodory function satisfying the subcritical growth condition \eqref{cond on g 0}. Furthermore, let
$$
0<\lambda<\frac{(q-p)^{\frac{q-p}{q-1}}(p-1)^{\frac{p-1}{q-1}}}{(a_{1} C_1)^{\frac{q-p}{q-1}}(a_{2} C_{2})^{\frac{p-1}{q-1}}(q-1) C_{p_{s_{\sharp}}^{*}}^{p}|\Omega|^{\frac{p_{s_{\sharp}}^{*}-q}{p_{s_{\sharp}}^{*}}\left(\frac{p-1}{q-1}\right)}|\Omega|^{\frac{p_{s_{\sharp}}^{*}-1}{p_{s_{\sharp}}^{*}}\left(\frac{q-p}{q-1}\right)}},
$$
where $C_{1}, C_{2}$ and $C_{p_{s_{\sharp}}^{*}}$ are the embedding constants of $L^{p_{s_{\sharp}}^{*}}(\Omega) \hookrightarrow L^{1}(\Omega), L^{p_{s_{\sharp}}^{*}}(\Omega) \hookrightarrow$ $L^{q}(\Omega)$, and $X_{p}(\Omega) \hookrightarrow L^{p_{s_{\sharp}}^{*}}(\Omega)$, respectively, and $|\Omega|$ is the Lebesgue measure of set $\Omega$. Then, the nonlocal elliptic problem \eqref{problem 0.1.2} admits a weak solution $u_{0, \lambda} \in X_{p}(\Omega)$ with 
$$
\|u_{0, \lambda}\|_{X_p(\Omega)}<\left(\frac{\lambda(q-1) a_{2} C_{2} C_{p_{s_{\sharp}}^{*}}^{q}|\Omega|^{({p_{s_{\sharp}}^{*}-q})/{p_{s_{\sharp}}^{*}}}}{p-1}\right)^{\frac{1}{p-q}}.
$$ 
\end{thm} 

The proof of Theorem \ref{main result subcrit 0} also follows the direct method of calculus of variations, similar to the approach used for Theorem \ref{main result wlsc 0}. The proof is relatively simple because no critical nonlinearity term is involved.

The following interesting result is a special case $(\lambda=1)$ of Theorem \ref{main result subcrit 0}. Specifically, it generalizes the classical local problem for $A_{\mu, p}=-\Delta$ studied in \cite{AC04} by utilizing the variational principle established by Ricceri \cite{Ric00}.

\begin{cor} 
     Let $\Omega$ be a bounded open subset of $\mathbb{R}^{N}$. Let $\mu=\mu^{+}-\mu^{-}$with $\mu^{+}$and $\mu^{-}$satisfying \eqref{measure 1}-\eqref{measure 3} and $s_{\sharp}$ be as in \eqref{measure 4}. Suppose that $1<p<q<p_{s_{\sharp}}^{*}$ and $s_{\sharp}p<N$, where $p_{s_{\sharp}}^{*}:=\frac{N p}{N-s_{\sharp}p}$. Let $g: \Omega \times \mathbb{R} \rightarrow \mathbb{R}$ be a Carath\'eodory function satisfying the subcritical growth condition
    \begin{equation} 
        |g(x, t)| \leq a_1+a_2|t|^{q-1}\quad \text{for}\,\,\text{a.e.}\,\,x\in \Omega,\,\, \forall t \in \mathbb{R}
    \end{equation}
    for some $a_1, a_2>0.$  Assume that $$ a_1^{\frac{q-p}{p-1}} a_2 < \frac{(q-p)^{\frac{q-p}{p-1}}(p-1)}{  C_1^{\frac{q-p}{p-1}} C_2 (q-1)^{\frac{q-1}{p-1}} C_{p_{s_{\sharp}}^{*}}^{p \left(\frac{q-1}{p-1} \right)} |\Omega|^{\frac{p_{s_\sharp}^*-q}{p_{s_\sharp}^*}} |\Omega|^{\frac{p_s^*-1}{p_{s_\sharp}^*} \left( \frac{q-p}{p-1} \right)} },$$
    where $C_1,$ $C_2$ and  $C_{p_{s_{\sharp}}^{*}}$ is the embedding constant of $L^{p_{s_\sharp}^*}(\Omega) \hookrightarrow L^1(\Omega),$ $L^{p_{s_\sharp}^*}(\Omega) \hookrightarrow L^q(\Omega),$ and  $X_p(\Omega) \hookrightarrow L^{p_{s_\sharp}^*}(\Omega),$ respectively,  and $|\Omega|$ is the Lebesgue measure of the set $\Omega.$ 
     Then, the following  nonlocal subcritical problem
    \begin{equation} 
       \begin{cases} 
A_{\mu, p} u=  g(x, u) \quad &\text{in}\quad \Omega, \\
u=0\quad & \text{in}\quad \mathbb{R}^N\backslash \Omega,
\end{cases}
    \end{equation} admits a weak solution $u_{0} \in X_p(\Omega)$ and 
    $$\|u_{0}\|_{X_p(\Omega)}< \left( \frac{(q-1) a_2 C_2 C_{p_{s_{\sharp}}^{*}}^q |\Omega|^{\frac{p_{s_\sharp}^*-q}{p_{s_\sharp}^*}}}{p-1} \right)^{\frac{1}{p-q}}.$$ 
\end{cor}

Clearly, if $g$ in \eqref{cond on g 0} is a pure power term or satisfies the Ambrosetti-Rabinowitz-type condition, then the standard machinery of the Mountain Pass Theorem can be applied to establish the existence of a nontrivial solution to problem \eqref{problem 0.1}. In such cases, the associated energy functional satisfies the Palais-Smale (PS) condition, allowing variational methods to be used to obtain critical points. That is the purpose of our next result where  we consider the Brezis-Nirenberg-type problem
\begin{equation}\label{problem 0.2}
 \left\{\begin{array}{cc}
    A_{\mu, p} u  =\lambda |u|^{q-2}u +|u|^{p_{s_\sharp}^{*}-2} u \text { in } \Omega,  \\
    u  = 0  \text { in } \mathbb{R}^{N} \backslash \Omega,
\end{array}\right.
\end{equation}
where $\lambda$ is a real positive parameter and $1<p<q<p_{s_\sharp}^{*}$. The exponent $p_{s_{\sharp}}^{*}$ is the fractional critical exponent defined in \eqref{critical exponent}.

The main result for the problem \eqref{problem 0.2} is stated in the following theorem.

\begin{thm}\label{main result 0.2}
    Let $ \Omega$ be a bounded subset of $\mathbb{R}^N$. Let $\mu=\mu^{+}-\mu^{-}$with $\mu^{+}$and $\mu^{-}$satisfying \eqref{measure 1}-\eqref{measure 3} and $s_{\sharp}$ be as in \eqref{measure 4}. Assume $1<p<q<p_{s_{\sharp}}^*$ and $s_{\sharp}p<N$, where $p_{s_{\sharp}}^{*}:=\frac{N p}{N-s_{\sharp}p}$. Then, there exists $\lambda^*>0$ such that problem \eqref{problem 0.2} admits at least one nontrivial solution for all $\lambda \geq \lambda^*$, provided that $\kappa \in [0, \kappa_0]$ for a sufficiently small $\kappa_0$ depending on $N$, $\Omega,$ $p$, $\mu,$ $\lambda$ and $s_\sharp$.
\end{thm}

As mentioned earlier, the proof of Theorem \ref{main result 0.2} relies on the Mountain-Pass Lemma of Ambrosetti and Rabinowitz \cite{AR: 1973}. Specifically, we first demonstrate that the energy functional associated with problem \eqref{problem 0.2} satisfies the Mountain-Pass geometry and that the corresponding minimax levels fulfil an appropriate asymptotic property. The key step is to establish the Palais-Smale condition for the energy functional. As stated in Lemma \ref{PS condition lem}, it turns out that, similar to the classical case, the Palais-Smale condition holds only within a specific range of minimax levels, bounded by an appropriate constant. This limitation arises due to the lack of compactness in the Sobolev embedding, which is discussed in detail in Section \ref{pre}.

We emphasize that our main results (Theorem \ref{main result wlsc 0}, Theorem \ref{thmnonmain}, Theorem \ref{main result subcrit 0} and Theorem \ref{main result 0.2}) are not only new in their wide generality given the structure of the superposition operator $A_{\mu, p}$, but they also possess many
specific cases which are also new. Let us discuss some interesting specific cases.

\begin{itemize}
    \item If $\mu=\delta_1,$ then the operator $A_{\mu, p}$ is reduced to the usual $p$-Laplacian. For this situation, Theorem \ref{main result wlsc 0} (see also Theorem \ref{thm1.1dupl}) and Theorem \ref{thmnonmain} were considered in \cite{FF: 2015}. We deal with this case in Subsection \ref{subsec2.1}.
\\
    \item By taking $\mu=\delta_s$ for $s \in (0, 1),$ the operator $A_{\mu, p}$ turn out to be the fractional $p$-Laplacian $(-\Delta_p)^s.$ In this case,  Theorem \ref{main result wlsc 0} and Theorem \ref{thmnonmain} recover the main results of \cite{MB: 2017}. This will be presented in Subsection \ref{subsec2.2}.
\\
    \item  If $\mu=\delta_1+\delta_s,$ for some $s \in (0, 1)$ then $A_{\mu, p}$ boils down to the mixed local and nonlocal operator $-\Delta_p+(-\Delta_p)^s.$  Theorem \ref{main result wlsc 0}, Theorem \ref{thmnonmain}, Theorem \ref{main result subcrit 0} and Theorem \ref{main result 0.2} are new in this situation even for $p=2.$ For more details we refer to Subsection \ref{subsec2.3}.
    \\
    \item  The operator corresponding to the choice of measure $\mu:=\delta_{s_1}+\delta_{s_2}$ for some $s_1, s_2 \in (0, 1)$ is the superposition of two nonlocal operators with different order $(-\Delta_p)^{s_1}+(-\Delta_p)^{s_2}.$ Our results are also novel in this situation. This is a particular case of a more general situation of a convergent series of Dirac measures considered below. 
    \\
    \item Now, we consider the two cases of convergent series of Dirac measures. Let $1 \geq s_0>s_1>s_2> \ldots \geq 0.$ Consider the operator corresponding to the measure
    $$\mu:= \sum_{k=0}^{+\infty} a_k \delta_{s_k} \quad \text{with} \quad \sum_{k=0}^{+\infty} a_k \in (0, +\infty).$$
    provided that 
    \begin{itemize}
        \item[(a)] either $a_0>0$ and $a_k \geq 0$ for all $k \geq 1;$
    \item[(b)] or there exist $\kappa \geq 0$ and $\bar{k} \in \mathbb{N}$ such that 
    \begin{equation*} 
       \sum_{k=\bar{k}+1}^{+\infty} a_k \leq \kappa \sum_{k=0}^{\bar{k}} a_k,
    \end{equation*}
    and $a_k>0$   for all $k \in \{0, 1, \ldots, \bar{k}\}.$
     
    \end{itemize}
    In this case, the corresponding operator is $A_{p, \mu}:= \sum_{k=0}^{+\infty} a_k (-\Delta_p)^{s_k}.$ 
    Our main results
    Theorem \ref{main result wlsc 0}, Theorem \ref{main result subcrit 0}, Theorem \ref{main result 0.2} and Theorem \ref{thmnonmain} are new in the literature. We also mention here that we need to consider an extra assumption to discuss  Theorem \ref{thmnonmain} in the specific case mentioned in item (b) in view of condition \eqref{extracondi}. We deal with these results in Subsection \ref{subsec2.4}.
    \\
    \item Let $s \in (0, 1)$ and $\mu:=\delta_1-\alpha \delta_s$ with $\alpha \geq 0$ sufficiently small. The corresponding operator is $-\Delta_p-\alpha(-\Delta_p)^s.$  This represents a case where the measure $\mu$ changes sign, and the second term of the operator has the ``wrong sign." This scenario is also novel in the context of our main results. It is worth noting that smallness of $\alpha$ is only required to discuss the specific case of Theorem \ref{thmnonmain}  in which we consider ``wrong sign" operator $ -\Delta_{p}+ (-\Delta_{p})^{s_1}-\alpha (-\Delta_{p})^{s_2}$
for some $1>s_1>s_2>0$ corresponding to the measure $\mu=\delta_1+\delta_{s_1}-\alpha\delta_{s_2}.$ 
This case is discussed in Subsection \ref{subsec2.5}.
\\
    \item  When the measure $\mu$ is absolutely continuous with respect to the Lebesgue measure $ds$ on $[0, 1],$ we can consider the continuous superposition of operator. Specifically, by choosing $d\mu:= f(s) ds$ we obtain  the continuous superposition of the fractional $p$-Laplacian, which has the following form:
    $$ \int_0^1 f(s) (-\Delta_p)^s u\, ds,$$
    where $f$ is a measurable and non-identically zero function.  Our results are also new in this setting.   We will treat this case in Subsection \ref{subsec2.6}. 
\end{itemize}

It goes without saying that particular cases are not limited to the aforementioned examples. Indeed, one can also consider more complex operators arising from the superposition of two different types of operators.
$$ \sum_{k=1}^M a_k (-\Delta_p)^{s_k} u(x)+ \int_0^1 f(s) (-\Delta_p)^s u(x)\,ds.$$

We conclude this section with an overview of the paper’s organization. In the next section, we introduce the fundamentals of fractional Sobolev spaces necessary for studying the nonlocal problem associated with the operator $A_{p, \mu}.$ Section \ref{sec3} presents and proves key technical results, including the lower semicontinuity result, which is crucial for further analysis. In Section \ref{sec4}, we provide the proof of Theorem \ref{main result wlsc 0}. Section \ref{sec5} explores the implications and applications of Theorem \ref{main result wlsc 0}, including Theorem \ref{thmnonmain} along with their proofs.
In Section \ref{sec6}, we consider the nonlocal subcritical problem associated with the nonlinear superposition operators $A_{p, \mu}$ of mixed order and establish the proof of Theorem \ref{main result subcrit 0}. Section \ref{sec7} is dedicated to the Brezis-Nirenberg-type problem, where we apply the Mountain Pass Theorem and present the proof of Theorem \ref{main result 2}. Finally, the last section discusses specific instances of the main results established in this paper.

\section{Preliminaries: Fractional Sobolev spaces and their embeddings} \label{pre}

The main objective of this section is to establish the functional analytic framework that focuses on the appropriate notions of fractional Sobolev spaces and their various properties, which are essential for studying our problem. For further details on this material, we refer to \cite{DPSV, DPSV1, DPSV2}. Additionally, we emphasize that some of the results discussed in this section are, to the best of our knowledge, new in the literature and we provide their proofs.

To begin with, for $s \in [0, 1]$ we define
\begin{equation*}
    [u]_{s, p}:= \begin{cases}\|u\|_{L^{p}\left(\mathbb{R}^{N}\right)} & \text { if } s=0, \\ \left(C_{N, s, p} \iint_{\mathbb{R}^{2 N}} \frac{|u(x)-u(y)|^{p}}{|x-y|^{N+s p}} d x d y\right)^{1 / p} & \text { if } s \in(0,1), \\ \|\nabla u\|_{L^{p}\left(\mathbb{R}^{N}\right)} & \text { if } s=1.\end{cases}
\end{equation*}
Here, $$C_{N,p,s}:=\frac{\frac{sp}{2}(1-s)2^{2s-1}}{\pi^{\frac{N-1}{2}}}\frac{\Gamma(\frac{N+ps}{2})}{\Gamma(\frac{p+1}{2})\Gamma(2-s)} $$ is the normalizing constant. 
Thanks to the normalizing constant $C_{N, s, p},$ we have that
$$
\lim _{s \searrow 0}[u]_{s, p}=[u]_{0, p} \quad \text { and } \quad \lim _{s \nearrow 1}[u]_{s, p}=[u]_{1, p}.
$$

We use the space $X_{p}(\Omega)$ defined as the set of measurable functions $u: \mathbb{R}^{N} \rightarrow \mathbb{R}$ such that $u=0$ in $\mathbb{R}^{N} \backslash \Omega$ for which the following norm is finite 
\begin{equation}\label{norm on Xp}
  \|u\|_{X_p(\Omega)} = \rho_{p}(u):=\left(\int_{[0,1]}[u]_{s, p}^{p} d \mu^{+}(s)\right)^{1 / p}< +\infty.
\end{equation}
This space was introduced in \cite{DPSV}. We recall the following results from \cite{DPSV}, which are essential for our study, without proof. 
\begin{lem}\cite[Lemma 4.2]{DPSV}\label{Uniform convexity}
    $X_{p}(\Omega)$ is a uniformly convex Banach space, for $1<p < \infty$.
\end{lem} 
\begin{lem}
      $X_{p}(\Omega)$ is a separable Banach space for $1 \leq p < \infty$.
\end{lem}
\begin{proof} To prove the separability of $X_p(\Omega),$ we define a map $T:X_p(\Omega) \rightarrow L^p(\Omega) \times L^p(\mathbb{R}^N \times \mathbb{R}^N) \times L^p(\Omega)$ by 
$$T(u):= (a_u, b_u, c_u),$$
where 
$a_u(x):= \mu^+(\{1\})\nabla u(x),$ $c_u(x):= \mu^+(\{0\}) u(x),$ and 
$$b_u(x, y):= u(x)-u(y) \left( \int_{(0, 1)} \frac{1}{|x-y|^{n+ps}} d\mu^+(s) \right)^{\frac{1}{p}}.$$

    Thus, for every $u \in X_p(\Omega),$ we have 
    $$\|Tu \|_{L^p(\Omega) \times L^p(\mathbb{R}^N \times \mathbb{R}^N) \times L^p(\Omega)} = \|u\|_{X_p(\Omega)}.$$
    Therefore, $T$ is an isometry that maps $X_p(\Omega)$ to a closed subspace of $L^p(\Omega) \times L^p(\mathbb{R}^N \times \mathbb{R}^N) \times L^p(\Omega).$ This immediately implies the separability of the space $X_p(\Omega).$
\end{proof}

In this setting, the most important fact is that in view of conditions \eqref{measure 1}–\eqref{measure 3}, the negative components of the signed measure $\mu$ can be ``reabsorbed" into the positive ones. 
\begin{lem}\label{reabsorb} \cite[Proposition 4.1]{DPSV}
    Let $p \in(1, N)$ and assume that \eqref{measure 2} and \eqref{measure 3} hold.
Then, there exists $c_{0}=c_{0}(N, \Omega, p)>0$ such that, for any $u \in X_{p}(\Omega)$, we have
\begin{equation*}
    \int_{[0, \bar{s}]}[u]_{s, p}^{p} d \mu^{-}(s) \leq c_{0} \kappa \int_{[\bar{s}, 1]}[u]_{s, p}^{p} d \mu(s)=c_{0} \kappa \int_{[\bar{s}, 1]}[u]_{s, p}^{p} d \mu^+(s).
\end{equation*}
\end{lem}

Furthermore, the following uniform Sobolev embedding result plays a crucial role in establishing embeddings for the space $X_p(\Omega)$.
\begin{lem}\label{Sobolev emb} \cite[Theorem 3.2]{DPSV}
    Let $\Omega$ be a bounded, open subset of $\mathbb{R}^{N}$ and $p \in(1, N)$.
Then, there exists $C=C(N, \Omega, p)>0$ such that, for every $s_1, s_2 \in[0,1]$ with $s_1 \leq s_2$ and every measurable function $u: \mathbb{R}^{N} \rightarrow \mathbb{R}$ with $u=0$ a.e. in $\mathbb{R}^{N} \backslash \Omega$, one has that
$$
[u]_{s_1, p} \leq C[u]_{s_2, p}.
$$
\end{lem}

Similar to \cite[Proposition 2.4]{DPSV2}, we can prove the following embedding results for the space $X_{p}(\Omega)$.

\begin{prop}\label{compact and cont embedding}
Assume that \eqref{measure 1}–\eqref{measure 3} hold. Let $s_{\sharp} \in[\bar{s}, 1]$ be as in \eqref{measure 4}.
Then, there exists a positive constant $\bar{c}=\bar{c}\left(N, \Omega, s_{\sharp}, p\right)$ such that, for any $u \in X_{p}(\Omega)$,
\begin{equation} \label{Xpusialemb}
[u]_{s_{\sharp},p} \leq \bar{c}\left(\int_{[0,1]}[u]_{s,p}^{p} \mathrm{~d} \mu^{+}(s)\right)^{\frac{1}{p}} .
\end{equation}
 
In particular, the embedding 
\begin{equation}\label{embedding}
    X_p(\Omega) \hookrightarrow L^{r}(\Omega)
\end{equation}
is continuous for any $r \in[1,p_{s_{\sharp}}^{*}]$ and compact for any $r \in[1,p_{s_{\sharp}}^{*})$.

\end{prop}

\begin{proof}
   By Lemma \ref{Sobolev emb}, used here with $s_{1}:=s_{\sharp}$ and $s_{2}:=s$, for all $s \in[s_{\sharp}, 1]$ we have that $[u]_{s_{\sharp}, p} \leq$ $c(N, \Omega, p)[u]_{s, p}$.

As a result,
$$
\mu^{+}\left(\left[s_{\sharp}, 1\right]\right)[u]_{s_{\sharp}, p}^{p} \leq c^{p}(N, \Omega, p) \int_{[s_{\sharp}, 1]}[u]_{s, p}^{p} \mathrm{~d} \mu^{+}(s) \leq c^{p}(N, \Omega, p) \int_{[0,1]}[u]_{s, p}^{p} \mathrm{~d} \mu^{+}(s) .
$$
This and \eqref{measure 4} yield the desired embedding \eqref{Xpusialemb}. Now, the continuous and compact embeddings results are a consequence of standard embedding theorems of fractional Sobolev spaces given in \cite{NPV12}.
\end{proof}

We now recall the following weak convergence results from \cite{DPSV}: This will be helpful to show the Palais-Smale condition for the functional $\mathcal{I}_{\lambda}.$

\begin{lem} \label{weak convergence}  \cite[ Lemma 5.8]{DPSV}
    Let $u_{n}$ be a bounded sequence in $X_{p}(\Omega)$.
Then, there exists $u \in X_{p}(\Omega)$ such that
\begin{gather}
\begin{split}
    \lim _{n \rightarrow+\infty} \int_{[0,1]}\left(\iint_{\mathbb{R}^{2 N}} \frac{\left|u_{n}(x)-u_{n}(y)\right|^{p-2}\left(u_{n}(x)-u_{n}(y)\right)(v(x)-v(y))}{|x-y|^{N+s p}} d x d y\right) d \mu^{ \pm}(s) \\
\quad=\int_{[0,1]}\left(\iint_{\mathbb{R}^{2 N}} \frac{|u(x)-u(y)|^{p-2}(u(x)-u(y))(v(x)-v(y))}{|x-y|^{N+s p}} d x d y\right) d \mu^{ \pm}(s)
\end{split}
\end{gather}
for any $v \in X_{p}(\Omega)$.
\end{lem}
The following result is a Brezis-Lieb-type lemma in our setting. 
\begin{lem}\label{B-L lemma} \cite[ Lemma 5.9]{DPSV}
   Let $u_{n}$ be a bounded sequence in $X_{p}(\Omega)$. Suppose that $u_{n}$ converges to some $u$ a.e. in $\mathbb{R}^{N}$ as $n \rightarrow+\infty$. Then we have
\begin{equation}
\int_{[0,1]}[u]_{s, p}^{p} d \mu^{ \pm}(s)=\lim _{n \rightarrow+\infty}\left(\int_{[0,1]}\left[u_{n}\right]_{s, p}^{p} d \mu^{ \pm}(s)-\int_{[0,1]}\left[u_{n}-u\right]_{s, p}^{p} d \mu^{ \pm}(s)\right).
\end{equation}
\end{lem}

The following lemma asserts that the normalizing constant $C_{N,p,s}$ as in \eqref{def c} is uniformly bounded for any $s\in(0,1)$.
\begin{lem}\label{lm 2.4}
    Let us consider
    \begin{equation}\label{2.10}
        \bar{\Gamma}_{N,p}:=\max_{s\in[0,1]}\frac{\Gamma(\frac{N+ps}{2})}{\Gamma(2-s)}~\text{and}~\underline{\Gamma}_{N,p}:=\min_{s\in[0,1]}\frac{\Gamma(\frac{N+ps}{2})}{\Gamma(2-s)}.
    \end{equation}
    Then, we have
    \begin{equation}\label{2.11}
        C_{N,p,s}\leq \bar{c}_{N,p}:=\frac{p\bar{\Gamma}_{N,p}}{\pi^{\frac{N-1}{2}}\Gamma(\frac{p+1}{2})}\in(0,\infty).
    \end{equation}
    In addition, let $\delta\in(0,\bar{s})$, where $\bar{s}$ is as in \eqref{measure 1}. If $s\in[\bar{s},1-\delta]$, then we have
     \begin{equation}\label{2.12}
        C_{N,p,s}\geq \underline{c}_{N,p}\delta,~\text{where}~\underline{c}_{N,p}:=\frac{p\bar{s}\delta\underline{\Gamma}_{N,p}}{4\pi^{\frac{N-1}{2}}\Gamma(\frac{p+1}{2})}\in(0,\infty).
    \end{equation}
\end{lem}
\begin{proof}
Since, Gamma function is continuous in $(0,\infty)$ and $\Gamma(t)>0$ for all $t>0$, we have
\begin{align*}
    0&<\min_{s\in[0,1]}\Gamma(2-s)\leq\max_{s\in[0,1]}\Gamma(2-s)<\infty~\text{and}\\
     0&<\min_{s\in[0,1]}\Gamma(\frac{N+ps}{2})\leq\max_{s\in[0,1]}\Gamma(\frac{N+ps}{2})<\infty
\end{align*}
Therefore, the proof of immediate from \eqref{2.11}, \eqref{2.12} and the definition of $C_{N, s, p}$.
\end{proof}
It is noteworthy to mention here that using  Lemma \ref{lm 2.4}, we can choose $\bar{\kappa}$ in \eqref{extracondi} as
\begin{equation}\label{2.13}
   \bar{\kappa}\in\left[0, \frac{\bar{s}\underline{\Gamma}_{N,p}}{4\bar{\Gamma}_{N,p}\max\{1, (2R)^p\}} \right). 
\end{equation}

\begin{lem}\label{lm 2.5}
    Let $\bar{s}$ be as in \eqref{measure 1} and $S\in[\bar{s},1]$. Then for any $\tau\in(0,2R)$, we have
    \begin{equation}\label{2.14}
        \min\{1,(2R)^{p\bar{s}}\}\mu^+([\bar{s},S))\int_{(0,\bar{s})}\frac{d\mu^-(s)}{\tau^{ps}} \leq  \max\{(2R)^{p\bar{s}}, (2R)^{p}\}\mu^-((0,\bar{s}))\int_{[\bar{s},S)}\frac{d\mu^+(s)}{\tau^{ps}}.
    \end{equation}
\end{lem}
\begin{proof}
Let $s\in(0,\bar{s})$ and $\theta\in[\bar{s},S)$. By using $\tau\in(0,2R)$, we get 
\begin{equation}\label{2.15}
    \left(\frac{2R}{\tau}\right)^{p\theta}>\left(\frac{2R}{\tau}\right)^{ps}.
\end{equation}
Therefore, integrating \eqref{2.15} with respect to $\mu^+$ in $\theta\in[\bar{s},S)$, we deduce that
\begin{equation}\label{2.16}
    \int_{[\bar{s},S)}\left(\frac{2R}{\tau}\right)^{p\theta} d\mu^+(\theta) \geq \mu^+([\bar{s},S))\left(\frac{2R}{\tau}\right)^{ps}.
\end{equation}
Again, integrating \eqref{2.16} with respect to $\mu^-$ in $s\in(0,\bar{s})$, we conclude that
\begin{equation}\label{2.17}
    \mu^-((0,\bar{s}))\int_{[\bar{s},S)}\left(\frac{2R}{\tau}\right)^{p\theta} d\mu^+(\theta)\geq \mu^+([\bar{s},S)) \int_{(0,\bar{s})}\left(\frac{2R}{\tau}\right)^{ps} d\mu^-(s).
\end{equation}
Now, using the monotonicity of the exponential function, we have
\begin{equation}\label{2.18}
    (2R)^{ps}\geq \min\{1,(2R)^{p\bar{s}}\}~\text{and}~ (2R)^{p\theta}\leq\max\{(2R)^{p\bar{s}}, (2R)^{p}\},
\end{equation}
for all $s\in(0,\bar{s})$ and $\theta\in[\bar{s},S)$. Therefore, using \eqref{2.17} and \eqref{2.18}, we obtain \eqref{2.14}. This completes the proof.
\end{proof}

The next lemma states that under the assumption \eqref{extracondi}, the Gagliardo-type energy of $|u|$ is less than that of $u$. The case $p=2$ was considered in \cite{DPSV25}. 
\begin{lem}\label{energycomparison}
Let $\mu$ satisfy \eqref{measure 1} and \eqref{measure 2} for some $\overline{s}\in (0,1)$. Let $R>0$ be such that $\Omega\subset B_R$ and let $\delta\in (0, 1-\overline{s}]$. Assume that \eqref{extracondi} holds. Then, for any $u\in X_p(\Omega)$, we have
\begin{equation}\label{eq2.15} 
\int_{[0,1]} [|u|]^p_s d\mu(s) \leq \int_{[0,1]} [u]^p_s d\mu(s).
\end{equation}
\end{lem}
\begin{proof}
Thanks to Lemma \ref{reabsorb}, we conclude that for any $u\in X_p(\Omega)$, we have 
$$\int_{[0,1]} [u]^p_s d\mu^+(s)<\infty~\text{and}~\int_{[0,1]} [u]^p_s d\mu^-(s)<\infty.$$
Recall the reverse triangle inequality $||u(x)|-|u(y)||\leq |u(x)-u(y)|.$ Thus, we get
$$\int_{[0,1]} [|u|]^p_s d\mu^+(s)<\infty~\text{and}~\int_{[0,1]} [|u|]^p_s d\mu^-(s)<\infty.$$
Therefore, using $\mu=\mu^+-\mu^-$ and \eqref{measure 2}, we conclude that \eqref{eq2.15} holds if
\begin{equation}\label{eq2.16}
\int_{[0, 1]} \big([|u|]^p_s - [u]^p_s\big) d\mu^+(s) - \int_{[0, \overline{s})} \big([|u|]^p_s - [u]^p_s\big) d\mu^-(s)\leq 0.
\end{equation}
Since, $$[|u|]_0 =\| |u|\|_{L^p(\Omega)}=
\| u\|_{L^p(\Omega)}=[u]_0~\text{and}~[|u|]_1 =\| |\nabla |u||\|_{L^p(\Omega)}=\| |\nabla u|\|_{L^p(\Omega)}=[u]_1,$$
the inequality \eqref{eq2.16} holds if
\begin{equation}\label{eq2.17}
\int_{(0,1)} \big([|u|]^p_s - [u]^p_s\big) d\mu^+(s) - \int_{(0, \overline{s})} \big([|u|]^p_s - [u]^p_s\big) d\mu^-(s)\leq 0.
\end{equation}
Now, using $||u(x)|-|u(y)||\leq |u(x)-u(y)|$, we obtain
\begin{align}\label{eq3.18}
    \int_{(0,1)} &\big([|u|]^p_s - [u]^p_s\big) d\mu^+(s) - \int_{(0, \overline{s})} \big([|u|]^p_s - [u]^p_s\big) d\mu^-(s)\nonumber\\
    =&\int_{(0, 1)}C_{N,p,s}\iint_{\mathbb{R}^{2N}}\frac{||u(x)|-|u(y)||^p-|u(x)-u(y)|^p}{|x-y|^{N+ps}}\,dx\, dy\, d\mu^+(s)\nonumber\\
&- \int_{(0, \overline s)}C_{N,p,s}\iint_{\mathbb{R}^{2N}}\frac{||u(x)|-|u(y)||^p-|u(x)-u(y)|^p}{|x-y|^{N+ps}}\,dx\, dy\, d\mu^-(s)\nonumber\\
&=\iint_{\mathbb{R}^{2N}}[|u(x)-u(y)|^p-||u(x)|-|u(y)||^p]\nonumber\\
&\quad\times\left[-\int_{(0, 1)}\frac{C_{N,p,s}}{|x-y|^{N+ps}}d\mu^+(s)+\int_{(0, \overline s)}\frac{C_{N,p,s}}{|x-y|^{N+ps}}d\mu^-(s)\right]dxdy.
\end{align}
We now claim that 
\begin{equation}\label{eq3.19}
\int_{(0, 1)}\frac{C_{N,p,s}}{|x-y|^{N+ps}}d\mu^+(s)>\int_{(0, \overline s)}\frac{C_{N,p,s}}{|x-y|^{N+ps}}d\mu^-(s).
\end{equation}
Indeed, from \eqref{2.11}, we have
\begin{equation}\label{eq3.20}
\int_{(0, \overline s)} \frac{C_{N,p,s}}{|x-y|^{N+ps}} d\mu^-(s) \leq \frac{\overline{c}_{N,p}}{|x-y|^N} \int_{(0, \overline s)}\frac{d\mu^-(s)}{|x-y|^{ps}}.
\end{equation}
On the other hand, from \eqref{extracondi}, we get $\mu^+\big([\overline s, 1-\delta]\big)>0$, otherwise, $\mu^-\big( (0,\overline s)\big)=0$. Therefore, from \eqref{2.12}, we get
\begin{equation}\label{eq3.21}
\int_{(0, 1)} \frac{C_{N,p,s}}{|x-y|^{N+ps}} d\mu^+(s) \geq \int_{[\overline s, 1-\delta]} \frac{C_{N,p,s}}{|x-y|^{N+ps}} d\mu^+(s) \geq \frac{\underline{c}_{N,p}\delta}{|x-y|^N} \int_{[\overline s, 1-\delta]}\frac{d\mu^+(s)}{|x-y|^{ps}}.
\end{equation}
From Lemma \ref{lm 2.5} with $S:=1-\delta\in[\overline s,1)$ and \eqref{eq3.20}, we obtain
\begin{equation}\label{eq3.22}
    \int_{(0, \overline s)} \frac{c_{N,p,s}}{|x-y|^{N+ps}} d\mu^-(s) \leq \frac{\overline{c}_{N,p}
\,\max\{(2R)^{p\overline{s}},(2R)^p\}\mu^-( (0,\overline{s}))
}{\min\{1,(2R)^{p\overline{s}}\}\,\mu^+( [\overline{s},1-\delta])|x-y|^N} \int_{[ \overline s,1-\delta]}\frac{d\mu^+(s)}{|x-y|^{ps}}.
\end{equation}
Combining \eqref{eq3.22} with \eqref{eq3.21} and using Lemma \ref{lm 2.4}, we deduce that
\begin{align}\label{eq3.23}
    \int_{(0, \overline s)} \frac{C_{N,p,s}}{|x-y|^{N+ps}}d\mu^-(s) &\leq \frac{\overline{c}_{N,p}\max\{(2R)^{p\overline{s}},(2R)^p\}\,\mu^-((0,\overline{s}))}{\underline{c}_{N,p}\delta\,\min\{1,(2R)^{p\overline{s}}\}\mu^+([\overline{s},1-\delta])} \int_{(0,1)} \frac{C_{N,p,s}}{|x-y|^{N+ps}} d\mu^+(s)\nonumber\\
&<\frac{\overline{c}_{N,p}\underline{\Gamma}_{N,p}\overline{s}\max\{(2R)^{p\overline{s}},(2R)^p\}}{4\underline{c}_{N,p}\overline{\Gamma}_{N,p}\min\{1,(2R)^{p\overline{s}}\}} \int_{(0,1)} \frac{C_{N,p,s}}{|x-y|^{N+ps}}d\mu^+(s)\nonumber\\
&=\frac{\max\{(2R)^{p\overline{s}},(2R)^p\}}{\min\{1,(2R)^{p\overline{s}}\} \max\{1, (2R)^p\}} \int_{(0,1)} \frac{C_{N,p,s}}{|x-y|^{N+ps}} d\mu^+(s)\nonumber\\
&=\int_{(0,1)} \frac{C_{N,p,s}}{|x-y|^{N+ps}} d\mu^+(s).
\end{align}
Hence, from \eqref{eq3.23}, \eqref{eq3.18} and \eqref{eq2.17}, we conclude that \eqref{eq2.15} holds.
\end{proof}

\section{Some technical results: Weak lower semicontinuity properties} \label{sec3}
In this section, we will state and prove all the necessary preliminary results which will be crucial to study the critical elliptic problems related to the nonlocal operator $A_{\mu, p}$ in the subsequent sections. First, let us denote by
\begin{equation}\label{Sobolev constant 2}
C_{p_{s_\sharp}^{*}}:=\sup _{u \in X_{p}(\Omega) \backslash\{0\}} \frac{\|u\|_{L^{p_{s_{\sharp}}^{*}}(\Omega)}}{\left(\int_{[0,1]}[u]_{s, p}^{p} d \mu^{+}(s)\right)^{1/p}} ,
\end{equation}
the best constant of the continuous Sobolev embedding $X_{p}(\Omega) \hookrightarrow L^{p_{s_\sharp}^{*}}(\Omega)$ given by \eqref{embedding} of Proposition \ref{compact and cont embedding}.

As discussed in the introduction, the proof of Theorem \ref{main result wlsc 0} is based on a weak lower semicontinuity result. For we purpose,  we define the functional  
\begin{equation}\label{wlsc functional}
\mathcal{L}_{\gamma}(u):=\frac{1}{p}\left[\rho_{p}(u)\right]^{p}-\frac{1}{p} \int_{[0, \bar{s}]}[u]_{s, p}^{p} d \mu^{-}(s)-\frac{\gamma}{p_{s_{\sharp}}^{*}} \int_{\Omega}|u|^{p_{s_\sharp}^{*}} d x ,
\end{equation}
for every $u \in X_{p}(\Omega)$.

We now present the main result of this section concerning the weak lower semicontinuity of the functional \eqref{wlsc functional}.

\begin{prop}\label{main prop on wlsc}  Let $ \Omega$ be a bounded subset of $\mathbb{R}^N$. Let $\mu=\mu^{+}-\mu^{-}$with $\mu^{+}$and $\mu^{-}$satisfying \eqref{measure 1}-\eqref{measure 3} and $s_{\sharp}$ be as in \eqref{measure 4}.
    Assume that   $\frac{N}{s_{\sharp}}>p \geq 2$ and $p_{s_\sharp}^*=\frac{p N}{N-{s_\sharp}p}$. Let
\begin{equation*}
    \overline{B_{X_{p}(\Omega)}(0, r)}:=\left\{u \in X_{p}(\Omega):\left(\int_{[0,1]}[u]_{s, p}^{p} d \mu^{+}(s)\right)^{1 / p} \leq r\right\}
\end{equation*}
be a closed ball in the space $X_{p}(\Omega)$ centered at $0$ with radius $r>0$. Then for every $\gamma>0$, there exists $\bar{r}_{\gamma}>0$ such that the functional $\mathcal{L}_{\gamma}$ defined in \eqref{wlsc functional} is sequentially weakly lower semicontinuous on $\overline{B_{X_{p}(\Omega)}\left(0, \bar{r}_{\gamma}\right)}$.
\end{prop}

\begin{proof}
Fix $\gamma>0$. Let $r>0$, and take a sequence $\left\{u_{j}\right\}_{j \in \mathbb{N}} \subset \overline{B_{X_{ p}(\Omega)}(0, r)}$ weakly convergent to some $u_{\infty} \in \overline{B_{X_{p}(\Omega)}(0, r)}$, i.e., we have
\begin{align}\label{weak convergence 1}
\begin{split}
    &\int_{[0,1]}\left(\iint_{\mathbb{R}^{2 N}}  \frac{C_{N, s, p}|u_j(x)-u_j(y)|^{p-2}\left(u_{j}(x)-u_{j}(y)\right)(v(x)-v(y))}{|x-y|^{N+p s}} d x d y \right) d\mu^{+}(s) \\
& \rightarrow \int_{[0,1]}\left(\iint_{\mathbb{R}^{2 N}} \frac{C_{N, s, p}|u_\infty(x)-u_\infty(y)|^{p-2}\left(u_{\infty}(x)-u_{\infty}(y)\right)(v(x)-v(y))}{|x-y|^{N+p s}} d x d y\right) d\mu^{+}(s)
\end{split}
\end{align}
as $j \rightarrow+\infty$, for every $v \in X_{p}(\Omega)$.
It follows  from the compact embedding as $p<p_{s_\sharp}^*$, that
\begin{equation}\label{eq3.11}
    u_j \rightarrow u_\infty~\text{in}~L^p(\Omega).
\end{equation}
Therefore, up to a subsequence
\begin{equation}\label{eq3.12}
    u_j(x) \rightarrow u_\infty(x)~\text{a.e in}~\Omega 
\end{equation}
and using a similar argument as in Lemma A.1 \cite{WM97}, we may assume that there exists $h\in L^p(\Omega)$ such that 
\begin{equation}\label{eq3.13}
    |u_j(x)|\leq h(x)~\text{a.e in}~\Omega, \forall\,k \in\mathbb{N}.
\end{equation}
Indeed from \eqref{eq3.12}, we may assume that there exists a subsequence $(w_j)$ of $(u_j)$ such that $$\|w_{k+1}-w_k\|_p\leq \frac{1}{2^k}~\forall\,k\in\mathbb{N}.$$
Define, $g(x)=w_1(x)+\sum_{k=1}^{\infty}\|w_{k+1}(x)-w_k(x)\|_p$. Thus $|w_j(x)|\leq h(x)$ and $|u_\infty(x)|\leq h(x).$
Testing \eqref{weak convergence 1} with $v=u_\infty$, we obtain

 \begin{align}\label{eq3.9}
     & \int_{[0, 1]} \left(\iint_{\mathbb{R}^{2N}} \frac{C_{N, s, p}|u_j(x)-u_j(y)|^{p-2}(u_j(x)-u_j(y)) (u_\infty(x)-u_\infty(y))}{| x-y|^{N+sp}} d x d y \right) d\mu^+(s) \nonumber\\& \rightarrow \int_{[0, 1]} \left( \iint_{\mathbb{R}^{2N}} \frac{C_{N, s, p}|u_\infty(x)-u_\infty(y)|^{p}}{| x-y|^{N+sp}} d x d y \right) d\mu^+(s).
 \end{align}
Notice that for $k=1, 2,3,\dots$, testing \eqref{weak convergence 1} with $v=u_k$ and taking $j\rightarrow\infty$, we get 
 \begin{align}\label{eq3.10}
     & \int_{[0, 1]} \left(\iint_{\mathbb{R}^{2N}} \frac{C_{N, s, p} |u_j(x)-u_j(y)|^{p-2}(u_j(x)-u_j(y)) (u_k(x)-u_k(y))}{|x-y|^{N+sp}} d x d y \right) d\mu^+(s) \nonumber\\& \rightarrow \int_{[0, 1]} \left(\iint_{\mathbb{R}^{2N}} \frac{C_{N, s, p}|u_\infty(x)-u_\infty(y)|^{p-2}(u_\infty(x)-u_\infty(y)) (u_k(x)-u_k(y))}{|x-y|^{N+sp}} d x d y\right) d\mu^+(s).
 \end{align}
Again from \eqref{eq3.11}, \eqref{eq3.12} and the Lebesgue dominated convergence theorem, we get
\begin{align}\label{eq3.14}
     & \int_{[0, 1]} \left(\iint_{\mathbb{R}^{2N}} \frac{C_{N, s, p} |u_j(x)-u_j(y)|^{p}}{|x-y|^{N+sp}} d x d y \right) d\mu^+(s) \\ \nonumber & \rightarrow \int_{[0, 1]} \left(\iint_{\mathbb{R}^{2N}} \frac{C_{N, s, p} |u_\infty(x)-u_\infty(y)|^{p}}{|x-y|^{N+sp}} d x d y \right) d\mu^+(s) .
 \end{align}
Therefore, 
 \begin{align}\label{weak convergence conseq}
     & \int_{[0, 1]} \left(\iint_{\mathbb{R}^{2N}} \frac{C_{N, s, p}|u_\infty (x)-u_\infty(y)|^{p-2}(u_\infty (x)-u_\infty(y)) (u_k(x)-u_k(y))}{|x-y|^{N+sp}} d x d y \right) d\mu^+(s) \nonumber\\& \rightarrow \int_{[0, 1]} \left(\iint_{\mathbb{R}^{2N}} \frac{C_{N, s, p}|u_\infty(x)-u_\infty(y)|^{p}}{|x-y|^{N+sp}} d x d y  \right) d\mu^+(s).
 \end{align}

To prove the result, it is enough to show that
\begin{equation}\label{wlsc}
L:=\liminf _{j \rightarrow+\infty}\left(\mathcal{L}_{\gamma}\left(u_{j}\right)-\mathcal{L}_{\gamma}\left(u_{\infty}\right)\right) \geq 0.
\end{equation}
We first show that 
\begin{equation} \label{eqicon46}
    \lim_{j\rightarrow\infty}\int_{[0,\bar{s})}[u_j]_{s, p}^{p} d \mu^{-}(s)=\int_{[0,\bar{s})}[u_{\infty}]_{s, p}^{p} d \mu^{-}(s).
\end{equation}
Note that the sequence $(u_j)$ is bounded in $X_p(\Omega)$, since it is weakly convergent to $u_{\infty}$ in $X_p(\Omega)$. Therefore, form Proposition \ref{compact and cont embedding}, we get 
\begin{equation}\label{eq4.11}
    u_j\rightarrow u_{\infty}~\text{strongly in}~L^r(\Omega),~\forall\,r\in[1,p_{s_{\sharp}}^*)~\text{and}~u_j\rightarrow u_{\infty}~\text{a.e. in }\Omega.
\end{equation}
In addition, by using  \cite[Theorem 1]{BM2018}, we obtain that, for any $s\in[0,\bar{s})$, there exists $\theta\in(0,1)$ such that
\begin{equation}\label{eq4.12}
    \|u_j-u_{\infty}\|_{s, p}\leq C\|u_j-u_{\infty}\|_{p}^{\theta}\|u_j-u_{\infty}\|_{\bar{s}, p}^{1-\theta}<\infty.
\end{equation}
Thus from \eqref{eq4.11} and \eqref{eq4.12}, we obtain
\begin{align}
    \lim_{j\rightarrow\infty}\int_{[0,\bar{s})}[u_j-u_{\infty}]_{s, p}^{p} d \mu^{-}(s)&\leq \lim_{j\rightarrow\infty}\int_{[0,\bar{s})}\|u_j-u_{\infty}\|_{s, p}^{p} d \mu^{-}(s)\nonumber\\
    &\leq C \mu^{-}([0,\bar{s}))\lim_{j\rightarrow\infty} \|u_j-u_{\infty}\|_{p}^{p\theta}\|u_j-u_{\infty}\|_{\bar{s}, p}^{p(1-\theta)}\nonumber\\
    &\leq C \mu^{-}([0,\bar{s}))\lim_{j\rightarrow\infty}\|u_j-u_{\infty}\|_{p}^{p\theta}\nonumber\\
    &=0.
\end{align} 
Now using the Brezis-Lieb type lemma (see Lemma \ref{B-L lemma}), we get
\begin{equation}\label{eq4.14}
    \lim_{j\rightarrow\infty}\int_{[0,\bar{s})}[u_j]_{s, p}^{p} d \mu^{-}(s)=\int_{[0,\bar{s})}[u_{\infty}]_{s, p}^{p} d \mu^{-}(s).
\end{equation}
Therefore, to prove \eqref{wlsc}, it is sufficient to deduce that
\begin{align*}
& \liminf _{j \rightarrow+\infty}\left\{\frac{1}{p}\left[\int_{[0,1]} [u_j]^p_{s,p}d \mu^{+}(s)-\int_{[0,1]} [u_\infty]^p_{s,p}d \mu^{+}(s)\right]\right. \\
& \left.\quad-\frac{\gamma}{p_{s_\sharp}^{*}}\left[\int_{\Omega}\left|u_{j}(x)\right|^{p_{s_\sharp}^{*}} d x-\int_{\Omega}\left|u_{\infty}(x)\right|^{p_{s_\sharp}^{*}} d x\right]\right\} \geq 0 .
\end{align*}
Let us recall the following inequality \cite[Lemma 4.2]{Lindqvist: 1990}:

\begin{equation}\label{useful ineq}
|b|^{p}-|a|^{p} \geq p|a|^{p-2} a(b-a)+2^{1-p}|a-b|^{p}, \quad  
\end{equation}
which holds for $\forall a, b \in \mathbb{R}$ and $p \geq 2$. Now, choosing
\begin{equation*}
    a:=u_{\infty}(x)-u_{\infty}(y) \quad \text { and } \quad b:=u_{j}(x)-u_{j}(y), \quad \forall x, y \in \mathbb{R}^{N}, \quad  j \in \mathbb{N},
\end{equation*}
and using \eqref{useful ineq} we get
\begin{align}\label{useful ineq appl}
\begin{split}
    & \frac{1}{p}\left[\int_{[0,1]} [u_j]^p_{s,p}d \mu^{+}(s)-\int_{[0,1]} [u_\infty]^p_{s,p}d \mu^{+}(s)\right]\\
& =\frac{1}{p}\int_{[0,1]} C_{N, s, p}\left[\iint_{\mathbb{R}^{2 N}} \frac{\left|u_{j}(x)-u_{j}(y)\right|^{p}}{|x-y|^{N+p s}} d x d y-\iint_{\mathbb{R}^{2 N}} \frac{\left|u_{\infty}(x)-u_{\infty}(y)\right|^{p}}{|x-y|^{N+p s}} d x d y\right]d \mu^{+}(s) \\
& \geq \int_{[0,1]} C_{N, s, p} \Bigg[\iint_{\mathbb{R}^{2 N}} \frac{\left|u_{\infty}(x)-u_{\infty}(y)\right|^{p-2}\left(u_{\infty}(x)-u_{\infty}(y)\right)\left(u_{j}(x)-u_{j}(y)\right)}{|x-y|^{N+p s}} d x d y  \\
&-\iint_{\mathbb{R}^{2 N}} \frac{\left|u_{\infty}(x)-u_{\infty}(y)\right|^{p}}{|x-y|^{N+p s}} d x d y\\
&+\frac{2^{1-p}}{p} \iint_{\mathbb{R}^{2 N}} \frac{\left|\left(u_{j}-u_{\infty}\right)(x)-\left(u_{j}-u_{\infty}\right)(y)\right|^{p}}{|x-y|^{N+p s}} d x d y \Bigg]d\mu^{+}(s). 
\end{split}
\end{align}
Moreover, by the Brezis-Lieb lemma \cite{BL: 1983}, we have 
\begin{equation}\label{BL lemma appl}
\liminf _{j \rightarrow+\infty}\left[\int_{\Omega}\left|u_{j}(x)\right|^{p_{s_\sharp}^{*}} d x-\int_{\Omega}\left|u_{\infty}(x)\right|^{p_{s_\sharp}^{*}} d x\right]=\liminf _{j \rightarrow+\infty} \int_{\Omega}\left|u_{j}(x)-u_{\infty}(x)\right|^{p_{s_\sharp}^{*}} d x. 
\end{equation}
Hence, applying \eqref{useful ineq appl}, \eqref{BL lemma appl} and \eqref{weak convergence conseq} in \eqref{wlsc}  we get
\begin{align}\label{final estim 1}
\begin{split}
    & \liminf _{j \rightarrow+\infty}\left(\mathcal{L}_{\gamma}\left(u_{j}\right)-\mathcal{L}_{\gamma}\left(u_{\infty}\right)\right) \\
& \geq \liminf _{j \rightarrow+\infty}\{ \frac{2^{1-p}}{p} \int_{[0,1]} \left( \iint_{\mathbb{R}^{2 N}} \frac{C_{N, s, p}\left|\left(u_{j}-u_{\infty}\right)(x)-\left(u_{j}-u_{\infty}\right)(y)\right|^{p}}{|x-y|^{N+p s}} d x d y \right) d \mu^{+}(s) \\
&-\frac{\gamma}{p_{s_\sharp}^{*}} \int_{\Omega}\left|u_{j}(x)-u_{\infty}(x)\right|^{p_{s_\sharp}^{*}} d x \} .
\end{split}
\end{align}
Now, using the fact that $\left\{u_{j}-u_{\infty}\right\}_{j \in \mathbb{N}} \subset \overline{B_{X_{p}(\Omega)}(0,2 r)}$ and the continuous embedding \eqref{embedding}, we obtain from \eqref{final estim 1} that
\begin{align}\label{final estim 2}
\begin{split}
    L & \geq \liminf _{j \rightarrow+\infty} \left\{ \frac{2^{1-p}}{p} [\rho_p(u_j-u_\infty)]^p - \frac{\gamma}{p_{s_\sharp}^{*}} \|u_j-u_\infty\|_{L^{p_{s_\sharp}^{*}}(\Omega)}^{p_{s_\sharp}^{*}}  \right\}\\
&\geq \liminf _{j \rightarrow+\infty}[\rho_p(u_j-u_\infty)]^p\left(\frac{2^{1-p}}{p}-\frac{\gamma C_{p_{s_\sharp}^{*}}^{p_{s_\sharp}^{*}}}{p_{s_\sharp}^{*}}[\rho_p(u_j-u_\infty)]^{p_{s_\sharp}^{*}-p}\right) \\
& \geq \liminf _{j \rightarrow+\infty}[\rho_p(u_j-u_\infty)]^p\left(\frac{2^{1-p}}{p}-\frac{\gamma C_{p_{s_\sharp}^{*}}^{p_{s_\sharp}^{*}} 2^{p_{s_\sharp}^{*}-p}}{p_{s_\sharp}^{*}} r^{p_{s_\sharp}^{*}-p}\right),
\end{split}
\end{align}
where $C_{p_{s_\sharp}^{*}}$ is the positive constant defined in \eqref{Sobolev constant 2}.
Consequently, for $r$ sufficiently small, such that,
\begin{equation*}
    0<r \leq \frac{1}{2}\left(\frac{2^{1-p} p_{s_\sharp}^{*}}{p \gamma C_{p_{s_\sharp}^{*}}^{p_{s_\sharp}^{*}}}\right)^{1 /(p_{s_\sharp}^{*}-p)}
\end{equation*}
inequality \eqref{wlsc} follows from \eqref{final estim 2}.
Therefore, for any  
\begin{equation*}
    \bar{r}_{\gamma} \in\left(0, \frac{1}{2}\left(\frac{2^{1-p} p_{s_\sharp}^{*}}{p \gamma C_{p_{s_\sharp}^{*}}^{p_{s_\sharp}^{*}}}\right)^{1 /(p_{s_\sharp}^{*}-p)}\right],
\end{equation*}
the functional $\mathcal{L}_{\gamma}$ is sequentially weak lower semicontinuous on $\overline{B_{X_{p}(\Omega)}\left(0, \bar{r}_{\gamma}\right)}$. This completes the proof. 
\end{proof}

Next, we prove some useful estimates essential for the proof of Theorem \ref{main result wlsc 0}.

Let us fix $\gamma, \lambda>0$ and define the functionals
\begin{equation*}
    \Phi_{p, s}(u):=\rho_p(u)=\left(\int_{[0,1]}\left(\iint_{\mathbb{R}^{2 N}} \frac{C_{N, s, p}|u(x)-u(y)|^{p}}{|x-y|^{N+p s}} d x d y\right) d \mu^{+}(s)\right)^{1 / p}
\end{equation*}
and
\begin{equation}\label{part 2 of decomp}
    \Psi_{\gamma, \lambda}(u):=\frac{1}{p} \int_{[0, \bar{s}]}[u]_{s, p}^{p} d \mu^{-}(s)+\frac{\gamma}{p_{s_\sharp}^{*}} \int_{\Omega}|u(x)|^{p_{s_\sharp}^{*}} d x+\lambda \int_{\Omega} G(x, u(x)) d x,
\end{equation}
for every $u \in X_{ p}(\Omega)$, where the potential $G$ is given by 
\begin{equation}
G(x, t):=\int_{0}^{t} g(x, \tau) d \tau, \quad \forall(x, t) \in \Omega \times \mathbb{R}. 
\end{equation}
 So, we have
\begin{equation}\label{functional as a sum}
    \mathcal{J}_{\gamma, \lambda}(u)=\frac{1}{p}\Phi_{p, s}(u)^p-\Psi_{\gamma, \lambda}(u),
\end{equation}
for every $u \in X_{ p}(\Omega)$.
Next, we set
\begin{equation*}
  \Theta_{\gamma, \lambda}(\eta, \zeta):=\sup _{v \in \Phi_{p, s}^{-1}([0, \eta])} \Psi_{\gamma, \lambda}(v)-\sup _{v \in \Phi_{p, s}^{-1}([0, \eta-\zeta])} \Psi_{\gamma, \lambda}(v),
\end{equation*}
for $0<\zeta<\eta$.  
\begin{lem}\label{tech lem 1}
    Let $p\in(1, +\infty)$ and $\gamma, \lambda>0$. Suppose that
\begin{equation}\label{cond 1.1}
\limsup _{\varepsilon \rightarrow 0^{+}} \frac{\Theta_{\gamma, \lambda}\left(r_{0}, \varepsilon\right)}{\varepsilon}<r^{p-1}_{0} 
\end{equation}
for some $r_{0}>0$. Then
\begin{equation}\label{cond 1.2}
\inf _{\sigma<r_{0}} \frac{\Theta_{\gamma, \lambda}\left(r_{0}, r_{0}-\sigma\right)}{r_{0}^{p}-\sigma^{p}}<\frac{1}{p}. 
\end{equation}
\end{lem}

\begin{proof}
Let us observe that if $\varepsilon \in\left(0, r_{0}\right)$, then we have
\begin{equation*}
  \frac{\Theta_{\gamma, \lambda}\left(r_{0}, \varepsilon\right)}{r_{0}^{p}-\left(r_{0}-\varepsilon\right)^{p}}=\frac{\Theta_{\gamma, \lambda}\left(r_{0}, \varepsilon\right)}{\varepsilon} \frac{-\varepsilon}{r^{p}_{0}\left[\left(1-\varepsilon / r_{0}\right)^{p}-1\right]}  
\end{equation*}
and
\begin{equation*}
    \lim _{\varepsilon \rightarrow 0^{+}} \frac{-\varepsilon}{r^{p}_{0}\left[\left(1-\varepsilon / r_{0}\right)^{p}-1\right]}=\frac{1}{p r^{p-1}_{0}}.
\end{equation*}
Hence, from \eqref{cond 1.1} it follows that
\begin{equation}\label{cond 1.3}
\limsup _{\varepsilon \rightarrow 0^{+}} \frac{\Theta_{\gamma, \lambda}\left(r_{0}, \varepsilon\right)}{r_{0}^{p}-\left(r_{0}-\varepsilon\right)^{p}}<\frac{1}{p}. 
\end{equation}
Further, by \eqref{cond 1.3}, we assert that there exists a constant $\bar{\varepsilon}>0$ such that
\begin{equation*}
    \frac{\Theta_{\gamma, \lambda}\left(r_{0}, \varepsilon\right)}{r_{0}^{p}-\left(r_{0}-\varepsilon\right)^{p}}<\frac{1}{p}
\end{equation*}
for every $\varepsilon \in(0, \bar{\varepsilon})$. Then, setting $\sigma_{0}:=r_{0}-\varepsilon_{0}$ with $\varepsilon_{0} \in(0, \bar{\varepsilon})$, we obtain
\begin{equation*}
    \frac{\Theta_{\gamma, \lambda}\left(r_{0}, r_{0}-\sigma_{0}\right)}{r_{0}^{p}-\sigma_{0}^{p}}<\frac{1}{p},
\end{equation*}
verifying the inequality \eqref{cond 1.2}.
\end{proof}

We now recall a fundamental property of lower semicontinuous functions defined on a topological space $\left(M, \tau_{M}\right)$. Let $\bar{A}$ be the topological closure of the set $A\in M$ for the topology $\tau_{M}$. If $g: \bar{A} \rightarrow \mathbb{R}$ is a lower semicontinuous function, then we have
\begin{equation*}
    \sup _{x \in \bar{A}} g(x)=\sup _{x \in A} g(x) .
\end{equation*}

Let us take $A:=\{v\in X_p(\Omega) : \Phi_{p, s}(v)=\sigma \}$. Then the weak closure of $A$ is $\bar{A}:=\Phi_{p, s}^{-1}([0, \sigma])$ and the functional $\Psi_{\gamma, \lambda}$ is sequentially weakly lower semicontinuous on $\Phi_{p, s}^{-1}([0, \sigma])$. Consequently, it follows from the aforementioned property that
\begin{equation}\label{sup equality}
\sup _{v \in \Phi_{p, s}^{-1}[[0, \sigma])} \Psi_{\gamma, \lambda}(v)=\sup _{v \in A} \Psi_{\gamma, \lambda}(v) 
\end{equation}
for every $\sigma>0$. 

\begin{lem}\label{tech lem 2}
Let $p\in(1, +\infty)$ and $\gamma, \lambda>0$. Suppose that \eqref{cond 1.2} holds for some $r_{0}>0$. Then we have
\begin{equation}\label{cond 2.1}
\inf _{u \in \Phi_{p, s}^{-1}\left(\left[0, r_{0}\right)\right)} \frac{\sup _{v \in \Phi_{p, s}^{-1}\left(\left[0, r_{0}\right]\right)} \Psi_{\gamma, \lambda}(v)-\Psi_{\gamma, \lambda}(u)}{r_{0}^{p}-\Phi_{p, s}(u)^{p}}<\frac{1}{p}. 
\end{equation}
\end{lem} 
\begin{proof}
In light of \eqref{cond 1.2} we have
\begin{equation}\label{cond 2.2}
\sup _{v \in \Phi_{p, s}^{-1}\left[\left[0, \sigma_{0}\right]\right)} \Psi_{\gamma, \lambda}(v)>\sup _{v \in \Phi_{p, s}^{-1}\left(\left[0, r_{0}\right]\right)} \Psi_{\gamma, \lambda}(v)-\frac{r_{0}^{p}-\sigma_{0}^{p}}{p}, 
\end{equation}
for some $0<\sigma_{0}<r_{0}$. Moreover, by \eqref{sup equality} there exists $u_{0} \in X_{ p}(\Omega)$ with $\Phi_{p, s}\left(u_{0}\right)=\sigma_{0}$, such that 
\begin{equation*}
\sup _{v \in \Phi_{p, s}^{-1}[[0, \sigma_{0}])} \Psi_{\gamma, \lambda}(v)=\sup _{v \in A} \Psi_{\gamma, \lambda}(v)= \Psi_{\gamma, \lambda}(u_0).
\end{equation*}
Consequently, from \eqref{cond 2.2} it follows that
\begin{equation}
\Psi_{\gamma, \lambda}\left(u_{0}\right)>\sup _{v \in \Phi_{p, s}^{-1}\left(\left[0, r_{0}\right]\right)} \Psi_{\gamma, \lambda}(v)-\frac{r_{0}^{p}-\Phi_{p, s}\left(u_{0}\right)^{p}}{p} .
\end{equation}
Thus, we get 
\begin{equation}
\inf _{u \in \Phi_{p, s}^{-1}\left(\left[0, r_{0}\right)\right)} \frac{\sup _{v \in \Phi_{p, s}^{-1}\left(\left[0, r_{0}\right]\right)} \Psi_{\gamma, \lambda}(v)-\Psi_{\gamma, \lambda}(u)}{r_{0}^{p}-\Phi_{p, s}(u)^{p}}<\frac{\sup _{v \in \Phi_{p, s}^{-1}\left(\left[0, r_{0}\right]\right)} \Psi_{\gamma, \lambda}(v)-\Psi_{\gamma, \lambda}\left(u_{0}\right)}{r_{0}^{p}-\Phi_{p, s}\left(u_{0}\right)^{p}}<\frac{1}{p}, 
\end{equation}
completing the proof.
\end{proof}

\begin{lem}\label{tech lem 3}
    Let $p\in(1, +\infty)$ and $\gamma, \lambda>0$. Assume that \eqref{cond 2.1} holds for some $r_{0}>0$. Then there exists $w \in \Phi_{p, s}^{-1}\left(\left[0, r_{0}\right)\right)$ such that
\begin{equation}\label{cond 3.1}
\mathcal{J}_{\gamma, \lambda}(w)=\frac{1}{p}[\rho_p(w)]^{p}-\Psi_{\gamma, \lambda}(w)<\frac{r_{0}^{p}}{p}-\Psi_{\gamma, \lambda}(u) 
\end{equation}
for every $u \in \Phi_{p, s}^{-1}\left(\left[0, r_{0}\right]\right)$.
\end{lem}
\begin{proof}
By assumption, we have
\begin{equation*}
    \inf _{u \in \Phi_{p, s}^{-1}\left(\left[0, r_{0}\right)\right)} \frac{\sup _{v \in \Phi_{p, s}^{-1}\left(\left[0, r_{0}\right]\right)} \Psi_{\gamma, \lambda}(v)-\Psi_{\gamma, \lambda}(u)}{r_{0}^{p}-\Phi_{p, s}(u)^{p}}<\frac{1}{p},
\end{equation*}
for some $r_{0}>0$. From this, we deduce that there exists $w \in \Phi_{p, s}^{-1}\left(\left[0, r_{0}\right)\right)$ such that
\begin{equation*}
    \Psi_{\gamma, \lambda}(u) \leq \sup _{v \in \Phi^{-1}\left(\left[0, r_{0}\right]\right)} \Psi_{\gamma, \lambda}(v)<\Psi_{\gamma, \lambda}(w)+\frac{r_{0}^{p}-\Phi_{p, s}(w)^{p}}{p},
\end{equation*}
for every $u \in \Phi_{p, s}^{-1}\left(\left[0, r_{0}\right]\right)$. Thus, in view of \eqref{functional as a sum}, we conclude that
\begin{equation*}
    \mathcal{J}_{\gamma, \lambda}(w)<\frac{r_{0}^{p}}{p}-\Psi_{\gamma, \lambda}(u),
\end{equation*}
for every $u \in \Phi_{p, s}^{-1}\left(\left[0, r_{0}\right]\right)$. 
\end{proof}

\section{Proof of Theorem \ref{main result wlsc 0}} \label{sec4}
The aim of this section is to present a proof of Theorem \ref{main result wlsc 0} concerning the existence of solutions to the following critical problem
\begin{equation}\label{problem 1}
 \left\{\begin{array}{cc}
    A_{\mu, p} u  =\lambda g(x,u) +\gamma |u|^{p_{s_\sharp}^{*}-2} u \text { in } \Omega,  \\
    u  = 0  \text { in } \mathbb{R}^{N} \backslash \Omega ,
\end{array}\right.
\end{equation}
where $\gamma$ and $\lambda$ are real positive parameters and $p_{s_\sharp}^{*}$ is a critical exponent given in \eqref{critical exponent}. Furthermore, $g: \Omega \times \mathbb{R} \rightarrow \mathbb{R}$ is a Carath\'eodory function verifying the following standard subcritical growth condition:

\begin{align}\label{cond on g}
\begin{split}
&\text{ there exist } a_{1}, a_{2}>0 \text{ and } q \in[1, p_{s_{\sharp}}^{*}), \text{ such that }\\
&|g(x, t)| \leq a_{1}+a_{2}|t|^{q-1} \text {, a.e. } x \in \Omega, \forall t \in \mathbb{R} . 
\end{split}
\end{align}

Next, we give the definition of a weak solution to the problem \eqref{problem 1}.

\begin{definition}
    A weak solution to the problem \eqref{problem 1} is a function $u\in X_{p}(\Omega)$ such that\footnote{We make it clear that we are using $$\int_{[0,1]}\left(\iint_{\mathbb{R}^{2 N}}  \frac{C_{N, s, p}|u(x)-u(y)|^{p-2}(u(x)-u(y))(v(x)-v(y))}{|x-y|^{N+s p}} d x d y\right) d \mu^{+}(s),$$ or same for $\mu^-$ with an abuse of notation. More precisely, it must be written as \begin{align*}
   &  \int_{(0,1)}\left(\iint_{\mathbb{R}^{2 N}}  \frac{C_{N, s, p}|u(x)-u(y)|^{p-2}(u(x)-u(y))(v(x)-v(y))}{|x-y|^{N+s p}} d x d y\right) d \mu^{+}(s) \\&\quad\quad+\mu^+(\{0\})\int_\Omega u(x) v(x)\, dx+ \mu^+(\{1\}) \int_\Omega |\nabla u|^{p-2} \langle \nabla u, \nabla v \rangle\, dx.
\end{align*}   }
\begin{align*}
& \int_{[0,1]}\left(\iint_{\mathbb{R}^{2 N}}  \frac{C_{N, s, p}|u(x)-u(y)|^{p-2}(u(x)-u(y))(v(x)-v(y))}{|x-y|^{N+s p}} d x d y\right) d \mu^{+}(s) \\
& =\int_{[0, \bar{s}]}\left(\iint_{\mathbb{R}^{2 N}} \frac{C_{N, s, p}|u(x)-u(y)|^{p-2}(u(x)-u(y))(v(x)-v(y))}{|x-y|^{N+s p}} d x d y\right) d \mu^{-}(s) \\
& +\lambda \int_{\Omega}g(x, u) v d x+\gamma\int_{\Omega}|u|^{p_{s_\sharp}^{*}-2} u v d x,
\end{align*}
for all $v \in X_{p}(\Omega)$.
\end{definition}

Due to the variational structure, the solutions of problem \eqref{problem 1} correspond to the critical points of the functional $\mathcal{J}_{\gamma,\lambda}: X_{p}(\Omega) \rightarrow \mathbb{R}$ defined as
\begin{equation}\label{functional for problem 1}
\mathcal{J}_{\gamma,\lambda}(u):=\frac{1}{p}\left[\rho_{p}(u)\right]^{p}-\frac{1}{p} \int_{[0, \bar{s}]}[u]_{s, p}^{p} d \mu^{-}(s)-\lambda \int_{\Omega}  G(x, u(x)) dx-\frac{\gamma}{p_{s_{\sharp}}^{*}} \int_{\Omega}|u|^{p_{s_\sharp}^{*}} d x, 
\end{equation}
where $G$ is given by
\begin{equation}\label{int of g}
G(x, t):=\int_{0}^{t} g(x, \tau) d \tau, \quad \forall(x, t) \in \Omega \times \mathbb{R}. 
\end{equation}
Moreover, we have
\begin{align*}
\langle \mathcal{J}^{\prime}_{\gamma,\lambda}(u), v\rangle&= \int_{[0,1]}\left(\iint_{\mathbb{R}^{2 N}}  \frac{C_{N, s, p}|u(x)-u(y)|^{p-2}(u(x)-u(y))(v(x)-v(y))}{|x-y|^{N+s p}} d x d y\right) d \mu^{+}(s) \\
& -\int_{[0, \bar{s}]}\left(\iint_{\mathbb{R}^{2 N}} \frac{C_{N, s, p}|u(x)-u(y)|^{p-2}(u(x)-u(y))(v(x)-v(y))}{|x-y|^{N+s p}} d x d y\right) d \mu^{-}(s) \\
& -\lambda \int_{\Omega}g(x, u) v d x-\gamma \int_{\Omega}|u|^{p_{s_\sharp}^{*}-2} u v d x,
\end{align*}
for all $u, v \in X_{p}(\Omega)$, where $\langle \cdot, \cdot\rangle$ is the dual product between the space $X_{p}(\Omega)$ and its dual.

At this point, we are in a position to begin the proof of Theorem \ref{main result wlsc 0}. 

\begin{proof}[Proof of Theorem \ref{main result wlsc 0}]  First, let us fix $\gamma>0$ and define the following function $h_{\gamma}:(0, +\infty) \to \mathbb{R}$ given by
\begin{equation}\label{function h}
h_{\gamma}(r):=\frac{r^{p-1}-\gamma C_{p_{s_\sharp}^{*}}^{p_{s_\sharp}^{*}} r^{p_{s_\sharp}^{*}-1}}{a_{1} C_{p_{s_\sharp}^{*}}|\Omega|^{(p_{s_\sharp}^{*}-1) / p_{s_\sharp}^{*}}+a_{2} C_{p_{s_\sharp}^{*}}^{q}|\Omega|^{(p_{s_\sharp}^{*}-q) / p_{s_\sharp}^{*}} r^{q-1}}, \quad \forall r \geq 0,
\end{equation}
where, $\Omega$ is an open subset of $\mathbb{R}^{N}$ with smooth boundary $\partial \Omega$, $|\Omega|$ is the Lebesgue measure of $\Omega$, $q \in[1, p_{s_\sharp}^{*})$, with $p_{s_\sharp}^{*}:=p N /(N- s_\sharp p),$ $C_{p_{s_\sharp}^{*}}$ is the critical Sobolev constant given in \eqref{Sobolev constant 2}, and $a_{1}, a_{2}$ are the positive real constants from \eqref{cond on g}. It is easy to see that the function $h_{\gamma}$ has a global maximum.
Let $r_{\gamma, \max }>0$ be the global maximum of a function $h_{\gamma}$ defined in \eqref{function h}. Set $r_{0, \gamma}:=\min \left\{r_{\gamma, \max }, \bar{r}_{\gamma}\right\}$, where $\bar{r}_{\gamma}$ is defined in Proposition \ref{main prop on wlsc}. If we take 
\begin{equation*}
    \lambda < \Lambda_{\gamma}:=h_{\gamma}\left(r_{0, \gamma}\right),
\end{equation*}
then we can find some $r_{0, \gamma, \lambda} \in\left(0, r_{0, \gamma}\right)$ such that
\begin{equation}\label{cond 4.1}
\lambda<\frac{r^{p-1}_{0,\gamma,\lambda}-\gamma C_{p_{s_\sharp}^{*}}^{p_{s_\sharp}^{*}} r_{0,\gamma,\lambda}^{p_{s_\sharp}^{*}-1}}{a_{1} C_{p_{s_\sharp}^{*}}|\Omega|^{(p_{s_\sharp}^{*}-1) / p_{s_\sharp}^{*}}+a_{2} C_{p_{s_\sharp}^{*}}^{q}|\Omega|^{(p_{s_\sharp}^{*}-q) / p_{s_\sharp}^{*}} r_{0,\gamma,\lambda}^{q-1}}. 
\end{equation}
Now, for $0<\varepsilon<r_{0, \gamma, \lambda},$ define
\begin{equation*}
    \Lambda_{\lambda, \gamma}\left(\varepsilon, r_{0, \gamma, \lambda}\right):=\frac{\sup _{v \in \Phi_{p, s}^{-1}([0, r_{0, \gamma, \lambda}])} \Psi_{\gamma, \lambda}(v)-\sup _{v \in \Phi_{p, s}^{-1}([0, r_{0, \gamma, \lambda}-\varepsilon])} \Psi_{\gamma, \lambda}(v)}{\varepsilon}.
\end{equation*}
Then, by rescaling $v$ and using \eqref{cond on g} and \eqref{part 2 of decomp}, we obtain
\begin{align*}
    \Lambda_{\lambda, \gamma}\left(\varepsilon, r_{0, \gamma, \lambda}\right)& \leq \frac{1}{\varepsilon}\left|\sup _{v \in \Phi_{p, s}^{-1}([0, r_{0, \gamma, \lambda}])} \Psi_{\gamma, \lambda}(v)-\sup _{v \in \Phi_{p, s}^{-1}([0, r_{0, \gamma, \lambda}-\varepsilon])} \Psi_{\gamma, \lambda}(v)\right|\\
    &\leq \sup _{v \in \Phi_{p, s}^{-1}([0,1])} \int_{\Omega}\left|\int_{\left(r_{0, \gamma, \lambda}-\varepsilon\right) v(x)}^{r_{0, \gamma, \lambda} v(x)} \frac{\left|g_{\gamma, \lambda}(x, t)\right|}{\varepsilon} d t\right| d x,
\end{align*}
where $g_{\gamma, \lambda}(x, t):=\gamma|t|^{p_{s_\sharp}^{*}-2} t+\lambda g(x, t),$ for a.e. $x \in \Omega$ and for all $t \in \mathbb{R}$. Next, using the growth condition \eqref{cond on g}, and applying Hölder inequality and the continuous embedding \eqref{embedding}, we obtain
\begin{align*}
& \sup _{v \in \Phi_{p, s}^{-1}([0,1])} \int_{\Omega}\left|\int_{\left(r_{0, \gamma, \lambda}-\varepsilon\right) v(x)}^{r_{0, \gamma, \lambda} v(x)} \frac{\left|g_{\gamma, \lambda}(x, t)\right|}{\varepsilon} d t\right| d x \leq \frac{\gamma C_{p_{s_\sharp}^{*}}^{p_{s_\sharp}^{*}}}{p^{*}_{s_{\sharp}}}\left(\frac{r_{0, \gamma, \lambda}^{p_{s_\sharp}^{*}}-\left(r_{0, \gamma, \lambda}-\varepsilon\right)^{p_{s_\sharp}^{*}}}{\varepsilon}\right) \\
& \quad+\lambda\left(a_{1} C_{p_{s_\sharp}^{*}}|\Omega|^{(p_{s_\sharp}^{*}-1) / p_{s_\sharp}^{*}}+a_{2} \frac{C_{p_{s_\sharp}^{*}}^{q}}{q}\left(\frac{r_{0, \gamma, \lambda}^{q}-\left(r_{0, \gamma, \lambda}-\varepsilon\right)^{q}}{\varepsilon}\right)|\Omega|^{(p_{s_\sharp}^{*}-q) / p_{s_\sharp}^{*}}\right),
\end{align*}
where $C_{p_{s_\sharp}^{*}}$ denotes the embedding constant for $X_{p}(\Omega) \hookrightarrow L^{p_{s_\sharp}^{*}}(\Omega)$.
Hence, passing to the limsup as $\varepsilon \rightarrow 0^{+}$, and using \eqref{cond 4.1}, we get
\begin{align}\label{cond 4.2}
\begin{split}
    \limsup _{\varepsilon \rightarrow 0^{+}} \Lambda_{\lambda, \gamma}\left(\varepsilon, r_{0, \gamma, \lambda}\right)&<\gamma C_{p_{s_\sharp}^{*}}^{p_{s_\sharp}^{*}} r_{0, \gamma, \lambda}^{p_{s_\sharp}^{*}-1}+\lambda\left(a_{1} C_{p_{s_\sharp}^{*}}|\Omega|^{(p_{s_\sharp}^{*}-1) / p_{s_\sharp}^{*}}+a_{2} C_{p_{s_\sharp}^{*}}^{q}|\Omega|^{(p_{s_\sharp}^{*}-q) / p_{s_\sharp}^{*}} r_{0, \gamma, \lambda}^{q-1}\right)\\
    &<r^{p-1}_{0, \gamma, \lambda}.
\end{split}
\end{align}
By combining \eqref{cond 4.2} with Lemmas \ref{tech lem 1} and \ref{tech lem 2}, we deduce that
\begin{equation*}
    \inf _{u \in \Phi_{p, s}^{-1}([0, r_{0, \gamma, \lambda}))} \frac{\sup _{v \in \Phi_{p, s}^{-1}([0, r_{0, \gamma, \lambda}])} \Psi_{\gamma, \lambda}(v)-\Psi_{\gamma, \lambda}(u)}{r_{0, \gamma, \lambda}^{p}-\Phi_{p, s}(u)^{p}}<\frac{1}{p}.
\end{equation*}
Therefore, by Lemma \ref{tech lem 3} there exists $w_{\gamma, \lambda} \in \Phi_{p, s}^{-1}([0, r_{0, \gamma, \lambda}))$ such that
\begin{equation}\label{contr inequality}
    \mathcal{J}_{\gamma, \lambda}\left(w_{\gamma, \lambda}\right)=\frac{1}{p}[\rho_{p}(w_{\gamma, \lambda})]^{p}-\Psi_{\gamma, \lambda}\left(w_{\gamma, \lambda}\right)<\frac{r_{0, \gamma, \lambda}^{p}}{p}-\Psi_{\gamma, \lambda}(u),
\end{equation}
for every $u \in \Phi_{p, s}^{-1}([0, r_{0, \gamma, \lambda}])$. Since $r_{0, \gamma, \lambda}<\bar{r}_{\gamma}$, and recalling that 
\begin{equation*} \mathcal{J}_{\gamma, \lambda}(u) = \mathcal{L}_{\gamma}(u) - \lambda \int_{\Omega} G(x, u(x)) dx, \end{equation*} 
it follows from Proposition \ref{main prop on wlsc} that the energy functional $\mathcal{J}_{\gamma, \lambda}$ is sequentially weakly lower semicontinuous on $\Phi_{p, s}^{-1}([0, r_{0, \gamma, \lambda}])$. Hence, the restriction $\left.\mathcal{J}_{\gamma, \lambda}\right|_{\Phi_{p, s}^{-1}([0, r_{0, \gamma, \lambda}])}$ has a global minimum $u_{0, \gamma, \lambda} \in$ $\Phi_{p, s}^{-1}([0, r_{0, \gamma, \lambda}])$.

Finally, we claim that $u_{0, \gamma, \lambda}$ belongs to $\Phi_{p, s}^{-1}([0, r_{0, \gamma, \lambda}))$. Indeed, suppose that  $\rho_{p}(u_{0, \gamma, \lambda})=r_{0, \gamma, \lambda}$. Then, by \eqref{contr inequality}, we have
\begin{equation*}
    \mathcal{J}_{\gamma, \lambda}\left(u_{0, \gamma, \lambda}\right)=\frac{r_{0, \gamma, \lambda}^{p}}{p}-\Psi_{\gamma, \lambda}\left(u_{0, \gamma, \lambda}\right)>\mathcal{J}_{\gamma, \lambda}\left(w_{\gamma, \lambda}\right),
\end{equation*}
which is a contradiction with the fact that $u_{0, \gamma, \lambda}$ is a global minimum for $\left.\mathcal{J}_{\gamma, \lambda}\right|_{\Phi_{p, s}^{-1}([0, r_{0, \gamma, \lambda}])}$. Therefore, we conclude that $u_{0, \gamma, \lambda} \in X_{p}(\Omega)$ is a local minimum for the energy functional $\mathcal{J}_{\gamma, \lambda}$ with
\begin{equation*}
   \rho_{p}(u_{0, \gamma, \lambda})<r_{0, \gamma, \lambda},
\end{equation*}
and hence, a weak solution to the problem \eqref{problem 1}. This completes the proof of Theorem \ref{main result wlsc 0}.
\end{proof}

\section{Some implications and the proof of Theorem \ref{thmnonmain}} \label{sec5}

In this section, we discuss applications and implications of Theorem \ref{main result wlsc 0}, along with the proof of Theorem \ref{thmnonmain}. The first application of Theorem \ref{main result wlsc 0} establishes the sign of the solution under an additional condition on the Carath\'eodory function $g$, given by \eqref{cond on g 0}.

\begin{thm} \label{subthm}
     Let $ \Omega$ be a bounded subset of $\mathbb{R}^N$. Let $\mu=\mu^{+}-\mu^{-}$ with $\mu^{+}$and $\mu^{-}$ satisfying \eqref{measure 1}-\eqref{measure 3} and \eqref{extracondi}, and let $s_{\sharp}$ be as in \eqref{measure 4}. Suppose that $\frac{N}{s_{\sharp}}>p \geq 2$ and $q \in[1, p_{s_\sharp}^*)$, where $p_{s_\sharp}^*=\frac{p N}{N-{s_\sharp}p}$. Let $g: \Omega \times \mathbb{R} \rightarrow \mathbb{R}$ be a Carath\'eodory function satisfying the subcritical growth condition
    \begin{equation} \label{growth1}
        |g(x, t)| \leq a_1+a_2|t|^{q-1}\quad \text{for}\,\,\text{a.e.}\,\,x\in \Omega,\,\, \forall\, t \in \mathbb{R},
    \end{equation}
    for some $a_1, a_2>0$   and $g(x, 0)=0$ for a.e. $x \in \Omega.$ We also assume that
     \begin{equation} \label{growth1-1}
        \liminf\limits_{t\rightarrow 0^+}\frac{G(x,t)}{t^p}=+\infty,~\text{ for a.e. } x\in\Omega, ~\text{where}~ G(x, t):=\int_0^t g(x, \tau)\, d \tau,
    \end{equation} holds.
    Then, for every $\gamma>0$ and sufficiently small $\lambda>0,$  the following nonlocal problem
    \begin{equation} \label{pro12}
       \begin{cases} 
A_{\mu, p} u= \gamma |u|^{p_{s_\sharp}^*-2}u+\lambda g(x, u) \quad &\text{in}\quad \Omega, \\
u=0\quad & \text{in}\quad \mathbb{R}^n\backslash \Omega,
\end{cases}
    \end{equation} has a nontrivial nonnegative weak solution $u_{\gamma, \lambda} \geq 0$ and a nontrivial nonpositive weak solution $v_{\gamma, \lambda} \leq 0$ in $X_p(\Omega).$  
    
\end{thm}

\begin{proof} 
   We begin by establishing the existence of a solution to problem \eqref{pro12}. Since we are concerned with nonnegative solutions, it suffices to consider the function
    $$G_+(x, t):=\int_0^t g_+(x, \tau)\, d \tau,$$
    where $$g_+(x, \tau):=\begin{cases}
        g(x, \tau) &\text{if}\quad \tau >0, \\
        0&\text{if} \quad \tau \leq 0,
    \end{cases}$$ for almost every $x \in \Omega$ and $\tau \in \mathbb{R}.$
It is easy to see that $g_+$ is a Carath\'eodory function and satisfies the condition \eqref{growth1} with $g_+(x, 0)=0.$

Now, we define the functional $\mathcal{J}_{\gamma,\lambda}^+:X_p(\Omega) \rightarrow \mathbb{R}$ as 

    \begin{equation}
        \mathcal{J}_{\gamma,\lambda}^+(u):=\frac{1}{p}\left[\rho_{p}(u)\right]^{p}-\frac{1}{p} \int_{[0, \bar{s}]}[u^+]_{s, p}^{p} d \mu^{-}(s)-\lambda \int_{\Omega} G_+(x, u(x)) d x-\frac{\gamma}{p_{s_{\sharp}}^{*}} \int_{\Omega}(u(x)^+)^{p_{s_\sharp}^{*}} d x
    \end{equation}
    for every $u \in X^{s,p}_0(\Omega).$
    Note that $ \mathcal{J}_{\gamma, \lambda}^+$ is Fr\'echet differentiable and therefore, we have 
     \begin{align} \label{M1}
   \nonumber   &\langle  ( \mathcal{J}_{\gamma, \lambda}^+)^{'} (u), v \rangle \\ &:= \int_{[0,1]}\left(\iint_{\mathbb{R}^{2 N}}  \frac{C_{N, s, p}|u(x)-u(y)|^{p-2}(u(x)-u(y))(v(x)-v(y))}{|x-y|^{N+s p}} d x d y\right) d \mu^{+}(s) \nonumber \\
& -\int_{[0, \bar{s}]}\left(\iint_{\mathbb{R}^{2 N}} \frac{C_{N, s, p}|u^+(x)-u^+(y)|^{p-2}(u^+(x)-u^+(y))(v(x)-v(y))}{|x-y|^{N+s p}} d x d y\right) d \mu^{-}(s)\nonumber \\&-\gamma \int_{\Omega} (u(x)^+)^{p_{s_\sharp}^*-1}  \varphi(x) d x-\lambda \int_{\Omega} g_+(x, u(x)) \varphi(x) d x,
       \end{align} for every $v \in X^{s,p}_0(\Omega).$ Then, as an application of Theorem \ref{main result wlsc 0}, for every $\gamma>0$ and for sufficiently small $\lambda>0,$ there exists a critical point $u_{\gamma, \lambda} \in X_p(\Omega)$ of  $\mathcal{J}_{\gamma, \lambda}^+$ and therefore, a weak solution of the problem \eqref{pro12} with $g_+$.

       It remains to show that the obtained solution $u_{\gamma, \lambda}$ is nonnegative on $\Omega.$ For this purpose, let us take $v$ as 
       $v:=u_{\gamma, \lambda}^{-}:=\max\{-u_{\gamma, \lambda}, 0\}.$ We note $u_{\gamma, \lambda} \in X_p(\Omega) $ and $||u_{\gamma, \lambda}(x)|-|u_{\gamma, \lambda}(y)|| \leq |u_{\gamma, \lambda}(x)-u_{\gamma, \lambda}(y)|$ imply that $|u_{\gamma, \lambda}| \in X_p(\Omega).$  Since $|u_{\gamma, \lambda}|$ and $u_{\gamma, \lambda}$ belong to $X_p(\Omega),$ it is clear, using the fact that $X_p(\Omega)$ is a vector space,  that $u_{\gamma, \lambda}^{-} \in X_p(\Omega).$ 
Thus, we test \eqref{M1} with $v= u_{\gamma, \lambda}^{-}$  and use the fact that $u_{\gamma, \lambda}$ is a critical point of $\mathcal{J}_{\gamma, \lambda}^+,$ that is, $\langle  (\mathcal{J}_{\gamma, \lambda}^+)^{'} (u), v \rangle=0$ for all $v \in X_p(\Omega)$, to  obtain that 

 \begin{align} \label{ruz5.5}
   \nonumber   0=&\langle  ( \mathcal{J}_{\gamma, \lambda}^+)^{'} (u_{\gamma, \lambda}), u_{\gamma, \lambda}^{-} \rangle \\ &= \int_{[0,1]}\left(\iint_{\mathbb{R}^{2 N}}  \frac{C_{N, s, p}|u_{\gamma, \lambda}(x)-u_{\gamma, \lambda}(y)|^{p-2}(u_{\gamma, \lambda}(x)-u_{\gamma, \lambda}(y))(u_{\gamma, \lambda}^{-}(x)-u_{\gamma, \lambda}^{-}(y))}{|x-y|^{N+s p}} d x d y\right) d \mu^{+}(s) \nonumber \\
& -\int_{[0, \bar{s}]}\left(\iint_{\mathbb{R}^{2 N}} \frac{C_{N, s, p}|u_{\gamma, \lambda}^+(x)-u_{\gamma, \lambda}^+(y)|^{p-2}(u_{\gamma, \lambda}^+(x)-u_{\gamma, \lambda}^+(y))(u_{\gamma, \lambda}^{-}(x)-u_{\gamma, \lambda}^{-}(y))}{|x-y|^{N+s p}} d x d y\right) d \mu^{-}(s),
       \end{align}
      using the definition of $g_+$ with the fact the $u_{\gamma, \lambda}^{-}=0$ when $u_{\gamma, \lambda} \geq 0.$
      Now, we use the inequality $(a-b)(a^--b^-) = -(a^- -b^-)^2-(a^+b^-+a^-b^+)$ in the first integral of \eqref{ruz5.5} and $(a^+-b^+)(a^--b^-)=-(a^+b^-+b^+a^-)$ in the second integral of \eqref{ruz5.5}, \eqref{eq3.19} and $|a^{\pm}-b^{\pm}|  \leq |a-b|$ to get
      \begin{align*}
          0=&\langle  ( \mathcal{J}_{\gamma, \lambda}^+)^{'} (u_{\gamma, \lambda}), u_{\gamma, \lambda}^{-} \rangle  \\&=  -\int_{[0,1]}\left(\iint_{\mathbb{R}^{2 N}}  \frac{C_{N, s, p}|u_{\gamma, \lambda}(x)-u_{\gamma, \lambda}(y)|^{p-2}(u_{\gamma, \lambda}^{-}(x)-u_{\gamma, \lambda}^{-}(y))^2}{|x-y|^{N+s p}} d x d y\right) d \mu^{+}(s) \nonumber\\& -\int_{[0,1]}\left(\iint_{\mathbb{R}^{2 N}}  \frac{C_{N, s, p}|u_{\gamma, \lambda}(x)-u_{\gamma, \lambda}(y)|^{p-2}(u_{\gamma, \lambda}^+(x)u_{\gamma, \lambda}^{-}(y)+u^+_{\gamma, \lambda}(y) u_{\gamma, \lambda}^{-}(x))}{|x-y|^{N+s p}} d x d y\right) d \mu^{+}(s) \nonumber  \\
& +\int_{[0, \bar{s}]}\left(\iint_{\mathbb{R}^{2 N}} \frac{C_{N, s, p}|u_{\gamma, \lambda}^+(x)-u_{\gamma, \lambda}^+(y)|^{p-2}(u_{\gamma, \lambda}^+(x)u_{\gamma, \lambda}^{-}(y)+u^+_{\gamma, \lambda}(y) u_{\gamma, \lambda}^{-}(x))}{|x-y|^{N+s p}} d x d y\right) d \mu^{-}(s)\nonumber \\&\leq   -\int_{[0,1]}\left(\iint_{\mathbb{R}^{2 N}}  \frac{C_{N, s, p}|u_{\gamma, \lambda}(x)-u_{\gamma, \lambda}(y)|^{p-2}(u_{\gamma, \lambda}^{-}(x)-u_{\gamma, \lambda}^{-}(y))^2}{|x-y|^{N+s p}} d x d y\right) d \mu^{+}(s) \nonumber\\& -\int_{[0,1]}\left(\iint_{\mathbb{R}^{2 N}}  \frac{C_{N, s, p}|u_{\gamma, \lambda}(x)-u_{\gamma, \lambda}(y)|^{p-2}(u_{\gamma, \lambda}^+(x)u_{\gamma, \lambda}^{-}(y)+u^+_{\gamma, \lambda}(y) u_{\gamma, \lambda}^{-}(x))}{|x-y|^{N+s p}} d x d y\right) d \mu^{+}(s) \nonumber \\
& +\int_{[0, 1]}\left(\iint_{\mathbb{R}^{2 N}} \frac{C_{N, s, p}|u_{\gamma, \lambda}(x)-u_{\gamma, \lambda}(y)|^{p-2}(u_{\gamma, \lambda}^+(x)u_{\gamma, \lambda}^{-}(y)+u^+_{\gamma, \lambda}(y) u_{\gamma, \lambda}^{-}(x))}{|x-y|^{N+s p}} d x d y\right) d \mu^{+}(s)\nonumber \\&=  -\int_{[0,1]}\left(\iint_{\mathbb{R}^{2 N}}  \frac{C_{N, s, p}|u_{\gamma, \lambda}(x)-u_{\gamma, \lambda}(y)|^{p-2}(u_{\gamma, \lambda}^{-}(x)-u_{\gamma, \lambda}^{-}(y))^2}{|x-y|^{N+s p}} d x d y\right) d \mu^{+}(s) \nonumber\\& \leq  -\int_{[0,1]}\left(\iint_{\mathbb{R}^{2 N}}  \frac{C_{N, s, p}|u_{\gamma, \lambda}^-(x)-u_{\gamma, \lambda}^-(y)|^{p-2}(u_{\gamma, \lambda}^{-}(x)-u_{\gamma, \lambda}^{-}(y))^2}{|x-y|^{N+s p}} d x d y\right) d \mu^{+}(s) \nonumber\\&=
          - \|u_{\gamma, \lambda}^{-}\|_{X_p(\Omega)}^p,
      \end{align*}
      \noindent which implies that 
\begin{align}
     \|u_{\gamma, \lambda}^{-}\|_{X_p(\Omega)} = 0,
\end{align} which further shows that $u_{\gamma, \lambda}^{-}(x)=0$ for a.e. $x \in \Omega$ leading to the conclusion that $u_{\gamma, \lambda}(x)\geq 0$ for a.e. $x \in \Omega.$  
We will now prove that $u_{\gamma, \lambda} \not\equiv 0$ in $\Omega$. Consider, $w\in C_c^\infty(\Omega)$ such that $w\geq 0$ and $w\not\equiv 0$. Since, $\liminf\limits_{t\rightarrow 0^+}\frac{G(x,t)}{t^p}=+\infty$ from \eqref{growth1-1}, then for all $M>0$ there exists a $\delta>0$ such that 
$$G(x,t)\geq M t^p,~\forall\, t\in(0,\delta).$$
Thus choosing $t\in(0,\frac{\delta}{\sup_\Omega w})$, we get
\begin{align*}
      \mathcal{J}_{\gamma, \lambda}^+ (tw)&=\frac{1}{p}\left[\rho_{p}(tw)\right]^{p}-\frac{1}{p} \int_{[0, \bar{s}]}[tw^+]_{s, p}^{p} d \mu^{-}(s)-\frac{\gamma}{p_{s_\sharp}^*}\int_\Omega|tw|^{p_{s_\sharp}^*}dx-\lambda \int_\Omega G^+(x,tw)dx\\
    &\leq \frac{1}{p}\left[\rho_{p}(tw)\right]^{p}-\frac{1}{p} \int_{[0, \bar{s}]}[tw^+]_{s, p}^{p} d \mu^{-}(s)-\lambda M \|w\|^p_{L^p(\Omega)}t^p-\frac{\gamma}{p_{s_\sharp}^*}t^{p_{s_\sharp}^*}\int_\Omega|w|^{p_{s_\sharp}^*}dx\\
    &<0~\text{for sufficiently large } M>0.
\end{align*}
Therefore, $0$ cannot be a local minimum of $\mathcal{J}_{\gamma, \lambda}^+$ and hence $u_{\gamma, \lambda}(x)$ is nontrivial, nonnegative weak solution to the problem \eqref{pro12}.

For the existence of a nonpositive weak solution, we work with the functional $\mathcal{J}_{\mu, \lambda}^-:X_p(\Omega) \rightarrow \mathbb{R}$ as 

    \begin{equation}
        \mathcal{J}_{\gamma, \lambda}^-(u):= \frac{1}{p}\left[\rho_{p}(u)\right]^{p}-\frac{1}{p} \int_{[0, \bar{s}]}[u^-]_{s, p}^{p} d \mu^{-}(s) -\frac{\gamma}{p_{s_\sharp}^*} \int_{\Omega} (u(x)^-)^{p_{s_\sharp}^*} d x-\lambda \int_{\Omega} G_-(x, u(x)) d x,
    \end{equation} for every $u \in X_p(\Omega).$ Here the function $G_-$ is given by
    $$G_-(x, t):=\int_0^t g_-(x, \tau)\, d \tau,$$
    where $$g_-(x, \tau):=\begin{cases}
        g(x, \tau) &\text{if}\quad \tau <0 \\
        0&\text{if} \quad \tau \geq 0,
    \end{cases}$$ for almost everywhere $x \in \Omega$ and $\tau \in \mathbb{R}.$ Following the same approach as for nonnegative solutions and applying Theorem \ref{main result wlsc 0}, we deduce that for every $\gamma>0$ and sufficiently small $\lambda$, problem \eqref{pro12} admits a nontrivial, nonpositive weak solution. This completes the proof of the theorem.
      \end{proof}
    
\begin{remark}
   It is important to note that condition \eqref{growth1-1} is essential for ensuring the existence of a nontrivial solution to \eqref{pro12}. Consequently, any equivalent condition to \eqref{growth1-1} that also satisfies \eqref{growth1} can be assumed to guarantee the existence of a nontrivial solution.
\end{remark} 
     Before proving Theorem \ref{thmnonmain}, we observe that the proof of Theorem \ref{main result wlsc 0} remains valid even when considering the nonlocal problem with a lower-order term $g: \Omega \times \mathbb{R} \rightarrow \mathbb{R}$ such that 
      \begin{equation}
          |g(x, t)| \leq a_1+a_2 |t|^{m-1}+a_3 |t|^{q-1}\quad \text{a. e.}\,\,\,x \in \Omega,\, \forall t \in \mathbb{R},
      \end{equation}
      for some positive constants $a_j, j=1,2,3$ and $1<m<p \leq q <p_{s_\sharp}^*.$ In fact, we have the following result: 
      \begin{thm} \label{thm1.1dupl}
    Let $ \Omega$ be a bounded subset of $\mathbb{R}^N$. Let $\mu=\mu^{+}-\mu^{-}$ with $\mu^{+}$and $\mu^{-}$ satisfying \eqref{measure 1}-\eqref{measure 3} and \eqref{extracondi}, and let $s_{\sharp}$ be as in \eqref{measure 4}. Let $\frac{N}{s_{\sharp}}>p \geq 2$. Suppose that $m$ and $q$ are two real constants such that 
\begin{equation} \label{growndpm}
    1 \leq m<p\leq q<p_{s_\sharp}^*=\frac{p N}{N-{s_\sharp}p}.
\end{equation} Let $g:\Omega \times \mathbb{R} \rightarrow \mathbb{R}$ be a Carath\'eodory function such that 
    \begin{equation} \label{growthdupl}
       |g(x, t)| \leq a_1+a_2 |t|^{m-1}+a_3 |t|^{q-1}\quad \text{a.e. }\,x \in \Omega,\, \forall\, t \in \mathbb{R},
    \end{equation}
    for some $a_1, a_2, a_3>0.$ Then, for every $\gamma>0,$ there exists   $\Sigma_\gamma>0$ such that the following nonlocal problem
    \begin{equation} \label{pro1dump}
       \begin{cases} 
A_{\mu, p} u= \gamma |u|^{p_s^*-2}u+\lambda g(x, u) \quad &\text{in}\quad \Omega, \\
u=0\quad & \text{in}\quad \mathbb{R}^N\backslash \Omega,
\end{cases}
    \end{equation} admits at least one  weak solution on $u_\lambda \in X^{s,p}_0(\Omega)$ for every $0<\lambda < \Sigma_\gamma,$ which is local minimum of the energy functional 
    \begin{equation} \label{endupl}
        \mathcal{J}_{\gamma, \lambda}(u):=\frac{1}{p}\left[\rho_{p}(u)\right]^{p}-\frac{1}{p} \int_{[0, \bar{s}]}[u]_{s, p}^{p} d \mu^{-}(s)-\lambda \int_{\Omega} \int_{0}^{u(x)} g(x, \tau) d \tau dx-\frac{\gamma}{p_{s_{\sharp}}^{*}} \int_{\Omega}|u|^{p_{s_\sharp}^{*}} d x,
    \end{equation} for every $u \in X_p(\Omega).$
\end{thm}

\begin{proof} The proof of this theorem closely follows that of Theorem \ref{main result wlsc 0}. However, for the sake of completeness, we provide the detailed proof.
    For  a fixed $\gamma>0,$ let us define the following  function $h_\gamma:(0, \infty) \rightarrow \mathbb{R}$ given by
\begin{equation} \label{rationfun2}
    h_{\gamma}(r):=\frac{r^{p-1}-\gamma C_{p_{s_\sharp}^{*}}^{p_{s_\sharp}^{*}} r^{p_{s_\sharp}^{*}-1}}{a_{1} C_{p_{s_\sharp}^{*}}|\Omega|^{(p_{s_\sharp}^{*}-1) / p_{s_\sharp}^{*}}+a_2C_{p_{s_\sharp}^{*}}^{m}|\Omega|^{(p_{s_\sharp}^{*}-m) / p_{s_\sharp}^{*}} r^{m-1}+a_{3} C_{p_{s_\sharp}^{*}}^{q}|\Omega|^{(p_{s_\sharp}^{*}-q) / p_{s_\sharp}^{*}} r^{q-1}}, \quad \forall r \geq 0,
\end{equation}
where $C_{p_{s_\sharp}^{*}}$ is the best constant in the fractional Sobolev inequality given by \eqref{Sobolev constant 2}. Since $p_s^*>q \geq p>m \geq 1,$
we note that $\lim_{r \rightarrow 0} h_\gamma(r)=0$  and  $\lim_{r \rightarrow \infty} h_\gamma(r)=-\infty.$ Arguing similarly to the proof of Theorem \ref{main result wlsc 0}, we deduce that $h_\gamma$ has a global maximum, that is,  there exists a  $r_{\gamma, \max}$ such that 
$$h_\gamma(r_{\gamma, \max})=\max_{r>0} h_\gamma(r).$$

Set $$r_{0, \gamma}:=\min\{r_{\gamma, \max}, \bar{r}_\gamma\},$$ where $\bar{r}_\gamma$ is defined by Proposition \ref{main prop on wlsc}. By taking any $\lambda < \Sigma_\gamma:= g_\gamma(r_{0, \gamma})$ we can find $r_{0, \gamma, \lambda} \in (0, r_{0, \gamma})$ such that 
\begin{equation} \label{S42a}
    \lambda< \frac{r_{0, \gamma, \lambda}^{p-1}-\gamma C_{p_{s_\sharp}^{*}}^{p_{s_\sharp}^{*}} r_{0, \gamma, \lambda}^{p_{s_\sharp}^{*}-1}}{a_{1} C_{p_{s_\sharp}^{*}}|\Omega|^{(p_{s_\sharp}^{*}-1) / p_{s_\sharp}^{*}}+a_2C_{p_{s_\sharp}^{*}}^{m}|\Omega|^{(p_{s_\sharp}^{*}-m) / p_{s_\sharp}^{*}} r_{0, \gamma, \lambda}^{m-1}+a_{3} C_{p_{s_\sharp}^{*}}^{q}|\Omega|^{(p_{s_\sharp}^{*}-q) / p_{s_\sharp}^{*}} r_{0, \gamma, \lambda}^{q-1}}.
\end{equation}

Now, we take $0<\epsilon< r_{0, \gamma, \lambda}$ and define
\begin{equation*}
    \Lambda_{\lambda, \gamma}\left(\varepsilon, r_{0, \gamma, \lambda}\right):=\frac{\sup _{v \in \Phi_{p, s}^{-1}([0, r_{0, \gamma, \lambda}])} \Psi_{\gamma, \lambda}(v)-\sup _{v \in \Phi_{p, s}^{-1}([0, r_{0, \gamma, \lambda}-\varepsilon])} \Psi_{\gamma, \lambda}(v)}{\varepsilon}.
\end{equation*}
Then, by rescaling $v$ and using \eqref{int of g} and \eqref{part 2 of decomp}, we obtain 
\begin{align*}
    \Lambda_{\lambda, \gamma}\left(\varepsilon, r_{0, \gamma, \lambda}\right)& \leq \frac{1}{\varepsilon}\left|\sup _{v \in \Phi_{p, s}^{-1}([0, r_{0, \gamma, \lambda}])} \Psi_{\gamma, \lambda}(v)-\sup _{v \in \Phi_{p, s}^{-1}([0, r_{0, \gamma, \lambda}-\varepsilon])} \Psi_{\gamma, \lambda}(v)\right|\\
    &\leq \sup _{v \in \Phi_{p, s}^{-1}([0,1])} \int_{\Omega}\left|\int_{\left(r_{0, \gamma, \lambda}-\varepsilon\right) v(x)}^{r_{0, \gamma, \lambda} v(x)} \frac{\left|f_{\gamma, \lambda}(x, t)\right|}{\varepsilon} d t\right| d x,
\end{align*}
where $f_{\lambda, \gamma}(x, t):=\gamma |t|^{p_s^*-1}+\lambda g(x, t)$ for a.e. $x \in \Omega$ and for all $t \in \mathbb{R}.$ 
Next, using the subcritical growth condition \eqref{growndpm} of $g$ along with the subelliptic continuous embedding $X_p(\Omega) \hookrightarrow L^{p_{s_\sharp}^*}(\Omega)$ and the H\"older inequality, we get
\begin{align*}
    \Lambda_{\lambda, \gamma}\left(\varepsilon, r_{0, \gamma, \lambda}\right)&\leq \sup _{v \in \Phi_{p, s}^{-1}([0,1])} \int_{\Omega}\left|\int_{\left(r_{0, \gamma, \lambda}-\varepsilon\right) v(x)}^{r_{0, \gamma, \lambda} v(x)} \frac{\left|f_{\gamma, \lambda}(x, t)\right|}{\varepsilon} d t\right| d x\\& \leq \frac{1}{\epsilon} \sup _{v \in \Phi_{p, s}^{-1}([0,1])} \int_{\Omega}\Big|\int_{\left(r_{0, \gamma, \lambda}-\varepsilon\right) v(x)}^{r_{0, \gamma, \lambda} v(x)} \left[ \gamma |t|^{p_{s_\sharp}^*-1} +\lambda (a_1+ a_2|t|^{m-1}+a_3|t|^{q-1})) \right] d t \Big| d x \\&\leq \frac{\gamma}{p_{s_\sharp}^*} \sup _{v \in \Phi_{p, s}^{-1}([0,1])} \|v\|_{L^{p_{s_\sharp}^*}(\Omega)}^{p_{s_\sharp}^*} \left( \frac{r_{0,\gamma, \lambda}^{p_{s_\sharp}^*}-(r_{0,\gamma, \lambda}-\epsilon)^{p_{s_\sharp}^*}}{\epsilon} \right)+ \lambda \sup _{v \in \Phi_{p, s}^{-1}([0,1])} \Bigg( a_1  \|v\|_{L^{1}(\Omega)}  \\&+\frac{a_2}{m}  \|v\|_{L^{m}(\Omega)}^m \left( \frac{r_{0,\gamma, \lambda}^{m}-(r_{0,\gamma, \lambda}-\epsilon)^{m}}{\epsilon} \right)+\frac{a_3}{q}  \|v\|_{L^{q}(\Omega)}^q \left( \frac{r_{0,\gamma, \lambda}^{q}-(r_{0,\gamma, \lambda}-\epsilon)^{q}}{\epsilon} \right)\Bigg) \\&\leq \frac{\gamma C_{p_{s_\sharp}^*}^{p_{s_\sharp}^*}}{p_{s_\sharp}^*} \left( \frac{r_{0,\gamma, \lambda}^{p_{s_\sharp}^*}-(r_{0,\gamma, \lambda}-\epsilon)^{p_{s_\sharp}^*}}{\epsilon} \right)\\&+ \lambda \Bigg( a_1 C_{p_{s_\sharp}^*} |\Omega|^{\frac{p_{s_\sharp}^*-1}{p_{s_\sharp}^*}}+a_2 \frac{C_{p_{_\sharp}^*}^m}{m} \left( \frac{r_{0,\gamma, \lambda}^{m}-(r_{0,\gamma, \lambda}-\epsilon)^{m}}{\epsilon} \right)|\Omega|^{\frac{p_{s_\sharp}^*-m}{p_{s_\sharp}^*}}\\&\quad \quad\quad\quad+a_3 \frac{C_{p_{s_\sharp}^*}^q}{q} \left( \frac{r_{0,\gamma, \lambda}^{q}-(r_{0,\gamma, \lambda}-\epsilon)^{q}}{\epsilon} \right)|\Omega|^{\frac{p_{s_\sharp}^*-q}{p_{s_\sharp}^*}}\Bigg),
\end{align*} where $C_{p_{s_\sharp}^*}$ is the best constant in the embedding $X_p(\Omega) \hookrightarrow L^{p_{s_\sharp}^*}(\Omega).$ 

Therefore, 
\begin{align} \label{S44a}
 \nonumber   \limsup_{\epsilon \rightarrow 0} \Lambda_{\lambda, \gamma}(\epsilon, r_{0,\gamma, \lambda}) &\leq \gamma C_{p_{s_\sharp}^*}^{p_{s_\sharp}^*} r_{0,\gamma, \lambda}^{p_{s_\sharp}^*-1} + \lambda \Bigg( a_1 C_{p_{s_\sharp}^*} |\Omega|^{\frac{p_{s_{\sharp}}^*-1}{p_{s_\sharp}^*}}+a_2 C_{p_{s_\sharp}^*}^{m} r_{0,\gamma, \lambda}^{m-1}|\Omega|^{\frac{p_{s_\sharp}^*-m}{p_{s_\sharp}^*}}\nonumber\\&\quad\quad\quad\quad+a_3 C_{p_{s_\sharp}^*}^q r_{0,\gamma, \lambda}^{q-1}|\Omega|^{\frac{p_{s_\sharp}^*-q}{p_{s_\sharp}^*}}\Bigg)
    < r_{0, \gamma, \lambda}^{p-1},
\end{align} 
by \eqref{S42a}. 

Combining \eqref{S44a} with Lemma \ref{tech lem 1} and Lemma \ref{tech lem 2}, we deduce that 
\begin{equation*}
    \inf _{u \in \Phi_{p, s}^{-1}([0, r_{0, \gamma, \lambda}))} \frac{\sup _{v \in \Phi_{p, s}^{-1}([0, r_{0, \gamma, \lambda}])} \Psi_{\gamma, \lambda}(v)-\Psi_{\gamma, \lambda}(u)}{r_{0, \gamma, \lambda}^{p}-\Phi_{p, s}(u)^{p}}<\frac{1}{p}.
\end{equation*}
Therefore, from  Lemma \ref{tech lem 3} there exists $w_{\gamma, \lambda} \in \Phi_{p, s}^{-1}([0, r_{0, \gamma, \lambda}))$ such that
\begin{equation}\label{P323}
    \mathcal{J}_{\gamma, \lambda}\left(w_{\gamma, \lambda}\right)=\frac{1}{p}[\rho_{p}(w_{\gamma, \lambda})]^{p}-\Psi_{\gamma, \lambda}\left(w_{\gamma, \lambda}\right)<\frac{r_{0, \gamma, \lambda}^{p}}{p}-\Psi_{\gamma, \lambda}(u),
\end{equation}
for every $u \in \Phi_{p, s}^{-1}([0, r_{0, \gamma, \lambda}])$. Since $r_{0, \gamma, \lambda}<\bar{r}_{\gamma}$, and recalling that 
\begin{equation*} \mathcal{J}_{\gamma, \lambda}(u) = \mathcal{L}_{\gamma}(u) - \lambda \int_{\Omega} G(x, u(x)) dx, \end{equation*} 
it follows from Proposition \ref{main prop on wlsc} that the energy functional $\mathcal{J}_{\gamma, \lambda}$ is sequentially weakly lower semicontinuous on  $\Phi_{p, s}^{-1}([0, r_{0, \gamma, \lambda}))=\overline{B_{X_p(\Omega)}(0, r_{0, \gamma, \lambda})}.$ Thus, the restriction $\mathcal{J}_{\gamma, \lambda}|_{\overline{B_{X_p(\Omega)}(0, r_{0, \gamma, \lambda})}}$ has a global minimum $u_{0, \gamma, \lambda}$ in $\overline{B_{X_p(\Omega)}(0, r_{0, \gamma, \lambda})}.$ Now, we claim that $u_{0, \gamma, \lambda} \in B_{X_p(\Omega)}(0, r_{0, \gamma, \lambda}).$ To prove this, suppose that $\Phi_{p, s}(u_{0, \gamma, \lambda}):=\|u_{0, \gamma, \lambda}\|_{X_p(\Omega)}=r_{0, \gamma, \lambda}$. Then by \eqref{P323} we have 
\begin{equation*}
    \mathcal{J}_{\gamma, \lambda}\left(u_{0, \gamma, \lambda}\right)=\frac{r_{0, \gamma, \lambda}^{p}}{p}-\Psi_{\gamma, \lambda}\left(u_{0, \gamma, \lambda}\right)>\mathcal{J}_{\gamma, \lambda}\left(w_{\gamma, \lambda}\right),
\end{equation*}
contradicting the fact that $u_{0, \gamma, \lambda}$ is a global minimum for $\mathcal{J}_{\gamma, \lambda}|_{\overline{B_{X_p(\Omega)}(0, r_{0, \gamma, \lambda})}}.$ Therefore, we conclude that $u_{0, \gamma,\lambda} \in X_p(\Omega)$ is a local minimum for the energy functional $\mathcal{J}_{\gamma, \lambda}$ with 
$$\Phi_{p, s}(u_{0, \gamma, \lambda}):=\|u_{0, \gamma, \lambda}\|_{X_p(\Omega)}<r_{0, \gamma, \lambda}$$ and 
hence, a weak solution to the problem \eqref{pro1dump}. This completes the proof.
\end{proof}
\begin{thm} 
    Let $ \Omega$ be a bounded subset of $\mathbb{R}^N$. Let $\mu=\mu^{+}-\mu^{-}$ with $\mu^{+}$and $\mu^{-}$ satisfying \eqref{measure 1}-\eqref{measure 3} and \eqref{extracondi}, and let $s_{\sharp}$ be as in \eqref{measure 4}. Let $\frac{N}{s_{\sharp}}>p \geq 2$. Suppose that $m$ and $q$ are two real constants such that 
$$1 \leq m<p\leq q<p_{s_\sharp}^*=\frac{p N}{N-{s_\sharp}p}.$$
Then, there exists an open interval $\Lambda \subset (0, +\infty)$ such that, for every $\lambda \in \Lambda,$ the  nonlocal problem 
\begin{equation} \label{pro15.7intro0}
       \begin{cases} 
A_{\mu, p} u=  |u|^{p_{s_\sharp}^*-2}u+\lambda (|u|^{m-1}+|u|^{q-1}) \quad &\text{in}\quad \Omega, \\
u=0\quad & \text{in}\quad \mathbb{R}^N\backslash \Omega,
\end{cases}
    \end{equation} has at least one nonnegative nontrivial weak solution $u_\lambda \in X^{p}(\Omega),$  which is a  local minimum of the energy functional 
     \begin{equation}
        \mathcal{J}_{ \lambda}(u):= \frac{1}{p}\left[\rho_{p}(u)\right]^{p}-\frac{1}{p} \int_{[0, \bar{s}]}[u]_{s, p}^{p} d \mu^{-}(s) -\frac{1}{p_{s_\sharp}^*} \int_{\Omega} |u(x)|^{p_{s_\sharp}^*} d x-\frac{\lambda}{m} \int_{\Omega} |u(x)|^m  d x-\frac{\lambda}{q} \int_{\Omega} |u(x)|^q  d x.
    \end{equation} 
\end{thm}
\begin{proof} We first observe that the statement and proof of Theorem \ref{subthm} remain valid if $g:\Omega \times \mathbb{R} \rightarrow \mathbb{R}$ is a Carath\'eodory function such that 
    \begin{equation} \label{growthdupl1}
       |g(x, t)| \leq a_1+a_2 |t|^{m-1}+a_3 |t|^{q-1}\quad \text{a.e.}\,\,\,x \in \Omega,\, \forall t \in \mathbb{R},
    \end{equation}
    for some $a_1, a_2, a_3>0.$   In our case we take $g(x, t)=|t|^{m-1}+|t|^{q-1},$ which satisfies  $g(x, 0)=0$ for a.e. $x \in \Omega$ making  $u \equiv 0$ a  trivial solution of \eqref{pro15.7intro0}. The proof of this theorem follows from the fact that if zero is not a local minimum of the energy functional  $\mathcal{J}_{\lambda},$ then any weak solution of problem \eqref{pro15.7intro0} obtained via Theorem \ref{thm1.1dupl} is necessarily nontrivial, as a consequence of Theorem \ref{subthm}.
   In particular, we note that zero is not a local minimum   for the  functional $\mathcal{J}_{ \lambda}^+$ given by 
 \begin{equation}
        \mathcal{J}_{\lambda}^+(u):= \frac{1}{p}\left[\rho_{p}(u)\right]^{p}-\frac{1}{p} \int_{[0, \bar{s}]}[u^+]_{s, p}^{p} d \mu^{-}(s) -\frac{1}{p_{s_\sharp}^*} \int_{\Omega} (u(x)^+)^{p_{s_\sharp}^*} d x-\lambda \int_{\Omega} G_+(x, u(x)) d x,
    \end{equation} for every $u \in X_p(\Omega),$ where $$G_+(x, u):= \begin{cases}
         \frac{u^m}{m}+\frac{u^q}{q}\quad \text{if} \quad u>0, \\
         0 \quad \text{if} \quad u \leq 0.
    \end{cases}$$
    To deduce that zero is not a local minimum, we fix a nonnegative function $u_0 \in X^{s, p}_0(\Omega) \backslash \{0\}.$ Then, for a small enough $\tau>0,$ we have 
   \begin{align*}
         \mathcal{J}_{\lambda}^+(\tau u_0)&= \frac{\tau^p}{p} \left[\rho_{p}(u)\right]^{p}-\frac{\tau^p}{p} \int_{[0, \bar{s}]}[u^+]_{s, p}^{p} d \mu^{-}(s)\\& -\frac{\tau^{p_{s_\sharp}^*}}{p_{s_\sharp}^*} \int_{\Omega} (u_0(x)^+)^{p_{s_\sharp}^*} d x- \frac{\lambda \tau^m}{m} \int_{\Omega} u_0(x)^m\, d x-\frac{\lambda \tau^q}{q} \int_{\Omega} u_0(x)^q\, d x<0,
        \end{align*}
     as $1 \leq m<p\leq q<p_s^*.$  This completes the proof of the theorem. \end{proof}

\section{Nonlocal subcritical problem associated with the nonlinear superposition\\ operators of mixed order} \label{sec6}

In this section, we consider the nonlocal subcritical problem associated with the nonlinear superposition operators $A_{\mu, p}$ of mixed order and apply the methodologies developed for the critical problems in the previous section. The proofs are considerably simpler compared to the critical case, owing to the subcritical nature of the problem. Specifically, we will investigate the nonlocal problem
\begin{equation}\label{problem 1.1}
\begin{cases}A_{\mu,p} u=\lambda g(x,u) & \text { in } \quad \Omega, \\ 
u=0 & \text { in } \mathbb{R}^{N} \backslash \Omega, \end{cases}
\end{equation}
where $\lambda>0$ is a real parameter and $g$ is a subcritical nonlinearity satisfying
\begin{align}\label{cond on g1}
\begin{split}
&\text{ there exist } a_{1}, a_{2}>0 \text{ and } q \in[1, p_{s_{\sharp}}^{*}), \text{ such that }\\
&|g(x, t)| \leq a_{1}+a_{2}|t|^{q-1} \text {, a.e. } x \in \Omega, \forall t \in \mathbb{R} . 
\end{split}
\end{align}
 Note that the problem \eqref{problem 1.1} is the problem \eqref{problem 1} with $\gamma=0$, that is, without critical nonlinearity. 

The energy functional $\mathcal{E}_{\lambda}: X_{p}(\Omega) \rightarrow \mathbb{R}$ associated to the problem \eqref{problem 1.1} is defined by
\begin{equation}\label{functional for problem 1.1}
\mathcal{E}_{\lambda}(u):=\frac{1}{p}\left[\rho_{p}(u)\right]^{p}-\frac{1}{p} \int_{[0, \bar{s}]}[u]_{s, p}^{p} d \mu^{-}(s)-\lambda \int_{\Omega}G(x, u) d x, 
\end{equation}
for every $u \in X_{p}(\Omega)$, where $G$ is given by \eqref{int of g}.
As a consequence of our assumption \eqref{cond on g} regarding the growth of $g$, the functional $\mathcal{E}_{\lambda}$ belongs to $C^{1}\left(X_{p}(\Omega), \mathbb{R}\right)$ and its derivative at $u \in X_{p}(\Omega)$ is expressed as
\begin{align*}
\left\langle\mathcal{E}_{\lambda}^{\prime}(u), v\right\rangle &= \int_{[0,1]}\left(\iint_{\mathbb{R}^{2 N}}  \frac{C_{N, s, p}|u(x)-u(y)|^{p-2}(u(x)-u(y))(v(x)-v(y))}{|x-y|^{N+s p}} d x d y\right) d \mu^{+}(s) \\
&-\int_{[0,\bar{s}]}\left(\iint_{\mathbb{R}^{2 N}}  \frac{C_{N, s, p}|u(x)-u(y)|^{p-2}(u(x)-u(y))(v(x)-v(y))}{|x-y|^{N+s p}} d x d y\right) d \mu^{-}(s) \\
& -\lambda \int_{\Omega}g(x, u) v d x
\end{align*}
for every $v \in X_{p}(\Omega)$. 
\begin{definition}
    A weak solution to the problem \eqref{problem 1.1} is a function $u\in X_{p}(\Omega)$ such that
\begin{align*}
& \int_{[0,1]}\left(\iint_{\mathbb{R}^{2 N}}  \frac{C_{N, s, p}|u(x)-u(y)|^{p-2}(u(x)-u(y))(v(x)-v(y))}{|x-y|^{N+s p}} d x d y\right) d \mu^{+}(s) \\
&=\int_{[0,\bar{s}]}\left(\iint_{\mathbb{R}^{2 N}}  \frac{C_{N, s, p}|u(x)-u(y)|^{p-2}(u(x)-u(y))(v(x)-v(y))}{|x-y|^{N+s p}} d x d y\right) d \mu^{-}(s)\\
& +\lambda \int_{\Omega}g(x, u) v d x
\end{align*}
for all $v \in X_{p}(\Omega)$.
\end{definition}
Observe that the weak solutions of problem \eqref{problem 1.1} correspond to the critical points of the functional $\mathcal{E}_{\lambda}$ defined in \eqref{functional for problem 1.1}.

Analogous to the previous section, for a fixed $\lambda>0$ we define the functional
\begin{equation*}
    \Phi_{p, s}(u):=\rho_p(u)=\left(\int_{[0,1]}\left(\iint_{\mathbb{R}^{2 N}} \frac{C_{N, s, p}|u(x)-u(y)|^{p}}{|x-y|^{N+p s}} d x d y\right) d \mu^{+}(s)\right)^{1 / p}
\end{equation*}
and
\begin{equation}\label{part 2 of decomposition 1}
    \Psi_{ \lambda}(u):=\frac{1}{p} \int_{[0, \bar{s}]}[u]_{s, p}^{p} d \mu^{-}(s)+\lambda \int_{\Omega} G(x, u) d x,
\end{equation}
for every $u \in X_{ p}(\Omega)$, where the potential $G$ is given by \begin{equation}\label{int of g1}
G(x, t):=\int_{0}^{t} g(x, \tau) d \tau, \quad \forall(x, t) \in \Omega \times \mathbb{R}. 
\end{equation} Consequently, we have
\begin{equation}\label{functional as a sum 1}
    \mathcal{E}_{ \lambda}(u)=\frac{1}{p}\Phi_{p, s}(u)^p-\Psi_{\lambda}(u),
\end{equation}
for every $u \in X_{ p}(\Omega)$. Next, we set
\begin{equation*}
  \Theta_{\lambda}(\eta, \zeta):=\sup _{v \in \Phi_{p, s}^{-1}([0, \eta])} \Psi_{ \lambda}(v)-\sup _{v \in \Phi_{p, s}^{-1}([0, \eta-\zeta])} \Psi_{ \lambda}(v),
\end{equation*}
for $0<\zeta<\eta$.

Note that, due to condition \eqref{cond on g1}, the functional $\Psi_{ \lambda}$ is well-defined and sequentially weakly (upper) continuous. As a result, the functional $\mathcal{E}_{\lambda}$ is sequentially weakly lower semicontinuous on $X_{p}(\Omega)$. Moreover, the functional $\mathcal{E}_{\lambda}$ is also coercive as a consequence of the continuous embedding \eqref{embedding}, the identity \eqref{eqicon46} and the fact that $1<q<p$. 

The following two lemmas will play a crucial role in the subsequent analysis. We state them here without proofs since they can be proved in the same way as Lemma \ref{tech lem 1} and Lemma \ref{tech lem 2}, respectively.

\begin{lem}\label{tech lem 1.1}
    Let $p\in(1, +\infty)$ and $ \lambda>0$. Suppose that
\begin{equation}\label{cond 1.1.1}
\limsup _{\varepsilon \rightarrow 0^{+}} \frac{\Theta_{\lambda}\left(r_{0}, \varepsilon\right)}{\varepsilon}<r^{p-1}_{0} 
\end{equation}
for some $r_{0}>0$. Then
\begin{equation}\label{cond 1.2.2}
\inf _{\sigma<r_{0}} \frac{\Theta_{ \lambda}\left(r_{0}, r_{0}-\sigma\right)}{r_{0}^{p}-\sigma^{p}}<\frac{1}{p}. 
\end{equation}
\end{lem}

\begin{lem}\label{tech lem 2.1}
Let $p\in(1, +\infty)$ and $\lambda>0$. Suppose \eqref{cond 1.2.2} holds for some $r_{0}>0$. Then we have
\begin{equation}\label{cond 2.1.1}
\inf _{u \in \Phi_{p, s}^{-1}\left(\left[0, r_{0}\right)\right)} \frac{\sup _{v \in \Phi_{p, s}^{-1}\left(\left[0, r_{0}\right]\right)} \Psi_{\lambda}(v)-\Psi_{ \lambda}(u)}{r_{0}^{p}-\Phi_{p, s}(u)^{p}}<\frac{1}{p}. 
\end{equation}
\end{lem} 

The main result of this section is stated below.

\begin{thm}\label{main result subcrit}
    Let $\Omega$ be a bounded open subset of $\mathbb{R}^{N}$. Let $\mu=\mu^{+}-\mu^{-}$with $\mu^{+}$and $\mu^{-}$satisfying \eqref{measure 1}-\eqref{measure 3} and let $s_{\sharp}$ be as in \eqref{measure 4}. Suppose that $1<p<q<p_{s_{\sharp}}^{*}$ and $s_{\sharp}p<N$, where $p_{s_{\sharp}}^{*}:=\frac{N p}{N-s_{\sharp}p}$. Let $g: \Omega \times \mathbb{R} \rightarrow \mathbb{R}$ be a Carath\'eodory function satisfying the subcritical growth condition \eqref{cond on g1}. Furthermore, let
$$
0<\lambda<\frac{(q-p)^{\frac{q-p}{q-1}}(p-1)^{\frac{p-1}{q-1}}}{(a_{1} C_1)^{\frac{q-p}{q-1}}(a_{2} C_{2})^{\frac{p-1}{q-1}}(q-1) C_{p_{s_{\sharp}}^{*}}^{p}|\Omega|^{\frac{p_{s_{\sharp}}^{*}-q}{p_{s_{\sharp}}^{*}}\left(\frac{p-1}{q-1}\right)}|\Omega|^{\frac{p_{s_{\sharp}}^{*}-1}{p_{s_{\sharp}}^{*}}\left(\frac{q-p}{q-1}\right)}},
$$
where $C_{1}, C_{2}$ and $C_{p_{s_{\sharp}}^{*}}$ are the embedding constants of $L^{p_{s_{\sharp}}^{*}}(\Omega) \hookrightarrow L^{1}(\Omega), L^{p_{s_{\sharp}}^{*}}(\Omega) \hookrightarrow$ $L^{q}(\Omega)$, and $X_{p}(\Omega) \hookrightarrow L^{p_{s_{\sharp}}^{*}}(\Omega)$, respectively, and $|\Omega|$ is the Lebesgue measure of the set $\Omega$. Then, the nonlocal elliptic problem \eqref{problem 1.1} admits a weak solution $u_{0, \lambda} \in X_{p}(\Omega)$ and we have
$$
\rho_{p}(u_{0, \lambda})<\left(\frac{\lambda(q-1) a_{2} C_{2} C_{p_{s_{\sharp}}^{*}}^{q}|\Omega|^{({p_{s_{\sharp}}^{*}-q})/{p_{s_{\sharp}}^{*}}}}{p-1}\right)^{\frac{1}{p-q}}.
$$
\end{thm} 

\begin{proof}
    First, we fix $\lambda$ such that
\begin{equation}\label{sub proof 1}
0<\lambda<\frac{(q-p)^{\frac{q-p}{q-1}}(p-1)^{\frac{p-1}{q-1}}}{(a_{1} C_1)^{\frac{q-p}{q-1}}(a_{2} C_{2})^{\frac{p-1}{q-1}}(q-1) C_{p_{s_{\sharp}}^{*}}^{p}|\Omega|^{\frac{p_{s_{\sharp}}^{*}-q}{p_{s_{\sharp}}^{*}}\left(\frac{p-1}{q-1}\right)}|\Omega|^{\frac{p_{s_{\sharp}}^{*}-1}{p_{s_{\sharp}}^{*}}\left(\frac{q-p}{q-1}\right)}}. 
\end{equation}

Next, we take $0<\varepsilon<r$ and define
\begin{equation*}
\Lambda_{\lambda}(\varepsilon, r):=\frac{\sup _{v \in \Phi_{p, s}^{-1}([0, r])} \Psi_{ \lambda}(v)-\sup _{v \in \Phi_{p, s}^{-1}([0, r-\varepsilon])} \Psi_{ \lambda}(v)}{\varepsilon}. 
\end{equation*}
By rescaling $v$ we obtain that
$$
\begin{aligned}
\Lambda_{\lambda}(\varepsilon, r) & \leq \frac{1}{\varepsilon}\left|\sup _{v \in \Phi_{p, s}^{-1}([0, r])} \Psi_{ \lambda}(v)-\sup _{v \in \Phi_{p, s}^{-1}([0, r-\varepsilon])} \Psi_{ \lambda}(v)\right| \\
& \leq \sup _{v \in \Phi_{p, s}^{-1}([0, 1])} \int_{\Omega}\left|\int_{(r-\varepsilon) v(x)}^{r v(x)} \frac{\lambda g(x, t)}{\varepsilon} d t\right| d x .
\end{aligned}
$$
Subsequently, by applying the subcritical growth condition \eqref{cond on g} for $g$ together with the continuous embedding $X_{p}(\Omega) \hookrightarrow L^{p_{s_{\sharp}}^{*}}(\Omega)$, and Hölder inequality for exponents $\frac{p_{s_{\sharp}}^{*}}{p_{s_{\sharp}}^{*}-q}$ and $\frac{p_{s_{\sharp}}^{*}}{q}$, we get
$$
\begin{aligned}
\Lambda_{\lambda}(\varepsilon, r) & \leq \sup _{v \in \Phi_{p, s}^{-1}([0, 1])} \int_{\Omega}\left|\int_{(r-\varepsilon) v(x)}^{r v(x)} \frac{\lambda g(x, t)}{\varepsilon} d t\right| d x \\
& \left.\left.\leq \frac{1}{\varepsilon} \sup _{v \in \Phi_{p, s}^{-1}([0, 1])} \int_{\Omega} \right\rvert\, \int_{(r-\varepsilon) v(x)}^{r v(x)} \lambda\left(a_{1}+a_{2}|t|^{q-1}\right)\right) d t \mid d x \\
& \leq \lambda \sup _{v \in \Phi_{p, s}^{-1}([0, 1])}\left(a_{1}\|v\|_{L^{1}(\Omega)}+\frac{a_{2}}{q}\left(\frac{r^{q}-(r-\varepsilon)^{q}}{\varepsilon}\right)\|v\|_{L^{q}(\Omega)}^{q}\right) \\
& \leq \lambda\left(a_{1} C_{1} C_{p_{s_{\sharp}}^{*}}|\Omega|^{\frac{p_{s_{\sharp}}^{*}-1}{p_{s_{\sharp}}^{*}}}+a_{2} C_{2} \frac{C_{p_{s_{\sharp}}^{*}}^{q}}{q}\left(\frac{r^{q}-(r-\varepsilon)^{q}}{\varepsilon}\right)|\Omega|^{\frac{p_{s_{\sharp}}^{*}-q}{p_{s_{\sharp}}^{*}}}\right),
\end{aligned}
$$
where $C_{1}, C_{2}$ and $C_{p_{s_{\sharp}}^{*}}$ are the embedding constants. Therefore,
\begin{equation}\label{sub proof 2}
\limsup _{\varepsilon \rightarrow 0} \Lambda_{\lambda}(\varepsilon, r) \leq \lambda\left(a_{1} C_{1} C_{p_{s_{\sharp}}^{*}}|\Omega|^{\frac{p_{s_{\sharp}}^{*}-1}{p_{s_{\sharp}}^{*}}}+a_{2} C_{2} C_{p_{s_{\sharp}}^{*}}^{q} r^{q-1}|\Omega|^{\frac{p_{s_{\sharp}}^{*}-q}{p_{s_{\sharp}}^{*}}}\right).
\end{equation}

Let us consider the real-valued function
$$
\varphi_{\lambda}(r)=\lambda\left(a_{1} C_{1} C_{p_{s_{\sharp}}^{*}}|\Omega|^{\frac{p_{s_{\sharp}}^{*}-1}{p_{s_{\sharp}}^{*}}}+a_{2} C_{2} C_{p_{s_{\sharp}}^{*}}^{q} r^{q-1}|\Omega|^{\frac{p_{s_{\sharp}}^{*}-q}{p_{s_{\sharp}}^{*}}}\right)-r^{p-1}
$$
for every $r>0$. It is easy to see that $\inf _{r>0} \varphi_{\lambda}(r)$ is attained at
$$
r_{0, \lambda}:=\left(\frac{\lambda(q-1) a_{2} C_{2} C_{p_{s_{\sharp}}^{*}}^{q}|\Omega|^{\frac{p_{s_{\sharp}}^{*}-q}{p_{s_{\sharp}}^{*}}}}{p-1}\right)^{\frac{1}{p-q}}.
$$
Moreover, from \eqref{sub proof 1}, we observe that
$$
\inf _{r>0} \varphi_{\lambda}(r)=\varphi_{\lambda}\left(r_{0, \lambda}\right)<0.
$$
Consequently, we deduce from \eqref{sub proof 2} that
\begin{equation}\label{sub proof 3}
\limsup _{\varepsilon \rightarrow 0} \Lambda_{\lambda}(\varepsilon, r) \leq r_{0, \lambda}^{p-1}.
\end{equation}

Combining \eqref{sub proof 3} with Lemma \ref{tech lem 1.1} and Lemma \ref{tech lem 2.1}, we deduce that
\begin{equation*}
\inf _{u \in \Phi_{p, s}^{-1}([0, r_{0, \lambda}))} \frac{\sup _{v \in \Phi_{p, s}^{-1}([0, r_{0, \lambda}])} \Psi_{\lambda}(v)-\Psi_{ \lambda}(u)}{r_{0, \lambda}^{p}-\Phi_{p, s}(u)^{p}}<\frac{1}{p}. 
\end{equation*}
The above relation implies that there exists $w_{\lambda} \in \Phi_{p, s}^{-1}([0, r_{0, \lambda}))$ such that
\begin{equation*}
\Psi_{\lambda}(u) \leq \sup _{v \in \Phi_{p, s}^{-1}([0, r_{0, \lambda}])} \Psi_{\lambda}(v)<\Psi_{\lambda}(w_{\lambda})+\frac{r_{0, \lambda}^{p}-\Phi_{p, s}(w_{\lambda})^{p}}{p}
\end{equation*}
for every $u \in \Phi_{p, s}^{-1}([0, r_{0, \lambda}])$. Therefore, we conclude that
\begin{equation}\label{sub proof 4}
\mathcal{E}_{\lambda}\left(w_{\lambda}\right)=\frac{1}{p} \Phi_{p, s}(w_{\lambda})^{p}-\Psi_{\lambda}\left(w_{\lambda}\right)<\frac{r_{0, \lambda}^{p}}{p}-\Psi_{\lambda}(u) 
\end{equation}
for every $u \in \Phi_{p, s}^{-1}([0, r_{0, \lambda}])$. 

Since the energy functional $\mathcal{E}_{\lambda}$ is sequentially weakly lower semicontinuous and coercive, by the direct method of variational calculus \cite[Theorem 1.2]{Struwe: 2008}, we conclude that its restriction to $\Phi_{p, s}^{-1}([0, r_{0, \lambda}])$ has a global minimum $u_{0, \lambda}$ in $\Phi_{p, s}^{-1}([0, r_{0, \lambda}])$. Note that $u_{0, \lambda} \in \Phi_{p, s}^{-1}([0, r_{0, \lambda}])$. Indeed, suppose that
$$
\Phi_{p, s}(u_{0, \lambda}):=\rho_p(u_{0, \lambda})=r_{0, \lambda}.
$$
Then by \eqref{sub proof 4} we have
$$
\mathcal{E}_{\lambda}\left(u_{0, \lambda}\right)=\frac{1}{p} r_{0, \lambda}^{p}-\Psi_{\lambda}\left(u_{0, \lambda}\right)>\mathcal{E}_{\lambda}\left(w_{\lambda}\right),
$$
contradicting the fact that $u_{0, \lambda}$ is a global minimum for $\mathcal{E}_{\lambda}$. Therefore, we conclude that $u_{0, \lambda} \in X_{p}(\Omega)$ is a local minimum for the energy functional $\mathcal{E}_{\lambda}$ with
$$
\Phi_{p, s}\left(u_{0, \lambda}\right):=\rho_p(u_{0, \lambda})<r_{0, \lambda}
$$
and hence, a weak solution to the problem \eqref{problem 1.1}. This completes the proof.
\end{proof}

\section{Brezis-Nirenberg type problem via mountain pass technique} \label{sec7}
The main aim of this section is to investigate the following Brezis-Nirenberg-type problem
\begin{equation}\label{problem 2}
 \left\{\begin{array}{cc}
    A_{\mu, p} u  =\lambda |u|^{q-2}u +|u|^{p_{s_\sharp}^{*}-2} u \text { in } \Omega,  \\
    u  = 0  \text { in } \mathbb{R}^{N} \backslash \Omega,
\end{array}\right.
\end{equation}
where $\lambda$ is a real positive parameter and $p<q<p_{s_\sharp}^{*}$. The exponent $p_{s_{\sharp}}^{*}$ is the fractional critical exponent defined in \eqref{critical exponent}.

\begin{definition}\label{weaksol MP}
    A weak solution of problem \eqref{problem 2} is a function $u \in X_{p}(\Omega)$ such that 
\begin{align*}
& \int_{[0,1]}\left(\int_{\mathbb{R}^{2 N}} \int_{[0, s)} \frac{C_{N, s, p}|u(x)-u(y)|^{p-2}(u(x)-u(y))(v(x)-v(y))}{|x-y|^{N+s p}} d x d y\right) d \mu^{+}(s) \\
& =\int_{[0, \bar{s}]}\left(\iint_{\mathbb{R}^{2 N}} \frac{C_{N, s, p}|u(x)-u(y)|^{p-2}(u(x)-u(y))(v(x)-v(y))}{|x-y|^{N+s p}} d x d y\right) d \mu^{-}(s) \\
& +\lambda \int_{\Omega}|u|^{q-2} u v d x+\int_{\Omega}|u|^{p_{s_\sharp}^{*}-2} u v d x,
\end{align*}
for all $v \in X_{p}(\Omega)$.
\end{definition}

The solutions of the problem \eqref{problem 2} coincide with the critical points of the functional $\mathcal{I}_{\lambda}: X_{p}(\Omega) \rightarrow \mathbb{R}$ given by
\begin{equation}\label{functional MP}
\mathcal{I}_{\lambda}(u):=\frac{1}{p}\left[\rho_{p}(u)\right]^{p}-\frac{1}{p} \int_{[0, \bar{s}]}[u]_{s, p}^{p} d \mu^{-}(s)-\frac{\lambda}{q} \int_{\Omega}|u|^{q} d x-\frac{1}{p_{s_{\sharp}}^{*}} \int_{\Omega}|u|^{p_{s_\sharp}^{*}} d x. 
\end{equation}

Moreover, we have
\begin{align*}
\langle \mathcal{I}^{\prime}_{\lambda}(u), v\rangle&= \int_{[0,1]}\left(\iint_{\mathbb{R}^{2 N}}  \frac{C_{N, s, p}|u(x)-u(y)|^{p-2}(u(x)-u(y))(v(x)-v(y))}{|x-y|^{N+s p}} d x d y\right) d \mu^{+}(s) \\
& -\int_{[0, \bar{s}]}\left(\iint_{\mathbb{R}^{2 N}} \frac{C_{N, s, p}|u(x)-u(y)|^{p-2}(u(x)-u(y))(v(x)-v(y))}{|x-y|^{N+s p}} d x d y\right) d \mu^{-}(s) \\
& -\lambda \int_{\Omega}|u|^{q-2}u v d x- \int_{\Omega}|u|^{p_{s_\sharp}^{*}-2} u v d x,
\end{align*}
for all $u, v \in X_{p}(\Omega)$.

The main result about the existence of a nontrivial solution to the problem \eqref{problem 2} is stated below.

\begin{thm}\label{main result 2}
  Let $ \Omega$ be a bounded subset of $\mathbb{R}^N$. Let $\mu=\mu^{+}-\mu^{-}$with $\mu^{+}$and $\mu^{-}$satisfying \eqref{measure 1}-\eqref{measure 3} and let $s_{\sharp}$ be as in \eqref{measure 4}. Assume $1<p<q<p_{s_{\sharp}}^*$ and $s_{\sharp}p<N$, where $p_{s_{\sharp}}^{*}:=\frac{N p}{N-s_{\sharp}p}$. Then, there exists $\lambda^*$ such that the problem \eqref{problem 2} admits at least a nontrivial solution for all $\lambda \geq \lambda^*$, provided that $\kappa \in [0, \kappa_0]$  with $\kappa_0$ being sufficiently small.
\end{thm}

The proof of Theorem \ref{main result 2} is based on the Mountain-Pass Lemma of Ambrosetti and Rabinowitz \cite{AR: 1973}. First, we show that the functional defined in \eqref{functional MP} satisfies the Mountain-Pass geometry.
\begin{lem}\label{MP geometry}
    Let $\lambda >0$. Then, 

    1) There exist $r >0$ and $\beta>0$ such that $\mathcal{I_\lambda}(u)\geq \beta$ for any $u\in X_{p}(\Omega)$, with $\rho_p(u)=r$. 

    2) There exists a function $\phi \in X_{p}(\Omega)$ such that $\rho_p(\phi)\geq r$ and $\mathcal{I_\lambda}(\phi)\leq \beta$ for some $r, \beta >0.$
\end{lem}

\begin{proof}
    1) Fix any $u\in X_{p}(\Omega)$. From the embedding \eqref{embedding}, Lemma \ref{reabsorb}, and the positivity of $\lambda$, it follows that 
    \begin{align*}
    \mathcal{I_\lambda}(u) & =\frac{1}{p}[\rho_p(u)]^{p}-\frac{1}{p} \int_{[0, \bar{s}]}[u]_{s, p}^{p} d \mu^{-}(s)-\frac{\lambda}{q}\|u\|_{L^q(\Omega)}^{q} -\frac{1}{p_{s_\sharp}^{*}}\|u\|_{L^{p_{s_\sharp}^{*}}(\Omega)}^{p_{s_\sharp}^{*}}\\
    & \geq \left(\frac{1}{p}-\frac{c_0\kappa}{p}\right)[\rho_p(u)]^{p}-\frac{\lambda}{q}\|u\|_{L^q(\Omega)}^{q}  -\frac{1}{p_{s_\sharp}^{*}}\|u\|_{L^{p_{s_\sharp}^{*}}(\Omega)}^{p_{s_\sharp}^{*}}\\
    & \geq \left(\frac{1}{p}-\frac{c_0\kappa}{p}\right)[\rho_p(u)]^{p}-\frac{\lambda}{q}C_q^q [\rho_p(u)]^{q} -\frac{1}{p_{s_\sharp}^{*}}C_{p_{s_\sharp}^{*}}^{p_{s_\sharp}^{*}}[\rho_p(u)]^{p_{s_\sharp}^{*}}.
\end{align*}
Thus, by setting
\begin{equation*}
    \eta_{\lambda}(t)=\left(\frac{1}{p}-\frac{c_0\kappa}{p}\right) t^p-\frac{\lambda}{q}C_q^q t^q-\frac{1}{p_{s_\sharp}^{*}}C_{p_{s_\sharp}^{*}}^{p_{s_\sharp}^{*}} t^{p_{s_\sharp}^*},
\end{equation*}
we can find some $r \in(0,1]$ sufficiently small such  that $\max _{t \in[0,1]} \eta_{\lambda}(t)=\eta_{\lambda}(r)>0$, as $p<q<p^*$ provided that we choose $\kappa >0$ small enough so that $0<c_0\kappa<1$. Therefore, it follows that  $\mathcal{J}_{\lambda}(u) \geq \beta=\eta_{\lambda}(r)>0$ for any $u \in X_{p}(\Omega)$, with $\rho_p (u)=r$. 

2) Take a nonnegative function $v\in X_{p}(\Omega)$, with $\rho_{p}(v)=1$ and $t>0$. Then we have
\begin{align*}
    \mathcal{I}_\lambda(tv) & = \frac{t^p}{p}[\rho_{p}(v)]^{p}-\frac{t^p}{p} \int_{[0, \bar{s}]}[v]_{s, p}^{p} d \mu^{-}(s)- \frac{\lambda t^q}{q} \|v\|_{L^q(\Omega)}^q -\frac{t^{p_{s_\sharp}^{*}}}{p_{s_\sharp}^{*}}\|v\|_{L^{p_{s_\sharp}^{*}}(\Omega)}^{p_{s_\sharp}^{*}}\\
    & = \frac{t^p}{p}\left(1-\int_{[0, \bar{s}]}[v]_{s, p}^{p} d \mu^{-}(s)\right)- \frac{\lambda t^q}{q} \|v\|_{L^q(\Omega)}^q -\frac{t^{p_{s_\sharp}^{*}}}{p_{s_\sharp}^{*}}\|v\|_{L^{p_{s_\sharp}^{*}}(\Omega)}^{p_{s_\sharp}^{*}}.\\
\end{align*}
Consequently, since $p<q<p_{s_\sharp}^{*}$, passing to the limit as $t\to +\infty$, we get that $$I_\lambda(tv) \to -\infty.$$ Hence, by taking $\phi=\bar{t}v$, with $\bar{t}>0$ sufficiently large, we get that $\mathcal{I}_{\lambda}(\phi)<0$ and $\rho_p (\phi) \geq 2$.   
\end{proof}

Notice that the function $\phi$, obtained in Lemma \ref{MP geometry} at some $\lambda_0>0$, satisfies $\mathcal{I}_{\lambda}(\phi)<0$ and $\rho_p (\phi) \geq 2> r$ for all $\lambda \geq \lambda_0$, since $r \in(0,1]$.

In order to study the compactness property of the functional $\mathcal{I}_{\lambda}$,  we employ the Palais-Smale condition at a suitable mountain pass level $c_{\lambda}$. Specifically, we fix $\lambda>0$ and define
\begin{equation}\label{minmax levels}
    c_{\lambda}=\inf _{\xi \in \Gamma} \max _{\tau \in[0,1]} \mathcal{I}_{\lambda}(\xi(\tau)),
\end{equation}
where
\begin{equation*}
    \Gamma=\left\{\xi \in C\left([0,1], X_p(\Omega)\right): \xi(0)=0, \xi(1)=\phi\right\} .
\end{equation*}

From Lemma \ref{MP geometry}, we immediately deduce that $c_{\lambda}>0$.

Recall that  a sequence $\left(u_k\right)_k \subset X_p(\Omega)$ is called a Palais-Smale (PS) sequence for $\mathcal{I}_{\lambda}$ at level $c_{\lambda} \in \mathbb{R}$ if
$$
\mathcal{I}_{\lambda}\left(u_k\right) \rightarrow c_{\lambda} \quad \text { and } \quad \mathcal{I}_{\lambda}^{\prime}\left(u_k\right) \rightarrow 0 \quad \text { as } k \rightarrow \infty .
$$

Furthermore, $\mathcal{I}_{\lambda}$ satisfies the PS condition at level $c_{\lambda}$ if every PS sequence $\left(u_k\right)_k$ at level $c_{\lambda}$ admits a convergent subsequence in $X_p(\Omega)$.

Next, we establish an asymptotic property of the level $c_{\lambda},$ which is crucial in the proof of the Theorem \ref{main result 2}. The proof is standard and we include it here for the sake of completeness.

\begin{lem}\label{decay of minmax}
Let $p<q<p_{s_\sharp}^{*}$. Then,
\begin{equation*}
    \lim _{\lambda \rightarrow \infty} c_\lambda=0,
\end{equation*}
where $c_\lambda$ is defined \eqref{minmax levels}.
\end{lem} 

\begin{proof}
    Fix $\lambda_0>0$. Let $\phi \in X_{p}(\Omega)$ be the function obtained by Lemma \ref{MP geometry}, depending possibly on $\lambda_0$. Hence $\mathcal{I}_{\lambda}$ satisfies the mountain pass geometry at $0$ and $\phi$ for all $\lambda \geq \lambda_0$. Thus there exists $\tau_{\lambda}>0$ such that $\mathcal{I}_{\lambda}\left(\tau_{\lambda} \phi\right)=\max _{\tau \geq 0} \mathcal{I}_{\lambda}(\tau \phi)$ and so $\left\langle\mathcal{I}_{\lambda}^{\prime}\left(\tau_{\lambda} \phi\right), \phi\right\rangle=0$. Then, we have
    \begin{align*}
\langle \mathcal{I}^{\prime}_{\lambda}(\tau_{\lambda} \phi), \phi\rangle&=\tau_{\lambda}^{p-1}\Bigg[\int_{[0,1]}\left(\iint_{\mathbb{R}^{2 N}}  \frac{C_{N, s, p}|\phi(x)-\phi(y)|^{p}}{|x-y|^{N+s p}} d x d y\right) d \mu^{+}(s) \\
& -\int_{[0, \bar{s}]}\left(\iint_{\mathbb{R}^{2 N}} \frac{C_{N, s, p}|\phi(x)-\phi(y)|^{p}}{|x-y|^{N+s p}} d x d y\right) d \mu^{-}(s) \Bigg]\\
& \quad \quad \quad -\lambda \tau_{\lambda}^{q-1} \int_{\Omega}|\phi|^{q} d x- \tau_{\lambda}^{p_{s_\sharp}^{*}-1} \int_{\Omega}|\phi|^{p_{s_\sharp}^{*}} d x\\
&= \tau_{\lambda}^{p-1}  (\rho_p(\phi))^p- \tau_{\lambda}^{p-1} \int_{[0, \bar{s}]}[\phi]_{s,p}^{p} d \mu^{-}(s) \\&\quad \quad \quad-\lambda \tau_{\lambda}^{q-1} \int_{\Omega}|\phi|^{q} d x- \tau_{\lambda}^{p_{s_\sharp}^{*}-1} \int_{\Omega}|\phi|^{p_{s_\sharp}^{*}} d x\\ &=0,
\end{align*}
and therefore,
\begin{align}\label{bounded ineq lambda 0}
    \tau_{\lambda}^{p-1} (\rho_p(\phi))^p&=\tau_{\lambda}^{p-1} \int_{[0, \bar{s}]}[\phi]_{s,p}^{p} d \mu^{-}(s) +\lambda \tau_{\lambda}^{q-1} \int_{\Omega}|\phi|^{q} d x+ \tau_{\lambda}^{p_{s_\sharp}^{*}-1} \int_{\Omega}|\phi|^{p_{s_\sharp}^{*}} d x \nonumber \\&\geq \lambda_0 \tau_{\lambda}^{q-1} \int_{\Omega}|\phi|^{q} d x.
\end{align}
Now, dividing both sides of \eqref{bounded ineq lambda 0} by $\tau_{\lambda}^{p-1},$ we obtain
\begin{equation*}
     (\rho_p(\phi))^p\geq \lambda_0 \tau_{\lambda}^{q-p} \int_{\Omega}|\phi|^{q} d x.
\end{equation*}
Consequently $\left\{\tau_{\lambda}\right\}_{\lambda \geq \lambda_0}$ is bounded in $\mathbb{R}$, since $p<q,$ $\int_{\Omega}|\phi|^{q} d x>0$ and $\phi$ depends only on $\lambda_0$ by Lemma \ref{MP geometry}.

Take now a sequence $\left(\lambda_k\right)_k \subset\left[\lambda_0, \infty\right)$ such that $\lambda_k \rightarrow \infty$ as $k \rightarrow \infty$. Clearly $\left(\tau_{ \lambda_k}\right)_k$ is bounded in $\mathbb{R}$. Thus, there exist a number $\ell \geq 0$ and a subsequence, still relabeled as $\left(\lambda_k\right)_k$, such that
\begin{equation*}
    \lim _{k \rightarrow \infty} \tau_{\lambda_k}=\ell .
\end{equation*}

From \eqref{bounded ineq lambda 0} there exists $L$ such that
\begin{equation}\label{bound lambda 0 1}
    \tau_{\lambda_k}^{p-1} \int_{[0, \bar{s}]}[\phi]_{s,p}^{p} d \mu^{-}(s)+\lambda_k \tau_{\lambda_k}^{q-1} \int_{\Omega}|\phi|^{q} d x+ \tau_{\lambda_k}^{p_{s_\sharp}^{*}-1} \int_{\Omega}|\phi|^{p_{s_\sharp}^{*}} d x \leq L
\end{equation}
for any $k \in \mathbb{N}$. We claim that $\ell=0$. Indeed, if $\ell>0$ we obtain
\begin{equation*}
    \lim _{k \rightarrow \infty}\left(\tau_{\lambda_k}^{p-1} \int_{[0, \bar{s}]}[\phi]_{s,p}^{p} d \mu^{-}(s)+\lambda_k \tau_{\lambda_k}^{q-1} \int_{\Omega}|\phi|^{q} d x+ \tau_{\lambda_k}^{p_{s_\sharp}^{*}-1} \int_{\Omega}|\phi|^{p_{s_\sharp}^{*}} d x\right)=\infty,
\end{equation*}
which contradicts \eqref{bound lambda 0 1}. Hence $\ell=0$ and
\begin{equation}\label{limit of lambda 0}
    \lim _{\lambda \rightarrow \infty} \tau_{\lambda}=0,
\end{equation}
since the sequence $\left(\lambda_k\right)_k$ is arbitrary. 

Consider now the path $\xi(\tau)=\tau \phi, \tau \in[0,1]$, belonging to $\Gamma$. By Lemma \ref{MP geometry} and \eqref{limit of lambda 0} we get
\begin{equation*}
    0<c_{\lambda} \leq \max _{\tau \in[0,1]} \mathcal{I}_{\lambda}(\tau \phi) \leq \mathcal{I}_{\lambda}\left(\tau_{\lambda} \phi\right) \leq \frac{1}{p} \tau_{\lambda}^p (\rho_p(\phi))^p \rightarrow 0
\end{equation*}
as $\lambda \rightarrow \infty$, since $\phi$ depends only on $\lambda_0$. This completes the proof of the lemma.
\end{proof}

Now we are in a position to prove the PS condition for the functional $\mathcal{I}_{\lambda}$. The proof is similar to that of \cite[Proposition 5.10]{DPSV}.
\begin{lem}\label{PS condition lem}
    Let $\theta_{0} \in(0,1)$ and
\begin{equation}\label{bound for level}
c^{*}:=\frac{s_{\sharp}}{N}\left(\left(1-\theta_{0}\right) \mathcal{S}(p)\right)^{N / s_{\sharp} p}, 
\end{equation}
where $$\mathcal{S}(p):=\inf_{\|u\|_{L^{p_{s_\sharp}^{*}}(\Omega)}=1} \int_{[0,1]}[u]_{s, p}^{p} d \mu^{+}(s)$$ is the Sobolev constant equivalently defined by \eqref{Sobolev constant 2}.
Then, there exists $\kappa_{0}>0$, depending on $N, \Omega, p, s_{\sharp}$ and $\theta_{0}$, such that if $\kappa \in\left[0, \kappa_{0}\right]$ and $c \in$ $\left(0, c^{*}\right)$, then the functional $\mathcal{I}_\lambda$ in \eqref{functional MP} satisfies the PS condition at level $c.$
\end{lem}

\begin{proof}
    Take $c \in\left(0, c^{*}\right)$ and let $(u_{n})_{n} \subset X_{p}(\Omega)$ be a PS sequence for $\mathcal{I}_{\lambda}$, that is, 
\begin{align}\label{level c}
\begin{split}
    \lim _{n \rightarrow+\infty} \mathcal{I}_{\lambda}\left(u_{n}\right) & =\lim _{n \rightarrow+\infty} \Bigg[\frac{1}{p}\left[\rho_{p}\left(u_{n}\right)\right]^{p}-\frac{1}{p} \int_{[0, \bar{s}]}\left[u_{n}\right]_{s, p}^{p} d \mu^{-}(s)\\
    &-\frac{\lambda}{q} \int_{\Omega}\left|u_{n}\right|^{q} d x-\frac{1}{p_{s_{\sharp}}^{*}} \int_{\Omega}\left|u_{n}\right|^{p_{s_{\sharp}}^{*}} d x \Bigg]=c
\end{split}
\end{align}
and
\begin{align}\label{derivative}
\begin{split}
    & \quad \lim _{n \rightarrow+\infty} \sup_{v \in X_p(\Omega)} \langle \mathcal{I}^{\prime}_{\lambda}\left(u_{n}\right), v\rangle  \\
&=  \lim _{n \rightarrow+\infty} \sup_{v \in X_p(\Omega)} \Bigg[ \int_{[0,1]}\left( \iint_{\mathbb{R}^{2 N}} \frac{C_{N, s, p}\left|u_{n}(x)-u_{n}(y)\right|^{p-2}\left(u_{n}(x)-u_{n}(y)\right)(v(x)-v(y))}{|x-y|^{N+s p}} d x d y\right) d \mu^{+}(s) \\
& -\int_{[0, \bar{s}]}\left( \iint_{\mathbb{R}^{2 N}} \frac{C_{N, s, p}\left|u_{n}(x)-u_{n}(y)\right|^{p-2}\left(u_{n}(x)-u_{n}(y)\right)(v(x)-v(y))}{|x-y|^{N+s p}} d x d y\right) d \mu^{-}(s) \\
& -\lambda \int_{\Omega}\left|u_{n}\right|^{p-2} u_{n} v d x-\int_{\Omega}\left|u_{n}\right|^{p_{s_\sharp}^{*}-2} u_{n} v d x\Bigg]=0. 
\end{split}
\end{align}
If we put $v:=-u_{n}$ in \eqref{derivative}, we obtain
\begin{align*}
0 & \leq \lim _{n \rightarrow+\infty} \Bigg[\int_{[0,1]}\left(C_{N, s, p} \iint_{\mathbb{R}^{2 N}} \frac{\left|u_{n}(x)-u_{n}(y)\right|^{p}}{|x-y|^{N+s p}} d x d y\right) d \mu^{+}(s) \\
& -\int_{[0, \bar{s}]}\left(C_{N, s, p} \iint_{\mathbb{R}^{2 N}} \frac{\left|u_{n}(x)-u_{n}(y)\right|^{p}}{|x-y|^{N+s p}} d x d y\right) d \mu^{-}(s) \\
& -\lambda \int_{\Omega}\left|u_{n}\right|^{q} d x-\int_{\Omega}\left|u_{n}\right|^{p_{s_\sharp}^{*}} d x\Bigg] \\
& =  \lim _{n \rightarrow+\infty} p \mathcal{I}_{\lambda}\left(u_{n}\right)+\lim _{n \rightarrow+\infty} \Bigg[\lambda\left(\frac{p}{q}-1\right) \|u_n\|_{L^q(\Omega)}^q + \left( \frac{p}{p_{s_\sharp}^{*}}-1\right)\|u_n\|_{L^{p_{s_\sharp}^{*}}(\Omega)}^{p_{s_\sharp}^{*}} \Bigg].
\end{align*}
Consequently, by \eqref{level c},
\begin{equation}
\lim _{n \rightarrow+\infty} \Bigg[\lambda\left(\frac{1}{p}-\frac{1}{q}\right) \|u_n\|_{L^q(\Omega)}^q + \left( \frac{1}{p}-\frac{1}{p_{s_\sharp}^{*}}\right)\|u_n\|_{L^{p_{s_\sharp}^{*}}(\Omega)}^{p_{s_\sharp}^{*}}\Bigg] \leq c.
\end{equation}
Furthermore, by Lemma \ref{reabsorb}, we have
\begin{equation*}
    \int_{[0, \bar{s}]}\left[u_{n}\right]_{s, p}^{p} d \mu^{-}(s) \leq c_{0} \kappa \int_{[\bar{s}, 1]}\left[u_{n}\right]_{s, p}^{p} d \mu^{+}(s) \leq c_{0} \kappa\left[\rho_{p}\left(u_{n}\right)\right]^{p},
\end{equation*}
for some $c_0:=c_0(N, \Omega, p).$
Therefore,
\begin{equation}\label{main ineq in proof}
\left[\rho_{p}\left(u_{n}\right)\right]^{p}-\int_{[0, \bar{s}]}\left[u_{n}\right]_{s, p}^{p} d \mu^{-}(s) \geq\left(1-c_0 \kappa\right)\left[\rho_{p}\left(u_{n}\right)\right]^{p}.
\end{equation}
From this and \eqref{level c}, we obtain that $\rho_{p}\left(u_{n}\right)$ is bounded uniformly in $n$, provided that $1-$ $c_0 \kappa>0$. Consequently, by Lemma \ref{Uniform convexity} and Proposition \ref{compact and cont embedding}, there exists $u \in X_{p}(\Omega)$ such that, up to a subsequence,
\begin{align}\label{convergences}
\begin{split}
    & u_{n} \rightharpoonup u \text { in } X_{p}(\Omega) \\
& u_{n} \rightarrow u \text { in } L^{r}(\Omega) \text { for any } r \in [1, p_{s_{\sharp}}^{*}),  \\
& u_{n} \rightarrow u \text { a.e. in } \Omega .
\end{split}
\end{align}

It remains to show that $u_{n} \rightarrow u$ in $X_{p}(\Omega)$ as $n \rightarrow+\infty$. To this end, we define $\widetilde{u}_{n}:=u_{n}-u$. Then, by \cite[Theorem 1]{BL: 1983}, we have that
\begin{equation}\label{ex eq 0}
    \|u\|_{L^{p_{s_\sharp}^{*}}(\Omega)}^{p_{s_\sharp}^{*}}=\lim _{n \rightarrow+\infty}\left(\|u_n\|_{L^{p_{s_\sharp}^{*}}(\Omega)}^{p_{s_\sharp}^{*}}-\|\widetilde{u}_n\|_{L^{p_{s_\sharp}^{*}}(\Omega)}^{p_{s_\sharp}^{*}}\right).
\end{equation}
Moreover, Lemma \ref{B-L lemma} implies that
\begin{equation}\label{ex eq 1}
\int_{[0,1]}[u]_{s, p}^{p} d \mu^{ \pm}(s)=\lim _{n \rightarrow+\infty} \left(\int_{[0,1]}\left[u_{n}\right]_{s, p}^{p} d \mu^{ \pm}(s)-\int_{[0,1]}\left[\widetilde{u}_{n}\right]_{s, p}^{p} d \mu^{ \pm}(s)\right).
\end{equation}
Next, putting $v:=u$ in Definition \ref{weaksol MP} yields
\begin{equation}\label{v=u}
\left[\rho_{p}(u)\right]^{p}-\int_{[0, \bar{s}]}[u]_{s, p}^{p} d \mu^{-}(s)=\lambda\|u\|_{L^q(\Omega)}^{q}+\|u\|_{L^{p_{s_\sharp}^{*}}(\Omega)}^{p_{s_\sharp}^{*}}.
\end{equation}
Similarly, testing identity \eqref{derivative} with $v:= \pm u_{n}$,
\begin{equation*}
\lim _{n \rightarrow+\infty} \left(\left[\rho_{p}(u_n)\right]^{p}-\int_{[0, \bar{s}]}[u_n]_{s, p}^{p} d \mu^{-}(s)-\lambda\|u_n\|_{L^q(\Omega)}^{q}-\|u_n\|_{L^{p_{s_\sharp}^{*}}(\Omega)}^{p_{s_\sharp}^{*}}\right)=0.
\end{equation*}
Together with \eqref{convergences}, this implies that
\begin{equation}\label{ex eq 2}
\lim _{n \rightarrow+\infty} \left(\left[\rho_{p}(u_n)\right]^{p}-\int_{[0, \bar{s}]}[u_n]_{s, p}^{p} d \mu^{-}(s)-\|u_n\|_{L^{p_{s_\sharp}^{*}}(\Omega)}^{p_{s_\sharp}^{*}}\right)=\lambda\|u\|_{L^q(\Omega)}^{q}.
\end{equation}
Then, in light of \eqref{v=u}, we deduce that 
\begin{align}\label{ex eq 00}
\begin{split}
    \lim _{n \rightarrow+\infty} &\left(\left[\rho_{p}(u_n)\right]^{p}-\int_{[0, \bar{s}]}[u_n]_{s, p}^{p} d \mu^{-}(s)-\|u_n\|_{L^{p_{s_\sharp}^{*}}(\Omega)}^{p_{s_\sharp}^{*}}\right)\\
&=\left[\rho_{p}(u)\right]^{p}-\int_{[0, \bar{s}]}[u]_{s, p}^{p} d \mu^{-}(s)-\|u\|_{L^{p_{s_\sharp}^{*}}(\Omega)}^{p_{s_\sharp}^{*}}.
\end{split}
\end{align}
Hence, applying \eqref{ex eq 0} and \eqref{ex eq 1}, we obtain from \eqref{ex eq 00} that
\begin{equation*}
\lim _{n \rightarrow+\infty} \left(\left[\rho_{p}(\widetilde{u}_n)\right]^{p}-\int_{[0, \bar{s}]}[\widetilde{u}_n]_{s, p}^{p} d \mu^{-}(s)-\|\widetilde{u}_n\|_{L^{p_{s_\sharp}^{*}}(\Omega)}^{p_{s_\sharp}^{*}}\right)=0.
\end{equation*}
By combining this with the definition of the Sobolev constant in \eqref{Sobolev constant 2}, we obtain 
\begin{equation*}
    \lim _{n \rightarrow+\infty}\left[\left[\rho_{p}\left(\widetilde{u}_{n}\right)\right]^{p}-\int_{[0, \bar{s}]}\left[\widetilde{u}_{n}\right]_{s, p}^{p} d \mu^{-}(s)-\left(\frac{1}{\mathcal{S}(p)} \int_{[0,1]}\left[\widetilde{u}_{n}\right]_{s, p}^{p} d \mu^{+}(s)\right)^{p_{s_\sharp}^{*} / p} \right]\leq 0,
\end{equation*}
that is
\begin{equation*}
  \lim _{n \rightarrow+\infty}\left[\left[\rho_{p}\left(\widetilde{u}_{n}\right)\right]^{p}-\int_{[0, \bar{s}]}\left[\widetilde{u}_{n}\right]_{s, p}^{p} d \mu^{-}(s)-\frac{\left[\rho_{p}\left(\widetilde{u}_{n}\right)\right]^{p_{s_\sharp}^{*}}}{(\mathcal{S}(p))^{p_{s_\sharp}^{*} / p}}\right] \leq 0.  
\end{equation*}
Now, applying \eqref{main ineq in proof} to $\widetilde{u}_{n}$ instead of $u_{n}$, we get
\begin{equation*}
    \lim _{n \rightarrow+\infty}\left[\left(1-c_0 \kappa\right)\left[\rho_{p}\left(\widetilde{u}_{n}\right)\right]^{p}-\frac{\left[\rho_{p}\left(\widetilde{u}_{n}\right)\right]^{p_{s_\sharp}^{*}}}{(\mathcal{S}(p))^{p_{s_\sharp}^{*} / p}}\right] \leq 0,
\end{equation*}
which is equivalent to
\begin{equation}\label{main ineq 2}
\lim _{n \rightarrow+\infty}\left[\rho_{p}\left(\widetilde{u}_{n}\right)\right]^{p}\left[\left(1-c_0 \kappa\right)(\mathcal{S}(p))^{p_{s_\sharp}^{*} / p}-\left[\rho_{p}\left(\widetilde{u}_{n}\right)\right]^{p_{s_{\sharp}}^{*}-p}\right] \leq 0 . 
\end{equation}

Now, observe that in order to complete the proof, it is enough to show that
$$\left(1-c_0 \kappa\right)(\mathcal{S}(p))^{p_{s_\sharp}^{*} / p}-\left[\rho_{p}\left(\widetilde{u}_{n}\right)\right]^{p_{s_{\sharp}}^{*}-p}> 0.$$
Applying \eqref{ex eq 2}, we obtain from \eqref{level c} that
\begin{align*}
c & =\lim _{n \rightarrow+\infty} \Bigg[\frac{1}{p}\left[\rho_{p}\left(u_{n}\right)\right]^{p}-\frac{1}{p} \int_{[0, \bar{s}]}\left[u_{n}\right]_{s, p}^{p} d \mu^{-}(s)-\frac{\lambda}{q} \int_{\Omega}\left|u_{n}\right|^{q} d x-\frac{1}{p_{s_\sharp}^{*}} \int_{\Omega}\left|u_{n}\right|^{p_{s_\sharp}^{*}} d x \Bigg]\\
& =\lim _{n \rightarrow+\infty}\Bigg[\left(\frac{1}{p}-\frac{1}{p_{s_\sharp}^{*}}\right)\left(\left[\rho_{p}\left(u_{n}\right)\right]^{p}-\int_{[0, \bar{s}]}\left[u_{n}\right]_{s, p}^{p} d \mu^{-}(s)\right)-\frac{\lambda}{q} \int_{\Omega}\left|u_{n}\right|^{q} d x\Bigg]+\frac{\lambda}{p_{s_{\sharp}}^{*}} \int_{\Omega}|u|^{q} d x .
\end{align*}
This and \eqref{convergences} yield that
\begin{equation*}
    c=\lim _{n \rightarrow+\infty} \frac{s_{\sharp}}{N}\left(\left[\rho_{p}\left(u_{n}\right)\right]^{p}-\int_{[0, \bar{s}]}\left[u_{n}\right]_{s, p}^{p} d \mu^{-}(s)\right)-\lambda\left(\frac{1}{q}-\frac{1}{p_{s_{\sharp}}^{*}}\right) \int_{\Omega}|u|^{q} d x .
\end{equation*}
Consequently, by invoking \eqref{ex eq 1}, we conclude that
\begin{align*}
\frac{c N}{s_{\sharp}}&=\lim _{n \rightarrow+\infty}\left(\left[\rho_{p}(u)\right]^{p}+\left[\rho_{p}\left(\widetilde{u}_{n}\right)\right]^{p}-\int_{[0, \bar{s}]}[u]_{s, p}^{p} d \mu^{-}(s)-\int_{[0, \bar{s}]}\left[\widetilde{u}_{n}\right]_{s, p}^{p} d \mu^{-}(s)\right)\\
&-\lambda \frac{ N}{s_{\sharp}}\left(\frac{1}{q}-\frac{1}{p_{s_{\sharp}}^{*}}\right) \int_{\Omega}|u|^{q} d x.
\end{align*}
Then, by \eqref{v=u},
\begin{align*}
    \frac{c N}{s_{\sharp}}&=\lim _{n \rightarrow+\infty}\left(\left[\rho_{p}\left(\widetilde{u}_{n}\right)\right]^{p}-\int_{[0, \bar{s}]}\left[\widetilde{u}_{n}\right]_{s, p}^{p} d \mu^{-}(s)\right)+\int_{\Omega}|u|^{p_{s_{\sharp}}^{*}} d x\\
    &+\lambda\left(1- \frac{ N}{s_{\sharp}}\left(\frac{1}{q}-\frac{1}{p_{s_{\sharp}}^{*}}\right)\right) \int_{\Omega}|u|^{q} d x.
\end{align*}
Since $1- \frac{ N}{s_{\sharp}}\left(\frac{1}{q}-\frac{1}{p_{s_{\sharp}}^{*}}\right)>0$, applying \eqref{main ineq in proof} again for $\widetilde{u}_{n}$ instead of $u_{n}$, we get
\begin{align*}
\frac{c N}{s_{\sharp}} & \geq \lim _{n \rightarrow+\infty}\left(1-c_0 \kappa\right)\left[\rho_{p}\left(\widetilde{u}_{n}\right)\right]^{p}+\|u\|_{L^{p_{s_\sharp}^{*}}(\Omega)}^{p_{s}^{*}}+\lambda\left(1- \frac{ N}{s_{\sharp}}\left(\frac{1}{q}-\frac{1}{p_{s_{\sharp}}^{*}}\right)\right) \|u\|_{L^q(\Omega)}^{q} \\
& \geq \lim _{n \rightarrow+\infty}\left(1-c_0 \kappa\right)\left[\rho_{p}\left(\widetilde{u}_{n}\right)\right]^{p} .
\end{align*}
Furthermore, by \eqref{bound for level} we have
\begin{equation*}
    \left(\left(1-\theta_{0}\right) \mathcal{S}(p)\right)^{N / s_{\sharp} p}=\frac{c^{*} N}{s_{\sharp}}>\frac{c N}{s_{\sharp}} \geq \lim _{n \rightarrow+\infty}\left(1-c_0 \kappa\right)\left[\rho_{p}\left(\widetilde{u}_{n}\right)\right]^{p},
\end{equation*}
and consequently
\begin{align*}
& \liminf _{n \rightarrow+\infty}\left(\left(1-c_0 \kappa\right)(\mathcal{S}(p))^{p_{s_{\sharp}}^{*} / p}-\left[\rho_{p}\left(\widetilde{u}_{n}\right)\right]^{p_{s_{\sharp}}^{*}-p}\right) \\
& \quad=\left(1-c_0 \kappa\right)(\mathcal{S}(p))^{N /\left(N-s_{\sharp} p\right)}-\limsup _{n \rightarrow+\infty}\left[\rho_{p}\left(\widetilde{u}_{n}\right)\right]^{s_{\sharp} p^{2} /\left(N-s_{\sharp} p\right)} \\
& \quad \geq\left(1-c_0 \kappa\right)(\mathcal{S}(p))^{N /\left(N-s_{\sharp} p\right)}-\left(\frac{\left(\left(1-\theta_{0}\right) \mathcal{S}(p)\right)^{N / s_{\sharp} p}}{1-c_0 \kappa}\right)^{s_{\sharp} p /\left(N-s_{\sharp} p\right)} \\
& \quad=\left[1-c_0 \kappa-\left(\frac{\left(1-\theta_{0}\right)^{N / s_{\sharp} p}}{1-c_0 \kappa}\right)^{s_{\sharp} p /\left(N-s_{\sharp} p\right)}\right](\mathcal{S}(p))^{N /\left(N-s_{\sharp} p\right)}.
\end{align*}
Since
\begin{align*}
   & \lim _{\kappa \rightarrow 0} \left[1-c_0 \kappa-\left(\frac{\left(1-\theta_{0}\right)^{N / s_{\sharp} p}}{1-c_0 \kappa}\right)^{s_{\sharp} p /\left(N-s_{\sharp} p\right)}\right] \\&\quad\quad\quad=1-\left(1-\theta_{0}\right)^{N /\left(N-s_{\sharp} p\right)}>0,
\end{align*}
we deduce that
\begin{equation*}
    \liminf _{n \rightarrow+\infty}\left(\left(1-c_0 \kappa\right)(\mathcal{S}(p))^{p_{s_{\sharp}}^{*} / p}-\left[\rho_{p}\left(\widetilde{u}_{n}\right)\right]^{p_{s_{\sharp}^{*}}^{*}-p}\right)>0,
\end{equation*}
for sufficiently small $\kappa$ such that $c_0\kappa<1.$
By combining this with \eqref{main ineq 2}, we conclude that 
\begin{equation*}
    \lim _{n \rightarrow+\infty} \rho_{p}\left(\widetilde{u}_{n}\right)=0,
\end{equation*}
which results in $u_{n} \rightarrow u$ in $X_{p}(\Omega)$ as $n \rightarrow+\infty$.
\end{proof}

Finally, the proof of Theorem \ref{main result 2} is concluded as follows.

\begin{proof}[Proof of Theorem \ref{main result 2}]
     By Lemma \ref{MP geometry} it follows that the functional $\mathcal{I}_\lambda$ verifies the geometry of the mountain pass lemma, meaning that the functional $\mathcal{I}_\lambda$ admits a Palais-Smale sequence at level $c_\lambda$. From Lemma \ref{PS condition lem} we deduce that $\mathcal{I}_\lambda$ satisfies the PS condition for some levels $c$ satisfying $c<c^{*}:=\frac{s_{\sharp}}{N}\left(\left(1-\theta_{0}\right) \mathcal{S}(p)\right)^{N / s_{\sharp} p}$ provided that $\kappa$ is sufficiently small. But, according to Lemma \ref{decay of minmax}, we have verified the existence of $\lambda^{*}$ such that $c_\lambda<c^{*}:=\frac{s_{\sharp}}{N}\left(\left(1-\theta_{0}\right) \mathcal{S}(p)\right)^{N / s_{\sharp} p}$ for all $\lambda \geq \lambda^{*}$.  Thus, for these values of $\lambda$ and $\kappa$, the functional $\mathcal{I}_\lambda$ satisfies the PS condition at the level $c_\lambda$. Moreover, since $\mathcal{I}_\lambda(u)=c_\lambda>0=\mathcal{I}_\lambda(0)$, we have $u\not\equiv 0$. Thus, $u$ is a nontrivial solution to the problem \eqref{problem 2}. 
\end{proof}

\section{Examples and applications} \label{secexpa}

In this section, we apply our main results from the previous sections to a variety of interesting examples, yielding new findings on existence, depending on the specific choice of the measure $\mu.$  We present particular cases of Theorems \ref{main result wlsc 0} and \ref{thmnonmain}. Similarly, one can discuss the applications of Theorems \ref{main result subcrit 0} and \ref{main result 2} in these specific settings. 
We will use the common notation $\mathbb{X}(\Omega)$ for the (fractional) Sobolev spaces involved in each individual case. This notation varies for each example, depending on the operator involved, and must be appropriately defined using $X_p(\Omega).$

\subsection{The $p$-Laplacian}\label{subsec2.1}
To begin, let us showcase the results of classical $p$-Laplacian and fractional $p$-Laplacian. In fact, with a specific choice of the measure $\mu^+$, Theorem \ref{main result wlsc 0} and Theorem \ref{thmnonmain} recover the main results of \cite{FF: 2015}.
\begin{cor} Let $ \Omega$ be a bounded subset of $\mathbb{R}^N$. Suppose that $N>p \geq 2$ and $q \in[1, p^*)$, where $p^*=\frac{p N}{N-p}$. Let $g: \Omega \times \mathbb{R} \rightarrow \mathbb{R}$ be a Carath\'eodory function satisfying the subcritical growth condition \begin{align*}
\begin{split}
&\text{ there exist } a_{1}, a_{2}>0 \text{ and } q \in[1, p^{*}), \text{ such that }\\
&|g(x, t)| \leq a_{1}+a_{2}|t|^{q-1} \text {, a.e. } x \in \Omega, \forall t \in \mathbb{R} . 
\end{split}
\end{align*} Then, for every $\gamma>0$, there exists $\Lambda_\gamma>0$ such that the problem 
\begin{equation}
 \begin{cases}
      -\Delta_{p} u  = \lambda g(x,u) +\gamma |u|^{p^{*}-2} u \quad & \text { in } \Omega,  \\ \quad
    u  = 0  &\text { in } \mathbb{R}^{N} \backslash \Omega, 
 \end{cases}
\end{equation}
admits at least one weak solution in the space $\mathbb{X}(\Omega)$ for every $0<\lambda<\Lambda_\gamma$ which is a local minimum of the energy functional
$$
\mathcal{J}_{\gamma, \lambda}(u):=\frac{1}{p} \int_{\Omega} |\nabla u|^p-\lambda \int_{\Omega} \int_{0}^{u(x)} g(x, \tau) d \tau dx-\frac{\gamma}{p^{*}} \int_{\Omega}|u|^{p^{*}} d x
$$
for every $u \in \mathbb{X}(\Omega)$.
    \end{cor}
    \begin{proof}
        We take $\mu=\delta_{1},$ the Dirac measure (unit mass) based on the point $1.$ Thus, it satisfies all conditions \eqref{measure 1}-\eqref{measure 3}  with $\bar{s}=1$ and $\kappa=0.$ Consequently,  we choose $s_\sharp:=1$  as in \eqref{measure 4}. Now, the proof of Corollary follows directly from Theorem \ref{main result wlsc 0}.
    \end{proof}
\begin{cor} Let $ \Omega$ be a bounded subset of $\mathbb{R}^N$. Let $N>p \geq 2$. Suppose that $m$ and $q$ are two real constants such that 
$$1 \leq m<p\leq q<p^*=\frac{p N}{N-p}.$$
Then, there exists an open interval $\Lambda \subset (0, +\infty)$ such that, for every $\lambda \in \Lambda,$ the  nonlocal problem 
\begin{equation} 
       \begin{cases} 
-\Delta_p u=  |u|^{p^*-2}u+\lambda (|u|^{m-1}+|u|^{q-1}) \quad &\text{in}\quad \Omega, \\
u=0\quad & \text{in}\quad \mathbb{R}^N\backslash \Omega,
\end{cases}
    \end{equation} has at least one nonnegative nontrivial weak solution $u_\lambda \in \mathbb{X}(\Omega).$ 
    
\end{cor}
 \begin{proof}
        We take $\mu=\delta_{1},$ the Dirac measure (unit mass) based on the point $1.$ Thus, it satisfies all conditions \eqref{measure 1}-\eqref{measure 3} and condition \eqref{extracondi}  with $\bar{s}=1$ and $\kappa=0$. Consequently, we choose $s_\sharp:=1$  as in \eqref{measure 4}. Now, the result follows directly from Theorem \ref{thmnonmain}.
    \end{proof}

\subsection{Fractional $p$-Laplacian} \label{subsec2.2}
The results obtained in this setting capture the main results of \cite{MB: 2017}.
\begin{cor} 
      Let $ \Omega$ be a bounded subset of $\mathbb{R}^N$ and let $s \in (0, 1)$.  Suppose that $\frac{N}{s}>p \geq 2$ and $q \in[1, p_{s}^*)$, where $p_{s}^*=\frac{p N}{N-sp}$. Let $g: \Omega \times \mathbb{R} \rightarrow \mathbb{R}$ be a Carath\'eodory function satisfying the subcritical growth condition: \begin{align}
\begin{split}
&\text{ there exist } a_{1}, a_{2}>0 \text{ and } q \in[1, p_{s}^{*}), \text{ such that }\\
&|g(x, t)| \leq a_{1}+a_{2}|t|^{q-1} \text {, a.e. } x \in \Omega, \forall t \in \mathbb{R} . 
\end{split}
\end{align}
Then, for every $\gamma>0$, there exists $\Lambda_\gamma>0$ such that the problem \begin{equation}
 \begin{cases}
    (-\Delta_{p})^s u  = \lambda g(x,u) +\gamma |u|^{p^{*}-2} u \quad & \text { in } \Omega,  \\ \quad
    u  = 0  &\text { in } \mathbb{R}^{N} \backslash \Omega, 
 \end{cases}
\end{equation} admits at least one weak solution in the space $\mathbb{X}(\Omega)$ for every $0<\lambda<\Lambda_\gamma$ which is a local minimum of the energy functional
$$
\mathcal{J}_{\gamma, \lambda}(u):=\frac{1}{p} \int_{\mathbb{R}^{2N}}\frac{|u(x)-u(y)|^p}{|x-y|^{N+ps}} dx\,dy -\lambda \int_{\Omega} \int_{0}^{u(x)} g(x, \tau) d \tau dx-\frac{\gamma}{p_{s}^{*}} \int_{\Omega}|u|^{p_{s_\sharp}^{*}} d x
$$
for every $u \in \mathbb{X}(\Omega)$. 

\end{cor} 
\begin{proof}
    For $s \in (0, 1),$ we choose the measure $\mu:=\delta_s,$ where $\delta_s$ is the Dirac measure centered at $s.$ Then it is easy to verify that $\mu$ satisfies all the conditions \eqref{measure 1}-\eqref{measure 3}  with $\bar{s}=s$ and $\kappa=0.$ In this situation, we choose $s_\sharp:=s$ and the proof of the corollary follows from Theorem \ref{main result wlsc 0}.
\end{proof}

\begin{cor} Let $ \Omega$ be a bounded subset of $\mathbb{R}^N$. Let $\frac{N}{s}>p \geq 2$. Suppose that $m$ and $q$ are two real constants such that 
$$1 \leq m<p\leq q<p_s^*=\frac{p N}{N-sp}.$$
Then, there exists an open interval $\Lambda \subset (0, +\infty)$ such that, for every $\lambda \in \Lambda,$ the  nonlocal problem 
\begin{equation} 
       \begin{cases} 
(-\Delta_{p})^s u=  |u|^{p_s^*-2}u+\lambda (|u|^{m-1}+|u|^{q-1}) \quad &\text{in}\quad \Omega, \\
u=0\quad & \text{in}\quad \mathbb{R}^N\backslash \Omega,
\end{cases}
    \end{equation} has at least one nonnegative nontrivial weak solution $u_\lambda \in  \mathbb{X}(\Omega).$ 
    
\end{cor}

\begin{proof}
    For $s \in (0, 1),$ we choose the measure $\mu:=\delta_s,$ where $\delta_s$ is the Dirac measure centered at $s.$ Then, it is easy to verify that $\mu$ satisfies all the conditions \eqref{measure 1}-\eqref{measure 3} and \eqref{extracondi} with $\bar{s}=s$ and $\kappa=0.$ In this situation, we choose $s_\sharp:=s$ and the proof of the corollary follows from Theorem \ref{thmnonmain}.
\end{proof}

\subsection{Mixed local and nonlocal operators} \label{subsec2.3}
The next result is about the existence of a weak solution for the mixed local and nonlocal operator $ -\Delta_{p} +(-\Delta_{p})^s.$ This also seems new even for the case $p=2.$
\begin{cor}
    Let $ \Omega$ be a bounded subset of $\mathbb{R}^N$ and let $s \in [0, 1)$. Suppose that $N>p \geq 2$ and $q \in[1, p^*)$, where $p^*=\frac{p N}{N-p}$. Let $g: \Omega \times \mathbb{R} \rightarrow \mathbb{R}$ be a Carath\'eodory function satisfying the subcritical growth condition \begin{align*}
\begin{split}
&\text{ there exist } a_{1}, a_{2}>0 \text{ and } q \in[1, p^{*}), \text{ such that }\\
&|g(x, t)| \leq a_{1}+a_{2}|t|^{q-1} \text {, a.e. } x \in \Omega, \forall t \in \mathbb{R} . 
\end{split}
\end{align*} Then, for every $\gamma>0$, there exists $\Lambda_\gamma>0$ such that the problem 
\begin{equation}
 \begin{cases}
      -\Delta_{p} u+(-\Delta_{p})^s u  = \lambda g(x,u) +\gamma |u|^{p^{*}-2} u \quad & \text { in } \Omega,  \\ \quad
    u  = 0  &\text { in } \mathbb{R}^{N} \backslash \Omega 
 \end{cases}
\end{equation}
admits at least one weak solution in the space $\mathbb{X}(\Omega)$ for every $0<\lambda<\Lambda_\gamma$ which is a local minimum of the energy functional
$$
\mathcal{J}_{\gamma, \lambda}(u):=\frac{1}{p} \int_{\Omega} |\nabla u|^p+\frac{1}{p} \int_{\mathbb{R}^{2N}}\frac{|u(x)-u(y)|^p}{|x-y|^{N+ps}} dx\,dy -\lambda \int_{\Omega} \int_{0}^{u(x)} g(x, \tau) d \tau dx-\frac{\gamma}{p^{*}} \int_{\Omega}|u|^{p^{*}} d x
$$
for every $u \in \mathbb{X}(\Omega)$.
\end{cor}
\begin{proof}
    In this case, we choose $\mu:=\delta_1+\delta_s,$ where $\delta_1$ and $\delta_s$ are the Dirac measures centered at $1$ and $s \in [0, 1),$ respectively. Then, $\mu$ satisfies all the conditions \eqref{measure 1}-\eqref{measure 3}  with $\bar{s}:=1$ and $\kappa:=0.$ Now, we choose $s_\sharp:=1$ and the proof of the corollary follows from Theorem \ref{main result wlsc 0}.
\end{proof}

\begin{cor}  Let $ \Omega$ be a bounded subset of $\mathbb{R}^N$ and let $N>p \geq 2.$  Suppose that $m$ and $q$ are two real constants such that 
$$1 \leq m<p\leq q<p^*=\frac{p N}{N-p}.$$
Then, there exists an open interval $\Lambda \subset (0, +\infty)$ such that, for every $\lambda \in \Lambda,$ the  nonlocal problem 
\begin{equation} 
       \begin{cases} 
 -\Delta_{p} u+(-\Delta_{p})^s u=  |u|^{p^*-2}u+\lambda (|u|^{m-1}+|u|^{q-1}) \quad &\text{in}\quad \Omega, \\
u=0\quad & \text{in}\quad \mathbb{R}^N\backslash \Omega,
\end{cases}
    \end{equation} has at least a nonnegative nontrivial weak solution $u_\lambda \in  \mathbb{X}(\Omega).$ 
    
\end{cor}

\begin{proof}
     In this case, we choose $\mu:=\delta_1+\delta_s,$ where $\delta_1$ and $\delta_s$ are the Dirac measures centered at $1$ and $s \in [0, 1),$ respectively. Then, $\mu$ satisfies all the conditions \eqref{measure 1}-\eqref{measure 3} and \eqref{extracondi}  with $\bar{s}:=1$ and $\kappa:=0.$ Now, we choose $s_\sharp:=1$ and the proof of the corollary follows from Theorem \ref{thmnonmain}.
\end{proof}

\subsection{Nonlocal operators associated with a convergent series of Dirac measures} \label{subsec2.4}
Now, we demonstrate that in our setting we can also choose $\mu$ as a convergent series of Dirac measures. We have the following result. 
\begin{cor}  Let $ \Omega$ be a bounded subset of $\mathbb{R}^N.$
    Let $p \in [2, N)$ and $1 \geq s_0>s_1>s_2> \ldots \geq 0.$ Consider the operator 
    $$\sum_{k=0}^{+\infty} a_k (-\Delta_p)^{s_k}\quad \text{with} \quad \sum_{k=0}^{+\infty} a_k \in (0, +\infty).$$
    Suppose that $a_0>0$ and $a_k \geq 0$ for all $k \geq 1.$ Let $p_{s_0}^*:= \frac{pN}{N-ps_{0}}$ be the fractional critical Sobolev exponent and $q \in [1, p_{s_0}^*).$  Let $g: \Omega \times \mathbb{R} \rightarrow \mathbb{R}$ be a Carath\'eodory function satisfying the subcritical growth condition \begin{align*}
\begin{split}
&\text{ there exist } a_{1}, a_{2}>0 \text{ and } q \in[1, p_{s_0}^*), \text{ such that }\\
&|g(x, t)| \leq a_{1}+a_{2}|t|^{q-1} \text {, a.e. } x \in \Omega, \forall t \in \mathbb{R} . 
\end{split}
\end{align*} Then, for every $\gamma>0$, there exists $\Lambda_\gamma>0$ such that the problem 
\begin{equation}
 \begin{cases}
     \sum_{k=0}^{+\infty} a_k (-\Delta_p)^{s_k} u  = \lambda g(x,u) +\gamma |u|^{p_{s_0}^*-2} u \quad & \text { in } \Omega,  \\ \quad
    u  = 0  &\text { in } \mathbb{R}^{N} \backslash \Omega, 
 \end{cases}
\end{equation}
admits at least one weak solution in the space $\mathbb{X}(\Omega)$ for every $0<\lambda<\Lambda_\gamma.$
\end{cor}

\begin{proof}
 To conclude the proof of this result from Theorem \ref{main result wlsc 0},
 we set 
 $$\mu:= \sum_{k=0}^{+\infty} a_k \delta_{s_k},$$
 where $\delta_{s_k}$ denote the Dirac measures at  $s_k.$ Next, we note that $\mu$ satisfies all the conditions \eqref{measure 1}-\eqref{measure 3} by choosing $\bar{s}:=s_0,$ $\kappa:=0$ and $s_\sharp:=s_0.$
 \end{proof}

\begin{cor}  Let $ \Omega$ be a bounded subset of $\mathbb{R}^N.$
    Let $p \in [2, N)$ and $1 \geq s_0>s_1>s_2> \ldots \geq 0.$ Consider the operator 
    $$\sum_{k=0}^{+\infty} a_k (-\Delta_p)^{s_k}\quad \text{with} \quad \sum_{k=0}^{+\infty} a_k \in (0, +\infty).$$
    Suppose that there exist $\kappa \geq 0$ and $\bar{k} \in \mathbb{N}$ such that 
    \begin{equation} \label{condis_k}
        a_k>0 \quad \forall \,\, k \in \{0, 1, \ldots, \bar{k}\}, \quad \text{and}\,\,\, \sum_{k=\bar{k}+1}^{+\infty} a_k \leq \kappa \sum_{k=0}^{\bar{k}} a_k.
    \end{equation}
    
     Let $p_{s_0}^*:= \frac{pN}{N-ps_{0}}$ be the fractional critical Sobolev exponent and $q \in [1, p_{s_0}^*).$  Let $g: \Omega \times \mathbb{R} \rightarrow \mathbb{R}$ be a Carath\'eodory function satisfying the subcritical growth condition \begin{align*}
\begin{split}
&\text{ there exist } a_{1}, a_{2}>0 \text{ and } q \in[1, p_{s_0}^*), \text{ such that }\\
&|g(x, t)| \leq a_{1}+a_{2}|t|^{q-1} \text {, a.e. } x \in \Omega, \forall t \in \mathbb{R} . 
\end{split}
\end{align*} Then, for every $\gamma>0$, there exists $\Lambda_\gamma>0$ such that the problem 
\begin{equation}
 \begin{cases}
     \sum_{k=0}^{+\infty} a_k (-\Delta_p)^{s_k} u  = \lambda g(x,u) +\gamma |u|^{p_{s_0}^*-2} u \quad & \text { in } \Omega,  \\ \quad
    u  = 0  &\text { in } \mathbb{R}^{N} \backslash \Omega, 
 \end{cases}
\end{equation}
admits at least one weak solution in the space $\mathbb{X}(\Omega)$ for every $0<\lambda<\Lambda_\gamma.$
\end{cor}
\begin{proof} We once again select measure $\mu$ as follows: 
 $$\mu:= \sum_{k=0}^{+\infty} a_k \delta_{s_k},$$
 where $\delta_{s_k}$ denote the Dirac measures supported at  $s_k.$

By choosing $\bar{s}=s_{\bar{k}}$ and $s_\sharp:=s_0$ along with utilizing condition \eqref{condis_k} for $\kappa$, we can easily check that $\mu$ satisfies  all the conditions \eqref{measure 1}-\eqref{measure 3}. Thus, we derived the result as an application of Theorem \ref{main result wlsc 0}. \end{proof}

\begin{cor}  Let $ \Omega$ be a bounded subset of $\mathbb{R}^N.$
    Let $p \in [2, N)$ and $1 \geq s_0>s_1>s_2> \ldots \geq 0.$ Consider the operator 
    $$\sum_{k=0}^{+\infty} a_k (-\Delta_p)^{s_k}\quad \text{with} \quad \sum_{k=0}^{+\infty} a_k \in (0, +\infty).$$
    Suppose that $a_0>0$ and $a_k \geq 0$ for all $k \geq 1.$ Let $p_{s_0}^*:= \frac{pN}{N-ps_{0}}$ be the fractional critical Sobolev exponent and let $m,\, q$ be two real constants such that 
$$1 \leq m<p\leq q<p_{s_0}^*=\frac{pN}{N-ps_{0}}.$$
Then, there exists an open interval $\Lambda \subset (0, +\infty)$ such that, for every $\lambda \in \Lambda,$ the  nonlocal problem 
\begin{equation} 
       \begin{cases} 
 \sum_{k=0}^{+\infty} a_k (-\Delta_p)^{s_k} u =  |u|^{p_{s_0}^*-2}u+\lambda (|u|^{m-1}+|u|^{q-1}) \quad &\text{in}\quad \Omega, \\
u=0\quad & \text{in}\quad \mathbb{R}^N\backslash \Omega,
\end{cases}
    \end{equation} has at least one nonnegative nontrivial weak solution $u_\lambda \in  \mathbb{X}(\Omega).$ 
\end{cor}
\begin{proof}
 To conclude the proof of this result from Theorem \ref{thmnonmain},
 we set 
 $$\mu:= \sum_{k=0}^{+\infty} a_k \delta_{s_k},$$
 where $\delta_{s_k}$ denote the Dirac measures  at  $s_k.$ Next, we note that $\mu$ satisfies all the conditions \eqref{measure 1}-\eqref{measure 3} and  \eqref{extracondi} by choosing $\bar{s}:=s_0,$ $\kappa:=0$ and $s_\sharp:=s_0.$
 \end{proof}

\begin{cor}  Let $ \Omega$ be a bounded subset of $\mathbb{R}^N$ such that $\Omega \subset B_R,$ for some $R>0.$
     Let $p \in [2, N)$ and $1 \geq s_0>s_1>s_2> \ldots \geq 0.$ Consider the operator 
    $$\sum_{k=0}^{+\infty} a_k (-\Delta_p)^{s_k}\quad \text{with} \quad \sum_{k=0}^{+\infty} a_k \in (0, +\infty).$$
    Suppose that there exist $\kappa \geq 0$ and $\bar{k} \in \mathbb{N}$ such that 
    \begin{equation} \label{cond212}
        a_0>0 \quad a_k\geq 0,~  \forall \,\, k \in \{0, 1, \ldots, \bar{k}\}, \quad \text{and}\,\,\, \sum_{k=\bar{k}+1}^{+\infty} a_k \leq \kappa \sum_{k=0}^{\bar{k}} a_k.
    \end{equation}
    In addition to this, we also assume that there exist $\bar{\kappa}:=\bar{\kappa}(N, s_0, R ) > 0,$ $\delta \in (0, 1-s_0]$ and $\bar{k} \in \mathbb{N}$  such that 
    \begin{equation} \label{condis_k'}
        \sum_{k=\bar{k}+1}^{+\infty} a_k \leq \bar{\kappa} \delta \sum_{k=1}^{\bar{k}} a_k.
    \end{equation}
     Let $p_{s_0}^*:= \frac{pN}{N-ps_{0}}$ be the fractional critical Sobolev exponent and let $m,\, q$ be two real constants such that 
$$1 \leq m<p\leq q<p_{s_0}^*=\frac{pN}{N-ps_{0}}.$$
Then, there exists an open interval $\Lambda \subset (0, +\infty)$ such that, for every $\lambda \in \Lambda,$ the  nonlocal problem 
\begin{equation} 
       \begin{cases} 
  \sum_{k=0}^{+\infty} a_k (-\Delta_p)^{s_k} u  =  |u|^{p_{s_0}^*-2}u+\lambda (|u|^{m-1}+|u|^{q-1}) \quad &\text{in}\quad \Omega, \\
u=0\quad & \text{in}\quad \mathbb{R}^N\backslash \Omega,
\end{cases}
    \end{equation} has at least one nonnegative nontrivial weak solution $u_\lambda \in  \mathbb{X}(\Omega).$ 
\end{cor}
\begin{proof} We select measure $\mu$ as follows: 
 $$\mu:= \sum_{k=0}^{+\infty} a_k \delta_{s_k},$$
 where $\delta_{s_k}$ denote the Dirac measures  at  $s_k.$

By choosing $\bar{s}=s_{\bar{k}}$ and $s_\sharp:=s_0$ along with utilizing condition \eqref{cond212} and condition \eqref{condis_k'}, we can easily check that $\mu$ satisfies  all the conditions \eqref{measure 1}-\eqref{measure 3} and condition \eqref{extracondi}. Thus, we derived the result as an application of Theorem \ref{thmnonmain}. \end{proof}

\subsection{Mixed local and nonlocal operator with wrong sign} \label{subsec2.5}
An intriguing situation arises when the measure $\mu$ changes sign. This implies, for instance, that the operator may incorporate a minor term with the ``wrong" sign. To the best of our knowledge, no existing literature addresses a result of this nature even for the case $p=2.$

\begin{cor}
Let $ \Omega$ be a bounded subset of $\mathbb{R}^N$ and let $s \in [0, 1)$. Suppose that $N>p \geq 2$ and $q \in[1, p^*)$, where $p^*=\frac{p N}{N-p}$. Let $g: \Omega \times \mathbb{R} \rightarrow \mathbb{R}$ be a Carath\'eodory function satisfying the subcritical growth condition \begin{align*}
\begin{split}
&\text{ there exist } a_{1}, a_{2}>0 \text{ and } q \in[1, p^{*}), \text{ such that }\\
&|g(x, t)| \leq a_{1}+a_{2}|t|^{q-1} \text {, a.e. } x \in \Omega, \forall t \in \mathbb{R} . 
\end{split}
\end{align*} Then, for every $\gamma>0$ and for all $\alpha \geq 0,$ there exists  $\Lambda_\gamma>0$ such that the problem 
\begin{equation}
 \begin{cases}
      -\Delta_{p} u-\alpha (-\Delta_{p})^s u  = \lambda g(x,u) +\gamma |u|^{p^{*}-2} u \quad & \text { in } \Omega,  \\ \quad
    u  = 0  &\text { in } \mathbb{R}^{N} \backslash \Omega,
 \end{cases}
\end{equation}
admits at least one weak solution in the space $\mathbb{X}(\Omega)$ for every $0<\lambda<\Lambda_\gamma$, which is a local minimum of the energy functional
$$
\mathcal{J}_{\gamma, \lambda}(u):=\frac{1}{p} \int_{\Omega} |\nabla u|^p-  \frac{\alpha}{p} \int_{\mathbb{R}^{2N}}\frac{|u(x)-u(y)|^p}{|x-y|^{N+ps}} dx\,dy -\lambda \int_{\Omega} \int_{0}^{u(x)} g(x, \tau) d \tau dx-\frac{\gamma}{p^{*}} \int_{\Omega}|u|^{p^{*}} d x
$$
for every $u \in \mathbb{X}(\Omega)$.
\end{cor}
\begin{proof}
    For the proof, we take $\mu:=\delta_1-\alpha \delta_s,$ where $\delta_1$ and $\delta_s$ are two Dirac measures centered at $1$ and $s \in [0, 1),$ respectively. Now, we choose $\bar{s}:=1$ and $s_\sharp:=1,$ so that the conditions  \eqref{measure 1} and \eqref{measure 2} hold. Moreover, note that 
    $$\mu^-([0, \bar{s}]) \leq \max \{0, \alpha\} = \max\{0, \alpha\} \mu^+([\bar{s}, 1]),$$
    which says that condition \eqref{measure 3} holds by choosing $\kappa:= \max\{0, \alpha\}.$ Hence, the proof follows as a direct application of Theorem \ref{main result wlsc 0}.
\end{proof}

\begin{cor}
    Let $ \Omega$ be a bounded subset of $\mathbb{R}^N$ and let  $N>p \geq 2$.  Let $\alpha \geq 0$ and $1>s_1>s_2>0.$ Suppose that $m$ and $q$ are two real constants such that 
$$1 \leq m<p\leq q<p^*=\frac{p N}{N-p}.$$
Then, there exists an open interval $\Lambda \subset (0, +\infty)$ and sufficiently small $\alpha_0:=\alpha_0(N, \Omega, p, s, \lambda, \kappa)$ such that, for every $\lambda \in \Lambda$ and $\alpha \leq \alpha_0,$ the  nonlocal problem 
\begin{equation} 
       \begin{cases} 
  -\Delta_{p} u+ (-\Delta_{p})^{s_1}-\alpha (-\Delta_{p})^{s_2} u =  |u|^{p^*-2}u+\lambda (|u|^{m-1}+|u|^{q-1}) \quad &\text{in}\quad \Omega, \\
u=0\quad & \text{in}\quad \mathbb{R}^N\backslash \Omega,
\end{cases}
    \end{equation} has at least one nonnegative nontrivial weak solution $u_\lambda \in  \mathbb{X}(\Omega).$ 
    
\end{cor}
\begin{proof} Let us take $\mu:= \delta_1+\delta_{s_1}-\alpha \delta_{s_2}$ for $1>s_1>s_2>0.$ Then, the conditions \eqref{measure 1} and \eqref{measure 2} hold with $\bar{s}:=s_1.$ We choose $s_\sharp:=1.$ Now, we will check if conditions \eqref{measure 3} and \eqref{extracondi} hold true for sufficiently small $\alpha.$ In fact, we note that 
$$\mu^-([0, s_1)) = \alpha = \frac{\alpha}{2} \mu^+([s_1, 1]),$$
which shows that condition \eqref{measure 3} is satisfied for any $\kappa \geq \frac{\alpha}{2}.$ Now, it remains to see that \eqref{extracondi} also holds true.
For this, we choose 
$$ \bar{\kappa}\in\left[0, \frac{\bar{s}\underline{\Gamma}_{N,p}}{4\bar{\Gamma}_{N,p}\max\{1, (2R)^p\}} \right),$$
where $\bar{\Gamma}_{N, p}$ and $\underline{\Gamma}_{N, p}$ is given by Lemma \ref{lm 2.4}. 
Also, by choosing $\alpha \in [0, \bar{\kappa}(1-s_1)],$ we can select $\delta:=\frac{\alpha}{\bar{\kappa}}$ and get that 
$$\mu^-([0, s_1])=\alpha= \alpha \mu^+ ([s_1, 1-\delta])=\bar{\kappa} \delta \mu^+([s_1, 1-\delta]).$$
This shows that condition \eqref{extracondi} is verified. Therefore, the proof of the results follows from Theorem \ref{thmnonmain}.
\end{proof}

\subsection{Nonlocal operator driven by continuous superposition of the fractional $p$-Laplacian} \label{subsec2.6}
We highlight another intriguing result arising from the continuous superposition of fractional operators of the $p$-Laplacian type. To the best of our knowledge, this result is also novel.

\begin{cor}
 Let $ \Omega$ be a bounded subset of $\mathbb{R}^N.$    Let $s_\sharp \in (0, 1), \kappa \geq 0,$  and let $f \not\equiv 0 $ be a measurable function such that 
    \begin{align} \label{condf}
        &f \geq 0 \quad \text{in} \quad (s_\sharp, 1), \nonumber \\
        &\int_{s_\sharp}^1 f(s)\, ds >0, \\
        \text{and}\quad &\int_0^{s_\sharp} \max\{0, -f(s)\}\, ds \leq \kappa \int_{s_\sharp}^1 f(s)\, ds. \nonumber
        \end{align}
    Suppose that $\frac{N}{s_{\sharp}}>p \geq 2$ and $q \in[1, p_{s_\sharp}^*)$, where $p_{s_\sharp}^*=\frac{p N}{N-{s_\sharp}p}$ is the fractional critical Sobolev exponent.
Let $g: \Omega \times \mathbb{R} \rightarrow \mathbb{R}$ be a Carath\'eodory function satisfying the subcritical growth condition 
\begin{align*}
\begin{split}
&\text{ there exist } a_{1}, a_{2}>0 \text{ and } q \in[1, p_{s_\sharp}^*), \text{ such that }\\
&|g(x, t)| \leq a_{1}+a_{2}|t|^{q-1} \text {, a.e. } x \in \Omega, \forall t \in \mathbb{R} . 
\end{split}
\end{align*} Then, for every $\gamma>0$, there exists $\Lambda_\gamma>0$ such that the problem 
\begin{equation}\label{p2.1}
 \begin{cases}
     \bigintss_0^1 f(s) (-\Delta_p)^s u\, ds  = \lambda g(x,u) +\gamma |u|^{p_{s_\sharp}^*-2} u \quad & \text { in } \Omega,  \\ \quad
    u  = 0  &\text { in } \mathbb{R}^{N},\backslash \Omega 
 \end{cases}
\end{equation}
admits at least one weak solution $u_\lambda \in \mathbb{X}(\Omega)$  for every $0<\lambda<\Lambda_\gamma.$
    
\end{cor}
\begin{proof}
    In this case, we define $d\mu(s):=f(s) ds$ with $f$ as in the statement of the result. Thus, the operator $A_{p, \mu}$ becomes
    $$ \int_0^1 f(s) (-\Delta_p)^s u\, ds.$$
    Thanks to the conditions presented in \eqref{condf}, all the conditions \eqref{measure 1}-\eqref{measure 3} are satisfied by taking $\bar{s}:=s_\sharp,$ which also plays a role of critical fractional Sobolev exponent.  Thus, the proof of this result is concluded from Theorem \ref{main result wlsc 0}.
\end{proof}

\begin{cor}
 Let $ \Omega$ be a bounded subset of $\mathbb{R}^N$ such that $\Omega \subset B_R$ for some $R>0.$     Let $s_\sharp \in (0, 1), \kappa \geq 0,$ and let $f \not\equiv 0 $ be a measurable function such that 
    \begin{align}
        &f \geq 0 \quad \text{in} \quad (s_\sharp, 1), \nonumber \\
        &\int_{s_\sharp}^1 f(s)\, ds >0, \\
        \text{and}\quad &\int_0^{s_\sharp} \max\{0, -f(s)\}\, ds \leq \kappa \int_{s_\sharp}^1 f(s)\, ds. \nonumber
        \end{align}
        Additionally, we assume that there exist $\bar{\kappa}:= \bar{\kappa}(N, R,  s_\sharp)>0$ and $\delta \in (0, 1-s_\sharp]$ such that 
        \begin{equation} \label{extrafcon}
            \int_0^{s_\sharp} \max\{0, -f(s)\}\, ds \leq \bar{\kappa} \delta \int_{s_\sharp}^{1-\delta} f(s)\, ds.
        \end{equation}
    Assume that $\frac{N}{s_{\sharp}}>p \geq 2$ and that $p_{s_\sharp}^* :=\frac{p N}{N-{s_\sharp}p}$ is the fractional critical Sobolev exponent.  Suppose that $m$ and $q$ are two real constants such that 
$$1 \leq m<p\leq q<p_{s_\sharp}^*.$$
Then, there exists an open interval $\Lambda \subset (0, +\infty)$  such that, for every $\lambda \in \Lambda$  the  nonlocal problem 
\begin{equation} 
       \begin{cases} 
   \bigintss_0^1 f(s) (-\Delta_p)^s u\, ds  =  |u|^{p_{s_\sharp}^*-2}u+\lambda (|u|^{m-1}+|u|^{q-1}) \quad &\text{in}\quad \Omega, \\
u=0\quad & \text{in}\quad \mathbb{R}^N\backslash \Omega,
\end{cases}
    \end{equation} has at least one nonnegative nontrivial weak solution $u_\lambda \in \mathbb{X}(\Omega).$
    
\end{cor}
\begin{proof}
   We define $d\mu(s):=f(s) ds$ with $f$ as in the statement of the result. Thus, the operator $A_{p, \mu}$ becomes
    $$ \int_0^1 f(s) (-\Delta_p)^s u\, ds.$$
    Thanks to the conditions presented in \eqref{condf} and condition \eqref{extrafcon}, all the conditions \eqref{measure 1}-\eqref{measure 3} and \eqref{extracondi} are satisfied by taking $\bar{s}:=s_\sharp,$ which also plays a role of critical fractional Sobolev exponent.  Thus, the proof of the corollary follows from Theorem \ref{thmnonmain}.
\end{proof}

\section*{Conflict of interest statement}
On behalf of all authors, the corresponding author states that there is no conflict of interest.

\section*{Data availability statement}
Data sharing is not applicable to this article as no datasets were generated or analysed during the current study.

\section*{Acknowledgement}
YA is supported by the Bolashak Government Scholarship of the Republic of Kazakhstan.  SG acknowledges the research facilities available at the Department of Mathematics, NIT Calicut under the DST-FIST support, Govt. of India [Project no. SR/FST/MS-1/2019/40 Dated. 07.01 2020.].  YA, VK, and MR are supported by the FWO Odysseus 1 grant G.0H94.18N: Analysis and Partial Differential Equations and the Methusalem program of the Ghent University Special Research Fund (BOF) (Grant number 01M01021). VK and MR are also supported by FWO Senior Research Grant G011522N.

\end{document}